\documentclass[reqno, eucal]{amsart}
\usepackage[a4paper,height=20cm,top=26mm,bottom=26mm,left=26mm,right=26mm]{geometry}

\usepackage[mathscr]{eucal} 
\usepackage{amsmath,amsfonts,amssymb,amsthm,amscd}
\usepackage{epic}
\usepackage{eepic}
\usepackage{longtable}
\usepackage{array}
\usepackage{here}

\usepackage{xypic}

\usepackage{here}
\usepackage{bm}
\usepackage[all]{xy}
\usepackage{tikz}
\usetikzlibrary{calc}

\newcommand\qarrow[2]{\draw[->,shorten >=2pt,shorten <=2pt] (#1) -- (#2) [thick];} 

\numberwithin{equation}{section}
\numberwithin{figure}{section}
\numberwithin{table}{section}

\newtheorem{theorem}{Theorem}[section]
\newtheorem{lemma}[theorem]{Lemma}

\newtheorem{proposition}[theorem]{Proposition}

\theoremstyle{definition}

\newtheorem{example}[theorem]{Example}
\newtheorem{remark}[theorem]{Remark}

\newcommand{\Z}{{\mathbb Z}}
\newcommand{\R}{{\mathbb R}}
\newcommand{\C}{{\mathbb C}}

\newcommand{\bb}{{\mathfrak b}}
\newcommand{\e}{{\mathrm e}}

\def\ve{{\varepsilon}}

\newcommand\trop{{\mathrm{trop}}}
\newcommand\rY{{\mathsf{Y}}}
\newcommand\Ad{{\mathrm{Ad}}}
\newcommand\bY{Y}

\begin{document}

\title[Tetrahedron equation and quantum cluster algebras]
{Tetrahedron equation and quantum cluster algebras}

\author[Rei Inoue]{Rei Inoue}
\address{Rei Inoue, Department of Mathematics and Informatics,
   Faculty of Science, Chiba University,
   Chiba, 263-8522, Japan}
\email{reiiy@math.s.chiba-u.ac.jp}

\author[Atsuo Kuniba]{Atsuo Kuniba}
\address{Atsuo Kuniba, Institute of Physics, Graduate School
of Arts and Sciences, University of Tokyo, Komaba, Tokyo, 153-8902, Japan}
\email{atsuo.s.kuniba@gmail.com}

\author[Yuji Terashima]{Yuji Terashima}
\address{Yuji Terashima, Graduate School of Science, Tohoku University,
6-3, Aoba, Aramaki-aza, Aoba-ku, Sendai, 980-8578, Japan}
\email{yujiterashima@tohoku.ac.jp}

\date{October 21, 2023, revised on January 11, 2024}

\begin{abstract}
We develop the quantum cluster algebra approach recently introduced by Sun and Yagi 
to investigate the tetrahedron equation, a three-dimensional generalization of the Yang-Baxter equation.
In the case of square quiver, 
we devise a new realization of quantum Y-variables in terms $q$-Weyl algebras and  
obtain a solution that possesses three spectral parameters. 
It is expressed in various forms, comprising four products of quantum dilogarithms 
depending on the signs in decomposing the quantum mutations 
into the automorphism part and the monomial part.
For a specific choice of them, our formula precisely reproduces 
Sergeev's $R$ matrix, which corresponds to a vertex formulation of the 
Zamolodchikov-Bazhanov-Baxter model when $q$ is specialized to a root of unity.
\end{abstract}

\maketitle

\section{Introduction}

The tetrahedron equation, originally proposed in \cite{Z80} as 
a key to integrability in three dimensions, 
stands as a remarkable generalization of the Yang-Baxter equation \cite{B82}.
A fundamental form of the equation in the so-called vertex formulation is given by
\begin{align}\label{TE}
R_{124}R_{135}R_{236}R_{456} = 
R_{456}R_{236}R_{135} R_{124},
\end{align}
where $R \in \mathrm{End}(V^{\otimes 3}$) for some vector space $V$, and the 
indices specify the tensor components of $V^{\otimes 6}$ on which it acts non-trivially.
Several interesting solutions have been obtained up to now. See for example 
\cite{Z81,B82,BB92, KMS93,KV94, SMS96, BS06, BMS08, BV15, KMY23}, ordered chronologically, 
and the references therein.

A quick way to see the origin of (\ref{TE}) is to postulate the Yang-Baxter equation {\em up to conjugation}:

\vspace*{-8cm}
\begin{figure}[H]
\hspace*{-1.1cm}
\includegraphics[clip,scale=0.6]{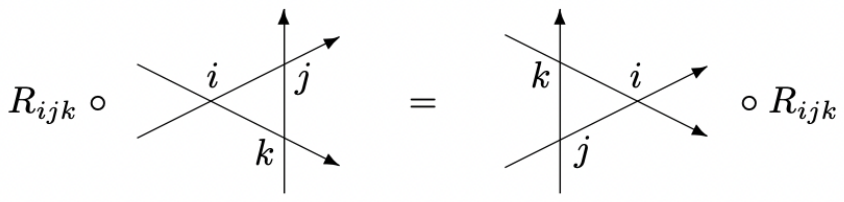}
\vspace*{-8.2cm}
\end{figure}

Here the conventional diagram representing the Yang-Baxter equation is regarded as an operator.  
As $R$ here plays the role of the structure constant, it has to satisfy the compatibility condition 
under introducing one more arrow\footnote{The convention here 
differs from the one which we will employ in Figure \ref{fig:rikisaku}.}:

\vspace*{-3.8cm}
\begin{figure}[H]
\hspace*{-1.1cm}
\includegraphics[clip,scale=0.6]{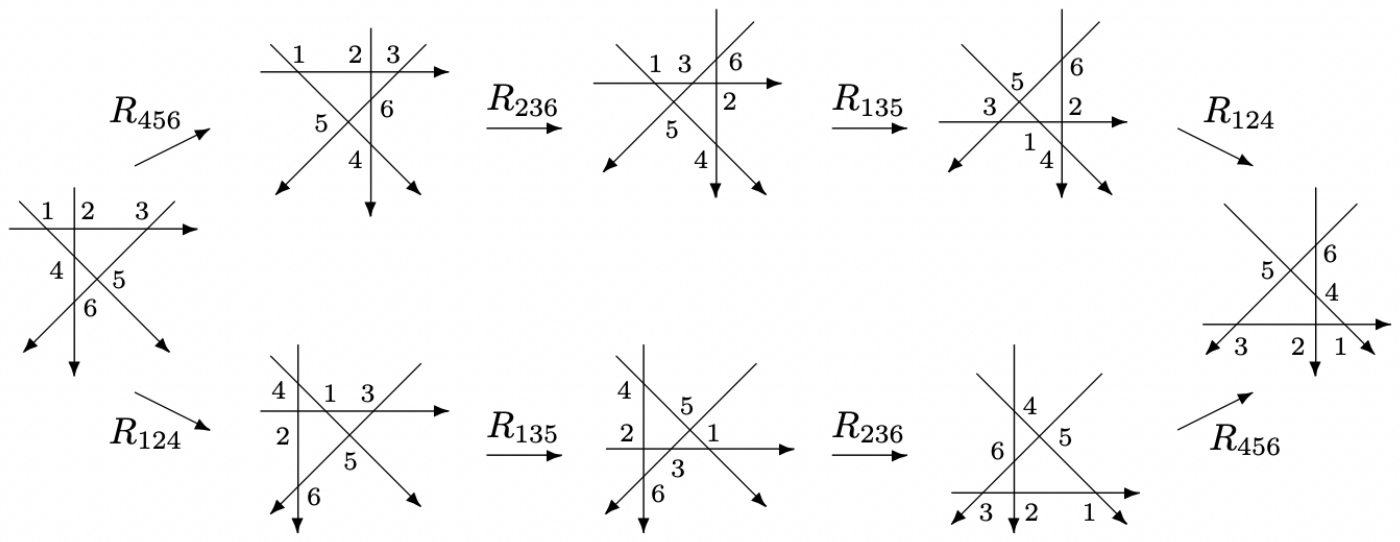}
\vspace*{-4.2cm}
\end{figure}

Systematic approaches to the tetrahedron equation have been made in quantum group theory, which 
utilizes the quantized coordinate rings as demonstrated in \cite{KV94}, 
or the PBW basis of $U_q^+$ as discussed in \cite{S08}, both specific to type A.
They are known to be equivalent beyond type A \cite{KOY13} 
and have been extensively developed, yielding numerous applications \cite{K22}.
In this context, the line configurations shown above are regarded as  {\em wiring diagrams} 
for the reduced expressions of the longest element in the Weyl group of $sl_4$.

The objective of this and the companion paper \cite{IKT23}  
is to develop another approach introduced in \cite{SY22},
where the wiring diagrams are accompanied with 
 {\em quivers} on which the {\em quantum cluster algebra} \cite{FG06} operates efficiently.
 
Cluster algebra has gained prominence as a ubiquitous structure in many branches 
of mathematics and theoretical physics (cf.~\cite{DGKY13}) in recent years.
A distinctive feature of quantum cluster algebras is the {\em mutation} of Y-variables, 
which can be decomposed into the monomial part and the automorphism (dilogarithm) part
in {\em two} ways (cf.~(\ref{mud})) depending on the choice of the {\em sign}. 
A key outcome of the theory is that 
the presence of a mutation loop within quivers inherently enforces the non-trivial identity for their compositions.

 We consider the {\em square quivers} \cite{SY22} associated to the wiring diagrams
 and devise a new realization of the Y-variables in terms of $q$-Weyl algebras.
 See Figure \ref{fig:para}  and  (\ref{Yw})--(\ref{Ypw}).
 It leads to a solution of the tetrahedron equation which is expressed as a product 
 of four quantum dilogarithms $\Psi_q$ and the monomial part.
While the solution remains unique within this framework, 
a variety of formulas can be derived depending on the selection of the four signs mentioned earlier.
For the  specific choice of $-,-,+,+$, it takes the form
\begin{align*}
R(\lambda_i,\lambda_j, \lambda_k)_{ijk} &= \Psi_q(\e^{w_k-w_j+\lambda_k})^{-1}
\Psi_q(\e^{u_k-w_i})^{-1}
\e^{\tfrac{1}{2\hbar}(u_i-w_i)(w_j-w_k)}
\\
&\qquad \times \rho_{jk}
\e^{\tfrac{\lambda_j-\lambda_k}{2\hbar}(u_k-w_i)}
 \Psi_q(\e^{w_i-u_k+\lambda_1-\lambda_k})
\Psi_q(\e^{w_j-w_k+\lambda_2}),
\end{align*}
where $\lambda_i$' are spectral parameters assigned to the crossings on the wiring diagram.
See (\ref{RL1}) and (\ref{pijk}) for the detail.
Our main finding is that it precisely reproduces the well-established R-matrix in \cite{S99}.
It is remarkable that the solution, which was originally constructed at the end of the last century,
has its origin in the theory of quantum cluster algebras.
In contrast to the result \cite[eq.(91)]{SY22}
where $R$ is considered as a transformation involving all the nine relevant Y-variables, 
our ``economical" realization using the $q$-Weyl algebra provides 
a suitable framework for the three-dimensional vertex model which turned out to fit  \cite{S99}.

The paper is organized as follows.
In Section \ref{s:qca}  we recall the basic ingredients from the quantum cluster algebras
following \cite{FG06}.
In Section \ref{s:ct}  we explain the approach by \cite{SY22} along our main target of the square quiver in this paper.
It leads to the cluster transformation $\widehat{R}$ satisfying the tetrahedron equation.
In Section \ref{s:qw} we present a novel realization of the Y-variables using $q$-Weyl algebras.
It enables us to express $\widehat{R}$ totally as an adjoint as
$\widehat{R}= \mathrm{Ad}(R(\lambda_1, \lambda_2, \lambda_3))$.
It is this $R(\lambda_1, \lambda_2, \lambda_3)$ that can be identified with \cite{S99}
and involves the spectral parameters $\lambda_1, \lambda_2$ and $\lambda_3$.
Extracting it out of $\widehat{R}$ is a fundamental step advancing \cite{SY22}, 
and requires realizing the monomial part also as an adjoint action via 
the Baker-Campbell-Hausdorff (BCH) formula as in (\ref{Pbch1}). 
Under some assumption (cf.~(\ref{XL})), it turns out to be possible, 
not always although, but specifically for the sign choices listed in Table \ref{tab:adt}.
In Section \ref{s:me} we present matrix elements of $R(\lambda_1, \lambda_2, \lambda_3)$ 
for an infinite dimensional representation and in the modular double setting.
Subsection \ref{ss:md}  is a review of a part of \cite[Chap.3]{S10}, to which 
we emphasize that the originality of the result belongs.
Section \ref{s:con} is the conclusion.
Appendix \ref{app:RY} contains a list of formulas for $\widehat{R}$ for general signs. 
Appendix \ref{ap:sup} is a supplement to Section \ref{ss:ms}.
Appendix \ref{app:ruw} presents the formulas for the monomial part 
$\tau^{uw}_{\ve_1,\ve_2, \ve_3, \ve_4}$ 
of the mutation with respect to the canonical variables and 
$R(\lambda_1, \lambda_2, \lambda_3)$
corresponding to the signs selected in Table \ref{tab:adt}.
Appendix \ref{ap:nc}  contains  integral formulas for the non-compact quantum dilogarithm used in Section \ref{s:me}.

In this paper we focus on the square quiver.
However, the approach in \cite{SY22} is applicable in principle to other quivers as well.
In fact, it can be generalized to cope with the three-dimensional reflection equation \cite{IK97, K22}.
Further results in such a direction along the 
so-called Fock-Goncharov quiver will be reported 
in our companion paper \cite{IKT23}\footnote{Some of the conventions used in \cite{IKT23} differs 
from the one employed in this paper.}. 

\section{Quantum cluster algebra}\label{s:qca}

\subsection{Mutation}\label{subsec:mutation}

We recall the definition of quantum cluster mutation by \cite{FG06}.
For a finite set $I$,  set
 $B = (b_{ij})_{i,j \in I}$ with $b_{ij} = -b_{ji} \in \Z$.
 We call $B$ the {\it exchange matrix}\footnote{ In this paper we will only encounter skew-symmetric exchange matrices.
 For a more general case of skew-symmetrizable case, see \cite{IKT23}.}.
An exchange matrix will be depicted in terms of a {\em quiver}.
It is an oriented 
graph with vertices labeled with $I$ and $b_{ij}$-fold arrows from $i$ to $j$ when $b_{ij}>0$.

Let $\mathcal{Y}(B)$ be a skew field generated by $q$-commuting variables 
$\bY = (Y_i)_{i \in I}$ with the relations 
\begin{align}\label{eq:q-Y}
  Y_i Y_j = q^{2b_{ij}} Y_j Y_i.  
\end{align}
We call the data $(B, \bY)$ a quantum $y$-seed. 
The parameter $q$ is assumed to be generic throughout.
For $(B, \bY)$ and $k \in I$, 
the mutation $\mu_k$ transforms $(B, \bY)$ to 
$(B', \bY') := \mu_k (B,\bY)$ as 
\begin{align}\label{eq:q-mutation}
  &b_{ij}' = 
  \begin{cases}
    -b_{ij} & i=k \text{ or } j=k,
    \\
    \displaystyle{b_{ij} + \frac{|b_{ik}| b_{kj} + b_{ik} |b_{kj}|}{2}}
    & \text{otherwise},
  \end{cases}
  \\ \label{eq:d-mutation}
  &Y_{i}' = 
  \begin{cases}
  Y_k^{-1} & i=k,
  \\
  \displaystyle{Y_i \prod_{j=1}^{|b_{ik}|}(1 + q^{2j-1} Y_k^{-\mathrm{sgn}(b_{ik})})^{-\mathrm{sgn}(b_{ik})}} & i \neq k.
  \end{cases}
\end{align}
The mutations are involutive, $\mu_k \mu_k = \mathrm{id.}$, 
and commutative, $\mu_k \mu_j = \mu_j \mu_k$ if $b_{jk}=b_{kj}=0$.
The mutation $\mu_k$ induces an isomorphism of skew fields 
$\mu_k^{\ast}: \mathcal{Y}(B') \to \mathcal{Y}(B)$, 
where $\mathcal{Y}(B')$ is a skew field generated by the variables $Y'=(Y'_i)_{i\in I}$ 
with the relations $Y'_i Y'_j = q^{2b'_{ij}} Y'_j Y'_i$. 

The map $\mu_k^{\ast}$ is decomposed into two parts,  
a monomial part and an automorphism part \cite{FG09}, in two ways \cite{Ke11}. 
To describe it we introduce an
isomorphism $\tau_{k,\varepsilon}$ of the skew fields for $\varepsilon \in \{+,-\}$ by 
\begin{align}\label{eq:mono-iso}
\tau_{k,\ve} : ~\mathcal{Y}(B')  \to \mathcal{Y}(B)
; \quad  Y'_i \mapsto 
\begin{cases} 
  Y_k^{-1} & i= k, 
  \\ 
  q^{-b_{ik}[\ve b_{ik}]_+}Y_i Y_k^{[\ve b_{ik}]_+} & i \neq k,
\end{cases}
\end{align} 
where $[a]_+ := \max[0,a]$.
The adjoint action $\mathrm{Ad}_{k,\ve}$ on $\mathcal{Y}(B)$ is defined by 
\begin{align}
&\mathrm{Ad}_{k,+} := \mathrm{Ad}(\Psi_{q}(Y_k)) , \quad 
\mathrm{Ad}_{k,-} := \mathrm{Ad}(\Psi_{q}(Y_k^{-1})^{-1})
\end{align}
with $\mathrm{Ad}(Y)(X) = YXY^{-1}$.
Here $\Psi_q(Y)$ denotes the quantum dilogarithm
\begin{align}\label{Psiq}
\Psi_q(Y) = \frac{1}{(-qY; q^2)_\infty}, \quad (z;q)_\infty = \prod_{n=0}^\infty (1-zq^n).
\end{align}
They have the expansions
\begin{align}\label{expa}
\Psi_q(Y) = \sum_{n = 0}^\infty \frac{(-qY)^n}{(q^2;q^2)_n},
\qquad 
\Psi_q(Y)^{-1} = \sum_{n = 0}^\infty \frac{q^{n^2} Y^n}{(q^2;q^2)_n},
\end{align}
where  $(z;q^2)_n = (z;q^2)_\infty/(zq^{2n};q^2)_\infty$ for any $n$. 
The fundamental properties of the quantum dilogarithm are
\begin{align}
\label{Prec}
&\Psi_q(q^2 U) \Psi_q(U)^{-1} = 1+qU,
\\
\label{penta}
&\Psi_q(U)\Psi_q(W) = \Psi_q(W) \Psi_q(q^{-1}UW) \Psi_q(U) ~~\text{ if } UW = q^2WU,
\end{align}
where the second one is called the pentagon identity.

Now the decomposition of $\mu^\ast_k$  in two ways mentioned in the above 
is given as
\begin{align}\label{mud}
\mu_k^{\ast} = \mathrm{Ad}_{k,+} \circ \tau_{k,+} 
= \mathrm{Ad}_{k,-} \circ \tau_{k,-}.
\end{align} 
Namely, the following diagram is commutative for both choices $\ve=+, -$:
\begin{align*}
\begin{picture}(100,65)(0,37)
\put(0,80){$\mathcal{Y}(B')$}
\put(37,90){ $\mu_k^{\ast}$} \put(30,83){\vector(1,0){35}} 
\put(70,80){$\mathcal{Y}(B)$}
\put(25,73){\vector(3,-2){38}}\put(33,51){$\tau_{k,\ve}$} 
\put(80,52){\vector(0,1){20}}\put(83,60){$\mathrm{Ad}_{k,\ve}$}
\put(70,40){$\mathcal{Y}(B)$} 
\end{picture} 
\end{align*}

\begin{example}
Let $I = \{1,2\}$ and the 2 by 2 exchange matrix be given by 
$B = \begin{pmatrix}0 & 1 \\ -1 & 0 \end{pmatrix}$ implying $Y_1Y_2 = q^2Y_2Y_1$.
Consider the mutation $\mu_2(B,\bY) = (B', \bY')$, where 
$\bY = (Y_1, Y_2)$ and $\bY' = (Y'_1, Y'_2)$. 
Then $B'=-B$ from (\ref{eq:q-mutation}) 
and $Y'_1 = Y_1(1+qY^{-1}_2)^{-1}$ from (\ref{eq:d-mutation}).
On the other hand, the same result 
is obtained also in the form $Y'_1 \rightarrow  Y_1(1+qY^{-1}_2)^{-1}$
in two ways according to (\ref{mud}) as follows:
\begin{align*}
Y'_1 & \overset{\tau_{2,+}}{\longrightarrow} q^{-1}Y_1Y_2 
\overset{\mathrm{Ad}_{2,+}}{\longrightarrow} q^{-1}\Psi_q(Y_2)Y_1Y_2\Psi_q(Y_2)^{-1}
= q^{-1}Y_1\Psi_q(q^{-2}Y_2)\Psi_q(Y_2)^{-1}Y_2 =  q^{-1}Y_1(1+q^{-1}Y_2)^{-1}Y_2,
\\
Y'_1 & \overset{\tau_{2,-}}{\longrightarrow} Y_1 
\overset{\mathrm{Ad}_{2,-}}{\longrightarrow} \Psi_q(Y_2^{-1})^{-1}Y_1 \Psi_q(Y_2^{-1})
= Y_1\Psi_q(q^2Y_2^{-1})^{-1}\Psi_q(Y_2^{-1}) =Y_1(1+qY^{-1}_2)^{-1}.
\end{align*}
\end{example}

We introduce the quantum torus algebra $\mathcal{T}(B)$ associated to $B$ for later use. 
This is the $\mathbb{Q}(q)$-algebra generated by non-commutative variables 
$\rY^\alpha~(\alpha \in \Z^{I})$ with the relations
\begin{align}\label{qyy}
q^{\langle \alpha,\beta \rangle} \rY^\alpha \rY^\beta 
= \rY^{\alpha + \beta},
\end{align}
where $\langle ~~,~~ \rangle$ is a skew-symmetric form given by 
$\langle \alpha,\beta \rangle = - \langle \beta,\alpha \rangle = - \alpha \cdot B \beta$.
Let $e_i$ be the standard unit vector of $\Z^I$, and write $\rY_i$ 
for $\rY^{e_i}$. Then we have $\rY_i \rY_j = q^{2 b_{ij}} \rY_j \rY_i$.
We identify $\rY_i$ with $Y_i$, which is consistent with \eqref{eq:q-Y}.

Let $\mathcal{FT}(B)$ be the fractional field of $\mathcal{T}(B)$. 
The mutations $\mu_k^{\ast}$ 
and their decompositions naturally 
induce the morphisms for the fractional fields of the quantum torus algebras. 
Especially, the monomial part \eqref{eq:mono-iso} of $\mu_k^{\ast}$ is written as
\begin{align}\label{eq:torus-iso} 
\tau_{k,\ve} : \mathcal{FT}(B') \to \mathcal{FT}(B); ~\; 
\rY'_i \mapsto \begin{cases} 
\rY_k^{-1} & i= k, 
\\ 
\rY^{e_i + e_k [\ve b_{ik}]_+} & i \neq k
\end{cases}
\end{align}
under the identification $\rY_i = Y_i, \rY'_i = Y'_i$.

\subsection{Tropical $y$-variables and tropical sign}

Let $\mathbb{P}(u)=\mathbb{P}_\trop(u_1,u_2,\ldots,u_p) := \{\prod_{i=1}^{p} u_i^{a_i}; ~a_i \in \Z \}$ 
be the tropical semifield of rank $p$, equipped with the addition $\oplus$ and multiplication $\cdot$ as
$$
  \prod_{i=1}^{p} u_i^{a_i} \oplus \prod_{i=1}^{p} u_i^{b_i}
  = 
  \prod_{i=1}^{p} u_i^{\min(a_i,b_i)}, 
  \qquad
  \prod_{i=1}^{p} u_i^{a_i} \cdot \prod_{i=1}^{p} u_i^{b_i}
  =  
  \prod_{i=1}^{p} u_i^{a_i+b_i}.
$$
For $s = \prod_{i \in I} u_i^{a_i} \in \mathbb{P}(u)$, we write $s = u^\alpha$
with $\alpha = (a_i)_{i \in I} \in \Z^{I}$.
We say $s$ is positive if  $\alpha \in (\Z_{\geq 0})^{I}$ and negative if 
$\alpha \in (\Z_{\leq 0})^{I}$.

For a quiver $Q$ with the vertex set $I$, let $\mathbb{P}(u)$ be a tropical semifield of rank $|I|$.
The data of the form $(B,y)$ where $B$ is the exchange matrix of $Q$ and 
$y = (y_i)_{i \in I} \in \mathbb{P}(u)^{I}$ is called a tropical $y$-seed.
For $k \in I$, the mutation\footnote{For simplicity we use the same symbol $\mu_k$ to denote a mutation for 
quantum $y$-seeds $(B,\bY)$ and tropical $y$-seeds $(B,y)$.} 
 $\mu_k (B, y) =: (B', y')$ is given by 
\eqref{eq:q-mutation} and 
\begin{equation}\label{eq:trop-mutation}
y_{i}' = 
  \begin{cases}
  y_k^{-1} & i=k,
  \\
  y_i (1 \oplus y_k^{-\mathrm{sgn}(b_{ik})})^{-b_{ik}} & i \neq k.
  \end{cases}
\end{equation}
For a tropical $y$-variable $y_i' = u^{\alpha'}$, the vector $\alpha' \in \Z^I$ 
is called the {\it $c$-vector} of $y_i'$.
The following theorem states the {\em sign coherence} of the $c$-vectors.

\begin{theorem}[\cite{FZ07,GHKK14}]
\label{thm:sign-coherence}
Let 
$(B',y') = \mu_{i_L} \cdots \mu_{i_2} \mu_{i_1}(B,u)$ be a tropical $y$-seed with 
$y'=(y'_i)_{i \in I}$.
For any sequence $(i_1,\ldots, i_L) \in I^L$, 
each $y'_i \in \mathbb{P}(u)$ is either positive or negative. 
\end{theorem}

Based on this theorem, for any tropical $y$-seed $(B',y')$ with $y'=(y'_i)_{i \in I}$ 
obtained from $(B,u)$ by applying mutations, we define the {\it tropical sign} $\ve_i'$ of $y_i'$ to be $+1$ (resp. $-1$) if $y_i'$ is positive (resp. $y_i$ is negative).  We also write $\ve_i'= \pm$ for $\ve_i'=\pm 1$ for simplicity.

\begin{remark}\label{re:sgni}
For the mutation $\mu_k (B,y) = (B',y')$ of a tropical $y$-seed, 
let $c_i, c_i', c_k$ be the $c$-vectors of $y_i, y_i', y_k$, 
and $\ve_k$ be the tropical sign of $y_k$. 
Then the tropical mutation \eqref{eq:trop-mutation} is written in terms of $c$-vectors as 
\begin{align}
  c_i' = 
  \begin{cases}
  -c_k & i=k,
  \\
  c_i + c_k [\ve_k b_{ik}]_+ & i \neq k.
\end{cases}
\end{align}
This coincides with the transformation of quantum torus \eqref{eq:torus-iso} 
on $\Z^I$ (i.e. the power of \eqref{eq:torus-iso}), when $\ve = \ve_k$. 
\end{remark}

\subsection{Sequence of mutations}

Let us describe the quantum Y-variables associated with the sequence  of mutations
 $\mu_{i_l} \mu_{i_{l-1}} \cdots \mu_{i_2} \mu_{i_1}$:
 \begin{align}\label{bys}
(B^{(1)},\bY^{(1)}) \stackrel{\mu_{i_1}}{\longleftrightarrow}
(B^{(2)},\bY^{(2)}) \stackrel{\mu_{i_2}}{\longleftrightarrow}
\cdots
\stackrel{\mu_{i_l}}{\longleftrightarrow} (B^{(l+1)},\bY^{(l+1)}).
\end{align}
For $t=1,\ldots, l+1$, let $\rY^\alpha(t) ~(\alpha \in \Z)$ 
be the generators of the quantum torus $\mathcal{T}(B^{(t)})$ in the sense explained around (\ref{qyy}).
We set $\rY_i(t) = \rY^{e_i}(t)$. Especially for $t=1$, 
we use the shorter notations $\rY^\alpha=\rY^\alpha(1)$ and $\rY_i=\rY_i(1)$.
As in (\ref{eq:torus-iso}), we identify $\rY_i$ with $Y_i = Y^{(1)}_i$ hence 
$\mathcal{Y}(B^{(1)})$ with $\mathcal{FT}(B^{(1)})$.
Then the quantum Y-variables $Y^{(t+1)}=(Y^{(t+1)}_i)_{i\in I}\, (t=0,\ldots, l)$ 
in (\ref{bys}) are expressed as 
\begin{align}\label{eq:ad-tau-decomp}
\begin{split}
Y_i^{(t+1)} 
&= \Ad(\Psi_{q}(\rY_{i_1}(1)^{\delta_1})^{\delta_1}) \tau_{i_1,\delta_1}\cdots \Ad(\Psi_{q}(\rY_{i_{t}}(t)^{\delta_{t}})^{\delta_{t}})  \tau_{i_{t},\delta_{t}}(\rY_i(t+1))
\\
&=\Ad(\Psi_{q}(\rY^{\delta_1 \beta_1})^{\delta_1} \cdots \Psi_{q}(\rY^{\delta_{t} \beta_{t}})^{\delta_{t}}) \tau_{i_1,\delta_1}\cdots \tau_{i_{t},\delta_{t}}(\rY_i(t+1)).
\end{split}
\end{align}
This holds true for any choice of the signs $\delta_1, \ldots, \delta_l \in \{+, -\}$, on which the LHS is independent.
Note that $Y^{(t+1)}_i$ is in general a ``complicated" element in $\mathcal{Y}(B^{(1)})$ 
generated from $(B^{(1)},Y^{(1)})$ by applying 
$\mu_{i_t}\cdots \mu_{i_2}\mu_{i_1}$ 
according to  (\ref{eq:d-mutation}), whereas 
$\rY_i(t+1)$ is just a basis of $\mathcal{T}(B^{(t+1)})$.
The first line of (\ref{eq:ad-tau-decomp}) tells that $Y^{(t+1)}_i$ is also obtained as the image of 
$\rY_i(t+1)$ under the composition
$\mu_{i_1}^\ast \cdots \mu^\ast_{i_{t-1}} \mu^\ast_{i_t}$ which is an isomorphism 
$\mathcal{FT}(B^{(t+1)}) \rightarrow \mathcal{FT}(B^{(1)}) = \mathcal{Y}(B^{(1)})$.
The second line is derived from the first line by pushing $\tau_{i,\delta}$'s to the right. 
Thus $\beta_1 = e_{i_1}$ and in general 
$\beta_r \in \Z^I$ is determined by 
$\rY^{\beta_r} = \tau_{i_1,\delta_1} \cdots \tau_{i_{r-1}, \delta_{r-1}}(\rY_{i_r}(r))$.

\subsection{A useful theorem}

We let transpositions $\sigma_{r,s} \in \mathfrak{S}_I\, (r,s \in I)$ act on
either classical $y$-seeds $(B,y)$ or quantum $y$-seeds $(B,Y)$ as the exchange of the indices $r$ and $s$.
For quantum $y$-seeds it is given by 
\begin{equation}\label{qys}
\bigl((b_{ij})_{i,j\in I} , (Y_i)_{i \in I} \bigl) \mapsto 
\bigl((b_{\sigma_{r,s}(i), \sigma_{r,s}(j)})_{i,j\in I}, (Y_{\sigma_{r,s}(i)})_{i \in I} \bigl),
\end{equation}
where $\sigma_{r,s}(r) = s, \sigma_{r,s}(s)=r$ 
and $\sigma_{r,s}(i) = i$ for $i \neq r,s$.
For classical $y$-seeds, the rule is similar.

Let 
\begin{equation}\label{nuL}
\nu = \nu_L\cdots \nu_1 := \sigma_{r_m,s_m} \cdots 
\mu_{i_l} \cdots \sigma_{r_1,s_1}  \cdots \mu_{i_1} \quad (L = l+m)
\end{equation}
be a composition of $l$ mutations $\mu_{i_1},\ldots, \mu_{i_l}$ 
and $m$ transpositions $\sigma_{r_1,s_1},\ldots, \sigma_{r_m, s_m}$ in an arbitrary order.
(So $\nu_L$ may actually be a mutation for example.)
For simplicity, we also call $\nu$ 
a mutation sequence even though a part of it consists of transpositions.

Consider the tropical $y$-seeds starting from $(B,y)$ 
and the quantum $y$-seeds starting from $(B,Y)$ which are generated along the mutation sequences
$\nu=\nu_L\cdots \nu_1$ and $\nu'= \nu'_{L'} \cdots  \nu'_1$ as follows:
\begin{align}
\label{eq:seq-tropy}
&(B,y) =:(B^{(1)},y^{(1)}) \stackrel{\nu_1}{\longleftrightarrow}
(B^{(2)},y^{(2)}) \stackrel{\nu_2}{\longleftrightarrow}
\cdots
\stackrel{\nu_L}{\longleftrightarrow} (B^{(L+1)},y^{(L+1)}) = \nu(B,y),
\\
\label{eq:seq-Y}
&(B,\bY) =:(B^{(1)},\bY^{(1)}) \stackrel{\nu_1}{\longleftrightarrow}
(B^{(2)},\bY^{(2)}) \stackrel{\nu_2}{\longleftrightarrow}
\cdots
\stackrel{\nu_L}{\longleftrightarrow} (B^{(L+1)},\bY^{(L+1)}) = \nu(B,Y),
\\
\label{eq:seq-tropy2}
&(B,y) =:(B^{(1)\prime },y^{(1)\prime}) \stackrel{\nu'_1}{\longleftrightarrow}
(B^{(2)\prime},y^{(2)\prime}) \stackrel{\nu'_2}{\longleftrightarrow}
\cdots
\stackrel{\nu'_{L'}}{\longleftrightarrow} (B^{(L+1)\prime},y^{(L+1)\prime}) = \nu'(B,y),
\\
\label{eq:seq-Y2}
&(B,\bY) =:(B^{(1)\prime},\bY^{(1)\prime}) \stackrel{\nu'_1}{\longleftrightarrow}
(B^{(2)\prime},\bY^{(2)\prime}) \stackrel{\nu'_2}{\longleftrightarrow}
\cdots
\stackrel{\nu'_{L'}}{\longleftrightarrow} (B^{(L+1)\prime},\bY^{(L+1)\prime}) = \nu'(B,Y).
\displaybreak[0]
\end{align}
The following theorem is obtained by combining the synchronicity \cite{N21} 
among $x$-seeds, $y$-seeds and tropical $y$-seeds, 
and the synchronicity between classical and quantum seeds 
\cite[Lemma 2.22]{FG09b}, \cite[Proposition 3.4]{KN11}.
 
\begin{theorem}\label{thm:piv}
In the situation in (\ref{eq:seq-tropy}) -- (\ref{eq:seq-Y2}),
the following two statements are equivalent:
\begin{itemize}

\item[(1)]
Tropical $y$-seeds satisfy $\nu (B,y) = \nu'(B,y)$. 

\item[(2)]
Quantum $y$-seeds satisfy $\nu (B,\bY) = \nu'(B,\bY)$. 

\end{itemize}
\end{theorem}
\noindent
It is noteworthy that (1) implies (2), and this fact will be utilized 
in the subsequent arguments.

\section{Cluster transformation $\widehat{R}$}\label{s:ct}

\subsection{Wiring diagram and square quiver}

Let us fix our convention of the wiring diagrams and associated square quivers  
using examples.
See also \cite[Sec.3]{SY22}.
Let $W(A_n)$ be the Weyl group of $A_n$ generated by the simple reflections
$s_1,\ldots, s_n$ obeying the Coxeter relations $s_i^2=1$, 
$s_is_js_i = s_js_is_j \,(|i-j|=1)$ and $s_is_j=s_js_i\, (|i-j|\ge 2)$.
A reduced expression $s_{i_1}\cdots s_{i_l}$ of an element in $W(A_n)$ 
is identified with the (reduced) word $i_1\ldots i_l \in [1,n]^l$.
A wiring diagram is a collection of $n$ wires which are horizontal except the vicinity of crossings.
In the aforementioned context, $i_k$ indicates that the $k$-th crossing from the left takes place at the 
$i_k$-th level, measured from the top.
Crossings are required to occur at distinct horizontal positions, 
although this restriction can be relaxed due to the identification of topologically 
equivalent diagrams which are transformable by $s_is_j=s_js_i, (|i-j|\ge 2)$.
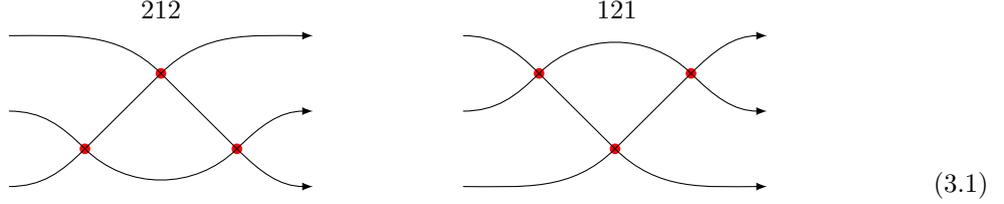
\begin{figure}[H]
\begin{align}
\begin{tikzpicture}
\begin{scope}[>=latex,xshift=0pt]
\draw (2,2.1) circle(0pt) node[above]{212};
\draw (8,2.1) circle(0pt) node[above]{121};
{\color{red}
\fill (1,0.5) circle(2pt) coordinate(A) node[below]{};
\fill (2,1.5) circle(2pt) coordinate(B) node[above]{};
\fill (3,0.5) circle(2pt) coordinate(C) node[below]{};
}
\draw [-] (0,2) to [out = 0, in = 135] (B);
\draw [-] (B) -- (C); 
\draw [->] (C) to [out = -45, in = 180] (4,0);
\draw [-] (0,1) to [out = 0, in = 135] (A); 
\draw [-] (A) to [out = -45, in = -135] (C);
\draw [->] (C) to [out = 45, in = 180] (4,1);
\draw [-] (0,0) to [out = 0, in = -135] (A); 
\draw [-] (A) -- (B);
\draw [->] (B) to [out = 45, in = 180] (4,2);
\coordinate (P1) at (4.5,1);
\coordinate (P2) at (5.5,1);
\end{scope}
\begin{scope}[>=latex,xshift=170pt]
{\color{red}
\fill (3,1.5) circle(2pt) coordinate(A) node[above]{};
\fill (2,0.5) circle(2pt) coordinate(B) node[below]{};
\fill (1,1.5) circle(2pt) coordinate(C) node[above]{};
}
\draw [-] (0,0) to [out = 0, in = -135] (B);
\draw [-] (B) -- (A); 
\draw [->] (A) to [out = 45, in = 180] (4,2);
\draw [-] (0,1) to [out = 0, in = -135] (C); 
\draw [-] (C) to [out = 45, in = 135] (A);
\draw [->] (A) to [out = -45, in = 180] (4,1);
\draw [-] (0,2) to [out = 0, in = 135] (C); 
\draw [-] (C) -- (B);
\draw [->] (B) to [out = -45, in = 180] (4,0);
\end{scope}
\end{tikzpicture}
\end{align}
\caption{Wiring diagrams for 
the reduced words 212 and 121 of the longest element 
$s_2s_1s_2=s_1s_2s_1$ of  $W(A_2)$.}
\end{figure}
\vspace{-0.2cm}
Given a wiring diagram, the associated square quiver 
is constructed by placing vertices on the edges and 
drawing arrows  in such a manner that each crossing in the wiring diagram is enclosed by a square, 
and the arrows on every square are oriented in a clockwise fashion.
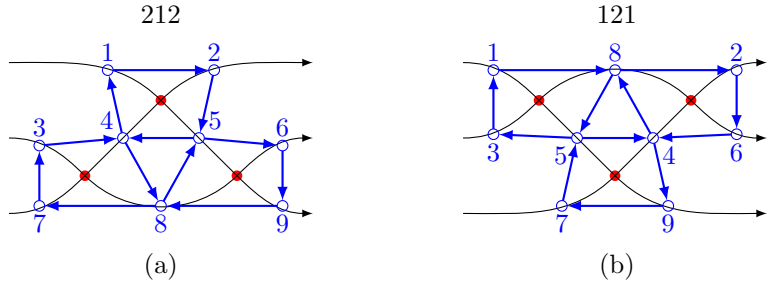
\begin{figure}[H]
\begin{align*}
\begin{tikzpicture}
\begin{scope}[>=latex,xshift=0pt]
\draw (2,2.4) circle(0pt) node[above]{212};
\draw (8,2.4) circle(0pt) node[above]{121};
\draw (2,-1) circle(0pt) node[above]{(a)};
\draw (8,-1) circle(0pt) node[above]{(b)};
{\color{red}
\fill (1,0.5) circle(2pt) coordinate(A) node[below]{};
\fill (2,1.5) circle(2pt) coordinate(B) node[above]{};
\fill (3,0.5) circle(2pt) coordinate(C) node[below]{};
}
\draw [-] (0,2) to [out = 0, in = 135] (B);
\draw [-] (B) -- (C); 
\draw [->] (C) to [out = -45, in = 180] (4,0);
\draw [-] (0,1) to [out = 0, in = 135] (A); 
\draw [-] (A) to [out = -45, in = -135] (C);
\draw [->] (C) to [out = 45, in = 180] (4,1);
\draw [-] (0,0) to [out = 0, in = -135] (A); 
\draw [-] (A) -- (B);
\draw [->] (B) to [out = 45, in = 180] (4,2);
\coordinate (P1) at (4.5,1);
\coordinate (P2) at (5.5,1);
%
%
{\color{blue}
\draw (1.3,1.9) circle(2pt) coordinate(1) node[above]{$1$};
\draw (2.7,1.9) circle(2pt) coordinate(2) node[above]{$2$};
\draw (0.4,0.9) circle(2pt) coordinate(3) node[above]{$3$};
\draw (1.5,1) circle(2pt) coordinate(4) node[above left]{$4$};
\draw (2.5,1) circle(2pt) coordinate(5) node[above right]{$5$};
\draw (3.6,0.9) circle(2pt) coordinate(6) node[above]{$6$};
\draw (0.4,0.1) circle(2pt) coordinate(7) node[below]{$7$};
\draw (2,0.1) circle(2pt) coordinate(8) node[below]{$8$};
\draw (3.6,0.1) circle(2pt) coordinate(9) node[below]{$9$};
\qarrow{1}{2}
\qarrow{2}{5}
\qarrow{5}{4}
\qarrow{4}{1}
\qarrow{4}{8}
\qarrow{8}{7}
\qarrow{7}{3}
\qarrow{3}{4}
\qarrow{5}{6}
\qarrow{6}{9}
\qarrow{9}{8}
\qarrow{8}{5}
}
\end{scope}
\begin{scope}[>=latex,xshift=170pt]
{\color{red}
\fill (3,1.5) circle(2pt) coordinate(A) node[above]{};
\fill (2,0.5) circle(2pt) coordinate(B) node[below]{};
\fill (1,1.5) circle(2pt) coordinate(C) node[above]{};
}
\draw [-] (0,0) to [out = 0, in = -135] (B);
\draw [-] (B) -- (A); 
\draw [->] (A) to [out = 45, in = 180] (4,2);
\draw [-] (0,1) to [out = 0, in = -135] (C); 
\draw [-] (C) to [out = 45, in = 135] (A);
\draw [->] (A) to [out = -45, in = 180] (4,1);
\draw [-] (0,2) to [out = 0, in = 135] (C); 
\draw [-] (C) -- (B);
\draw [->] (B) to [out = -45, in = 180] (4,0);
{\color{blue}
\draw (1.3,0.1) circle(2pt) coordinate(7) node[below]{$7$};
\draw (2.7,0.1) circle(2pt) coordinate(9) node[below]{$9$};
\draw (0.4,1.05) circle(2pt) coordinate(3) node[below]{$3$};
\draw (1.5,1) circle(2pt) coordinate(5) node[below left]{$5$};
\draw (2.5,1) circle(2pt) coordinate(4) node[below right]{$4$};
\draw (3.6,1.05) circle(2pt) coordinate(6) node[below]{$6$};
\draw (0.4,1.9) circle(2pt) coordinate(1) node[above]{$1$};
\draw (2,1.9) circle(2pt) coordinate(8) node[above]{$8$};
\draw (3.6,1.9) circle(2pt) coordinate(2) node[above]{$2$};
\qarrow{1}{8}
\qarrow{8}{5}
\qarrow{5}{3}
\qarrow{3}{1}
\qarrow{4}{8}
\qarrow{8}{2}
\qarrow{2}{6}
\qarrow{6}{4}
\qarrow{5}{4}
\qarrow{4}{9}
\qarrow{9}{7}
\qarrow{7}{5}
}
\end{scope}
\end{tikzpicture}
\end{align*}
\caption{Square quivers (depicted in blue) associated with the wiring diagrams.
Given the labels $1,\ldots, 9$ of the quiver vertices in (a), those in (b) are 
determined following the mutation sequence in Figure \ref{fig:mus}.}
\label{fig:quiver121}
\end{figure}

Let $(B, \bY)$ and $(B', \bY')$ be the quantum $y$-seeds corresponding to the 
square quivers for the reduced words 212 and 121 in Figure \ref{fig:quiver121}, respectively.
From (\ref{eq:q-Y}) the Y-variables 
$\bY=(Y_1,\ldots, Y_9)$ and $\bY'=(Y'_1,\ldots, Y'_9)$ are commutative 
except the following:
\begin{alignat}{4}
Y_1Y_2 &= q^2Y_2Y_1, \quad & Y_2Y_5 &= q^2 Y_5Y_2, \quad & 
Y_5 Y_4 &= q^2 Y_4Y_5,  \quad & Y_4Y_1 &= q^2 Y_1Y_4,
\nonumber\\
Y_3Y_4 &= q^2Y_4Y_3,\quad & Y_4Y_8 &= q^2Y_8Y_4, \quad &
Y_8Y_7 &= q^2Y_7Y_8, \quad & Y_7Y_3 &= q^2 Y_3Y_7,
\nonumber\\
Y_5Y_6 &= q^2Y_6Y_5,\quad & Y_6Y_9 &= q^2Y_9Y_6, \quad &
Y_9Y_8 &= q^2Y_8Y_9, \quad & Y_8Y_5 &= q^2 Y_5Y_8,
\nonumber\\
Y'_1Y'_8 & = q^2 Y'_8Y'_1, \quad & Y'_8Y'_5 &= q^2Y'_5Y'_8, \quad &
Y'_5Y'_3 &= q^2Y'_3Y'_5, \quad &  Y'_3Y'_1 &= q^2Y'_1Y'_3, 
\nonumber\\
Y'_8Y'_2 & = q^2 Y'_2Y'_8, \quad & Y'_2Y'_6 &= q^2Y'_6Y'_2, \quad &
Y'_6Y'_4 &= q^2Y'_4Y'_6, \quad &  Y'_4Y'_8 &= q^2Y'_8Y'_4, 
\nonumber\\
Y'_5Y'_4 &= q^2Y'_4Y'_5, \quad & Y'_4Y'_9 &= q^2Y'_9Y'_4, \quad &
Y'_9Y'_7 &= q^2Y'_7Y'_9, \quad & Y'_7Y'_5 &= q^2Y'_5Y'_7.
\nonumber
\displaybreak[0]
\end{alignat}

\begin{remark}\label{re:c1}
The center of $\mathcal{Y}(B)$ (resp.~$\mathcal{Y}(B')$) is generated by 
$Y_9Y_5Y_1,\; Y_3Y_8Y_6$ and $Y_2Y_4Y_7$
(resp. $Y'_9Y'_5Y'_1,\; Y'_3Y'_8Y'_6$ and $Y'_2Y'_4Y'_7$), which are 
products of Y-variables along the wires.
\end{remark}

\subsection{Cluster transformation $\widehat{R}$ as composition of mutations}

Let $(B^{(1)}, \bY^{(1)}) = (B, \bY)$ and 
$(B^{(6)}, \bY^{(6)}) = (B', \bY')$  be the quantum $y$-seeds corresponding to Figure \ref{fig:quiver121} (a) and (b), respectively.
We connect them by the following mutation sequence
\begin{align}\label{mus}
(B^{(1)}, \bY^{(1)}) \underset{\varepsilon_1}{\overset{\mu_8}{\longleftrightarrow}}
(B^{(2)}, \bY^{(2)}) \underset{\varepsilon_2}{\overset{\mu_5}{\longleftrightarrow}}
(B^{(3)}, \bY^{(3)}) \underset{\varepsilon_3}{\overset{\mu_4}{\longleftrightarrow}}
(B^{(4)}, \bY^{(4)}) \underset{\varepsilon_4}{\overset{\mu_8}{\longleftrightarrow}}
(B^{(5)}, \bY^{(5)}) \overset{\sigma_{45}}{\longleftrightarrow}
(B^{(6)}, \bY^{(6)}),
\end{align}
where $\bY^{(t)}=(Y^{(t)}_1,\ldots, Y^{(t)}_9)$.
The symbol $\sigma_{ij}$ denotes the exchange of the indices $i$ and $j$ in the 
exchange matrix and Y-variables. See (\ref{qys}).
Thus we have $\sigma_{45}(Y^{(6)}_5) = Y^{(5)}_4, \, 
\sigma_{45}(Y^{(6)}_4) = Y^{(5)}_5$ and $\sigma_{45}(Y^{(6)}_i) = Y^{(5)}_i$ for 
$i \neq 4,5$.
As for the exchange matrices, $\sigma_{45}$ is represented as the exchange of labels 4 and 5 in the quivers corresponding to $B^{(5)}$ and $B^{(6)}$. 
We have also attached the signs $\varepsilon_i=\pm 1$ along which the decomposition (\ref{mud})
into the automorphism part and the monomial part will be considered.  
See Figure \ref{fig:mus}.

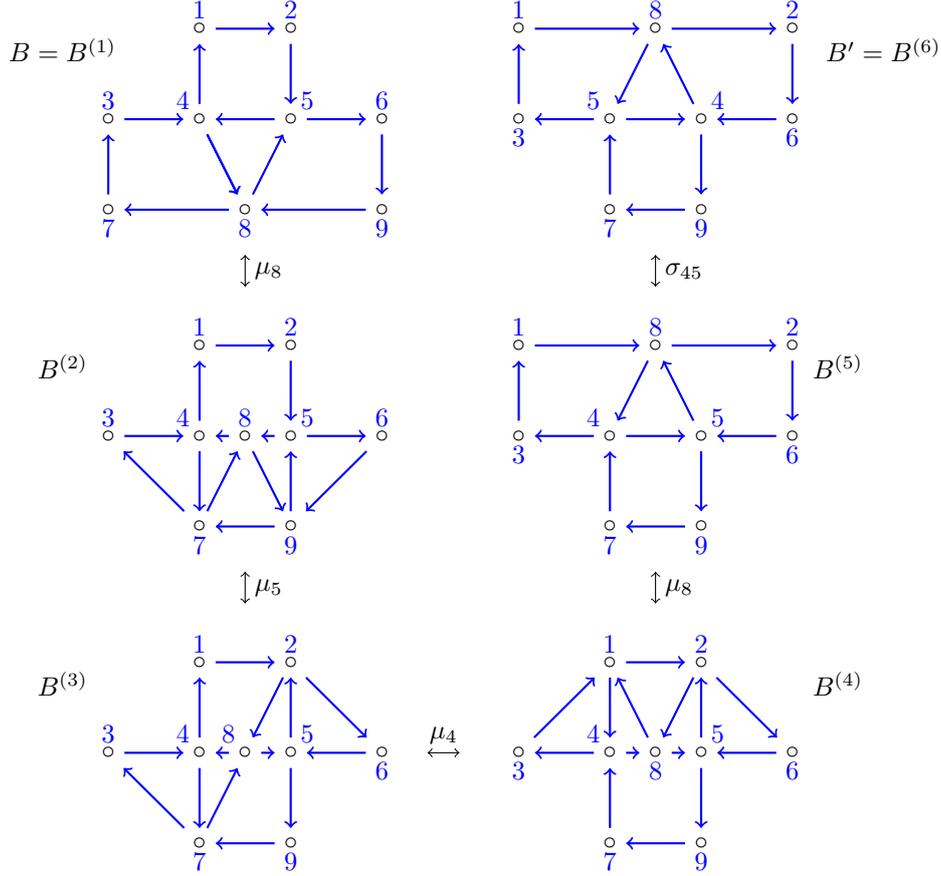
\begin{figure}[H]
\begin{align*}
\begin{tikzpicture}[scale=0.6]
\node (q1[1]) at (2,2+7) {$\circ$};
\node (q2[1]) at (4,2+7) {$\circ$};
\node (q3[1]) at (0,0+7) {$\circ$};
\node (q4[1]) at (2,0+7) {$\circ$};
\node (q5[1]) at (4,0+7) {$\circ$};
\node (q6[1]) at (6,0+7) {$\circ$};
\node (q7[1]) at (0,-2+7) {$\circ$};
\node (q8[1]) at (3,-2+7) {$\circ$};
\node (q9[1]) at (6,-2+7) {$\circ$};
{\color{blue}
\draw [->] (node cs:name=q1[1]) -- (node cs:name=q2[1])  [thick]  node[above] {$2$};
\draw [->] (node cs:name=q2[1]) -- (node cs:name=q5[1])  [thick] node[above right] {$5$};
\draw [->] (node cs:name=q5[1]) -- (node cs:name=q6[1])  [thick] node[above] {$6$};
\draw [->] (node cs:name=q6[1]) -- (node cs:name=q9[1]) [thick] node[below] {$9$};
\draw [->] (node cs:name=q8[1]) -- (node cs:name=q5[1]) [thick] ;
\draw [->] (node cs:name=q3[1]) -- (node cs:name=q4[1]) [thick] ;
\draw [->] (node cs:name=q5[1]) -- (node cs:name=q4[1]) [thick]  node[above  left] {$4$};
\draw [->] (node cs:name=q4[1]) -- (node cs:name=q8[1]) [thick]  node[below] {$8$};
\draw [->] (node cs:name=q4[1]) -- (node cs:name=q8[1]) [thick] ;
\draw [->] (node cs:name=q8[1]) -- (node cs:name=q7[1]) [thick]  node[below] {$7$};
\draw [->] (node cs:name=q9[1]) -- (node cs:name=q8[1]) [thick] ;
\draw [->] (node cs:name=q7[1]) -- (node cs:name=q3[1]) [thick] node[above] {$3$};
\draw [->] (node cs:name=q4[1]) -- (node cs:name=q1[1]) [thick] node[above] {$1$};
}
\node (q1[2]) at (2,2) {$\circ$};
\node (q2[2]) at (4,2) {$\circ$};
\node (q3[2]) at (0,0) {$\circ$};
\node (q4[2]) at (2,0) {$\circ$};
\node (q5[2]) at (4,0) {$\circ$};
\node (q6[2]) at (6,0) {$\circ$};
\node (q7[2]) at (2,-2) {$\circ$};
\node (q8[2]) at (3,0) {$\circ$};
\node (q9[2]) at (4,-2) {$\circ$};
{\color{blue}
\draw [->] (node cs:name=q1[2]) -- (node cs:name=q2[2]) [thick] node[above] {$2$};
\draw [->] (node cs:name=q2[2]) -- (node cs:name=q5[2]) [thick] node[above right] {$5$};
\draw [->] (node cs:name=q5[2]) -- (node cs:name=q6[2]) [thick] node[above] {$6$};
\draw [->] (node cs:name=q6[2]) -- (node cs:name=q9[2]) [thick] node[below] {$9$};
\draw [->] (node cs:name=q9[2]) -- (node cs:name=q5[2]) [thick];
\draw [->] (node cs:name=q5[2]) -- (node cs:name=q8[2]) [thick] node[above] {$8$};
\draw [->] (node cs:name=q8[2]) -- (node cs:name=q9[2]) [thick];
\draw [->] (node cs:name=q9[2]) -- (node cs:name=q7[2]) [thick] node[below] {$7$};
\draw [->] (node cs:name=q7[2]) -- (node cs:name=q8[2]) [thick];
\draw [->] (node cs:name=q8[2]) -- (node cs:name=q4[2])  [thick] node[above left] {$4$};
\draw [->] (node cs:name=q7[2]) -- (node cs:name=q3[2]) [thick] node[above] {$3$};
\draw [->] (node cs:name=q3[2]) -- (node cs:name=q4[2]) [thick];
\draw [->] (node cs:name=q4[2]) -- (node cs:name=q1[2])  [thick]node[above] {$1$};
\draw [->] (node cs:name=q4[2]) -- (node cs:name=q7[2])[thick];
}
\node (q1[3]) at (2,2-7) {$\circ$};
\node (q2[3]) at (4,2-7) {$\circ$};
\node (q3[3]) at (0,0-7) {$\circ$};
\node (q4[3]) at (2,0-7) {$\circ$};
\node (q5[3]) at (4,0-7) {$\circ$};
\node (q6[3]) at (6,0-7) {$\circ$};
\node (q7[3]) at (2,-2-7) {$\circ$};
\node (q8[3]) at (3,0-7) {$\circ$};
\node (q9[3]) at (4,-2-7) {$\circ$};
{\color{blue}
\draw [->] (node cs:name=q1[3]) -- (node cs:name=q2[3])[thick] node[above] {$2$};
\draw [->] (node cs:name=q5[3]) -- (node cs:name=q2[3])[thick];
\draw [->] (node cs:name=q6[3]) -- (node cs:name=q5[3])[thick];
\draw [->] (node cs:name=q2[3]) -- (node cs:name=q6[3])[thick]
node[below] {$6$};
\draw [->] (node cs:name=q5[3]) -- (node cs:name=q9[3])[thick]
node[below] {$9$};
\draw [->] (node cs:name=q8[3]) -- (node cs:name=q5[3]) [thick]node[above right] {$5$};
\draw [->] (node cs:name=q2[3]) -- (node cs:name=q8[3])[thick]
node[above left] {$8$};
\draw [->] (node cs:name=q9[3]) -- (node cs:name=q7[3]) [thick]node[below] {$7$};
\draw [->] (node cs:name=q7[3]) -- (node cs:name=q8[3])[thick];
\draw [->] (node cs:name=q8[3]) -- (node cs:name=q4[3]) [thick]node[above left] {$4$};
\draw [->] (node cs:name=q7[3]) -- (node cs:name=q3[3])[thick] node[above] {$3$};
\draw [->] (node cs:name=q3[3]) -- (node cs:name=q4[3])[thick];
\draw [->] (node cs:name=q4[3]) -- (node cs:name=q1[3]) [thick]node[above] {$1$};
\draw [->] (node cs:name=q4[3]) -- (node cs:name=q7[3])[thick];
}
\node (q1[4]) at (2+9,2-7) {$\circ$};
\node (q2[4]) at (4+9,2-7) {$\circ$};
\node (q3[4]) at (0+9,0-7) {$\circ$};
\node (q4[4]) at (2+9,0-7) {$\circ$};
\node (q5[4]) at (4+9,0-7) {$\circ$};
\node (q6[4]) at (6+9,0-7) {$\circ$};
\node (q7[4]) at (2+9,-2-7) {$\circ$};
\node (q8[4]) at (3+9,0-7) {$\circ$};
\node (q9[4]) at (4+9,-2-7) {$\circ$};
{\color{blue}
\draw [->] (node cs:name=q1[4]) -- (node cs:name=q2[4]) [thick]node[above] {$2$};
\draw [->] (node cs:name=q5[4]) -- (node cs:name=q2[4])[thick];
\draw [->] (node cs:name=q6[4]) -- (node cs:name=q5[4])[thick];
\draw [->] (node cs:name=q2[4]) -- (node cs:name=q6[4])[thick]
node[below] {$6$};
\draw [->] (node cs:name=q5[4]) -- (node cs:name=q9[4])[thick]
node[below] {$9$};
\draw [->] (node cs:name=q8[4]) -- (node cs:name=q5[4]) [thick]node[above right] {$5$};
\draw [->] (node cs:name=q2[4]) -- (node cs:name=q8[4])[thick]
node[below] {$8$};
\draw [->] (node cs:name=q9[4]) -- (node cs:name=q7[4]) [thick]node[below] {$7$};
\draw [->] (node cs:name=q8[4]) -- (node cs:name=q1[4])[thick];
\draw [->] (node cs:name=q4[4]) -- (node cs:name=q8[4])[thick];
\draw [->] (node cs:name=q3[4]) -- (node cs:name=q1[4])[thick]
node[above] {$1$};
\draw [->] (node cs:name=q4[4]) -- (node cs:name=q3[4])[thick]
node[below] {$3$};
\draw [->] (node cs:name=q1[4]) -- (node cs:name=q4[4])[thick]
node[above left] {$4$};
\draw [->] (node cs:name=q7[4]) -- (node cs:name=q4[4])[thick];
}
\node (q1[5]) at (0+9,2) {$\circ$};
\node (q2[5]) at (6+9,2) {$\circ$};
\node (q3[5]) at (0+9,0) {$\circ$};
\node (q4[5]) at (2+9,0) {$\circ$};
\node (q5[5]) at (4+9,0) {$\circ$};
\node (q6[5]) at (6+9,0) {$\circ$};
\node (q7[5]) at (2+9,-2) {$\circ$};
\node (q8[5]) at (3+9,2) {$\circ$};
\node (q9[5]) at (4+9,-2) {$\circ$};
{\color{blue}
\draw [->] (node cs:name=q4[5]) -- (node cs:name=q5[5])[thick];
\draw [->] (node cs:name=q6[5]) -- (node cs:name=q5[5])[thick]
node[above right] {$5$};
\draw [->] (node cs:name=q2[5]) -- (node cs:name=q6[5])[thick]
node[below] {$6$};
\draw [->] (node cs:name=q5[5]) -- (node cs:name=q9[5])[thick]
node[below] {$9$};
\draw [->] (node cs:name=q5[5]) -- (node cs:name=q8[5]) [thick]node[above] {$8$};
\draw [->] (node cs:name=q8[5]) -- (node cs:name=q2[5])[thick]
node[above] {$2$};
\draw [->] (node cs:name=q9[5]) -- (node cs:name=q7[5])[thick] node[below] {$7$};
\draw [->] (node cs:name=q1[5]) -- (node cs:name=q8[5])[thick];
\draw [->] (node cs:name=q8[5]) -- (node cs:name=q4[5])[thick]
node[above left] {$4$};
\draw [->] (node cs:name=q3[5]) -- (node cs:name=q1[5])[thick]
node[above] {$1$};
\draw [->] (node cs:name=q4[5]) -- (node cs:name=q3[5])[thick]
node[below] {$3$};
\draw [->] (node cs:name=q7[5]) -- (node cs:name=q4[5])[thick];
}
\node (q1[6]) at (0+9,2+7) {$\circ$};
\node (q2[6]) at (6+9,2+7) {$\circ$};
\node (q3[6]) at (0+9,0+7) {$\circ$};
\node (q4[6]) at (2+9,0+7) {$\circ$};
\node (q5[6]) at (4+9,0+7) {$\circ$};
\node (q6[6]) at (6+9,0+7) {$\circ$};
\node (q7[6]) at (2+9,-2+7) {$\circ$};
\node (q8[6]) at (3+9,2+7) {$\circ$};
\node (q9[6]) at (4+9,-2+7) {$\circ$};
{\color{blue}
\draw [->] (node cs:name=q4[6]) -- (node cs:name=q5[6])[thick];
\draw [->] (node cs:name=q6[6]) -- (node cs:name=q5[6])[thick]
node[above right] {$4$};
\draw [->] (node cs:name=q2[6]) -- (node cs:name=q6[6])[thick]
node[below] {$6$};
\draw [->] (node cs:name=q5[6]) -- (node cs:name=q9[6])[thick]
node[below] {$9$};
\draw [->] (node cs:name=q5[6]) -- (node cs:name=q8[6]) [thick]node[above] {$8$};
\draw [->] (node cs:name=q8[6]) -- (node cs:name=q2[6])[thick]
node[above] {$2$};
\draw [->] (node cs:name=q9[6]) -- (node cs:name=q7[6]) [thick]node[below] {$7$};
\draw [->] (node cs:name=q1[6]) -- (node cs:name=q8[6])[thick];
\draw [->] (node cs:name=q8[6]) -- (node cs:name=q4[6])[thick]
node[above left] {$5$};
\draw [->] (node cs:name=q3[6]) -- (node cs:name=q1[6])[thick]
node[above] {$1$};
\draw [->] (node cs:name=q4[6]) -- (node cs:name=q3[6])[thick]
node[below] {$3$};
\draw [->] (node cs:name=q7[6]) -- (node cs:name=q4[6])[thick];
}
\draw (-1,9) circle(0pt) node[below]{$B=B^{(1)}$};
\draw[<->] (3,4) -- node[auto=left] {$\mu_8$} (3,3.3);
\draw (-1,9-7) circle(0pt) node[below]{$B^{(2)}$};
\draw[<->] (3,4-7) -- node[auto=left] {$\mu_5$} (3,3.3-7);
\draw (-1,9-14) circle(0pt) node[below]{$B^{(3)}$};
\draw[<->] (7,0-7) -- node[auto=left] {$\mu_4$} (7.7,0-7);
\draw (-1+17,9-14) circle(0pt) node[below]{$B^{(4)}$};
\draw[<->] (3+9,3.3-7) -- node[auto=right] {$\mu_8$} (3+9,4-7);
\draw (-1+17,9-7) circle(0pt) node[below]{$B^{(5)}$};
\draw[<->] (3+9,3.3) -- node[auto=right] {$\sigma_{45}$} (3+9,4);
\draw (-1+18,9) circle(0pt) node[below]{$B'=B^{(6)}$};
\end{tikzpicture}
\end{align*}
\caption{The quivers $B=B^{(1)}, \ldots, B^{(6)}=B'$ and the mutations connecting them.
We do not consider the wiring diagrams corresponding to the intermediate ones 
$B^{(2)}, \ldots, B^{(5)}$.}
\label{fig:mus}
\end{figure}

We introduce the cluster transformation
$\widehat{R}: \mathcal{Y}(B')  \rightarrow \mathcal{Y}(B)$
corresponding to the mutation sequence (\ref{mus}) by applying 
(\ref{eq:ad-tau-decomp}) as\footnote{For simplicity 
we identify $Y^{(t)}_i$ and $\rY_i(t)$ in the description from now on.}
\begin{align}\label{rhat}
\widehat{R} = 
\mathrm{Ad}\bigl(\Psi_q((Y^{(1)}_8)^{\varepsilon_1})^{\varepsilon_1}\bigr)\tau_{8,\varepsilon_1}
\mathrm{Ad}\bigl(\Psi_q((Y^{(2)}_5)^{\varepsilon_2})^{\varepsilon_2}\bigr)\tau_{5,\varepsilon_2}
\mathrm{Ad}\bigl(\Psi_q((Y^{(3)}_4)^{\varepsilon_3})^{\varepsilon_3}\bigr)\tau_{4,\varepsilon_3}
\mathrm{Ad}\bigl(\Psi_q((Y^{(4)}_8)^{\varepsilon_4})^{\varepsilon_4}\bigr)\tau_{8,\varepsilon_4}
\sigma_{45}.
\end{align}
The selection of $(\varepsilon_1, \varepsilon_2, \varepsilon_3, \varepsilon_4) \in \{1,-1\}^4$ 
influences the expressions, but $\widehat{R}$ itself remains independent of it.
For reader's convenience, we list the action of $\tau_{k,\varepsilon}$ appearing 
in (\ref{rhat}) explicitly.
\vspace{0.2cm}

\begin{tabular}{c|ccccccccc}
 & $Y^{(5)}_1$ & $Y^{(5)}_2$ & $Y^{(5)}_3$ & $Y^{(5)}_4$ & $Y^{(5)}_5$ 
 & $Y^{(5)}_6$ & $Y^{(5)}_7$ & $Y^{(5)}_8$ & $Y^{(5)}_9$ \\
\hline
$\tau_{8,+}$ & $Y^{(4)}_1$ &$q^{-1}Y^{(4)}_2Y^{(4)}_8$  & $Y^{(4)}_3$  &  
$q^{-1}Y^{(4)}_4Y^{(4)}_8$
& $Y^{(4)}_5$ & $Y^{(4)}_6$ & $Y^{(4)}_7$  &  $(Y^{(4)}_8)^{-1}$ & $Y^{(4)}_9$
\\
$\tau_{8,-}$ & $qY^{(4)}_1Y^{(4)}_8$ & $Y^{(4)}_2$ &  $Y^{(4)}_3$ & $Y^{(4)}_4$
& $qY^{(4)}_5Y^{(4)}_8$ & $Y^{(4)}_6$ & $Y^{(4)}_7$  & $(Y^{(4)}_8)^{-1}$ & $Y^{(4)}_9$
\end{tabular}

\noindent
\vspace{0.5cm}

\begin{tabular}{c|ccccccccc}
 & $Y^{(4)}_1$ & $Y^{(4)}_2$ & $Y^{(4)}_3$ & $Y^{(4)}_4$ & $Y^{(4)}_5$ 
 & $Y^{(4)}_6$ & $Y^{(4)}_7$ & $Y^{(4)}_8$ & $Y^{(4)}_9$ \\
\hline
$\tau_{4,+}$ & $Y^{(3)}_1$ &  $Y^{(3)}_2$ & $q^{-1}Y^{(3)}_3Y^{(3)}_4$ & $(Y^{(3)}_4)^{-1}$ 
& $Y^{(3)}_5$ & $Y^{(3)}_6$ &  $Y^{(3)}_7$ & $qY^{(3)}_4Y^{(3)}_8$ & $Y^{(3)}_9$
\\
$\tau_{4,-}$ & $qY^{(3)}_1Y^{(3)}_4$ &  $Y^{(3)}_2$ &  $Y^{(3)}_3$ & $(Y^{(3)}_4)^{-1}$
& $Y^{(3)}_5$ & $Y^{(3)}_6$ & $q^{-1}Y^{(3)}_4Y^{(3)}_7$ & $Y^{(3)}_8$ & $Y^{(3)}_9$
\end{tabular}

\noindent
\vspace{0.5cm}

\begin{tabular}{c|ccccccccc}
 & $Y^{(3)}_1$ & $Y^{(3)}_2$ & $Y^{(3)}_3$ & $Y^{(3)}_4$ & $Y^{(3)}_5$ 
 & $Y^{(3)}_6$ & $Y^{(3)}_7$ & $Y^{(3)}_8$ & $Y^{(3)}_9$ \\
\hline
$\tau_{5,+}$ & $Y^{(2)}_1$ & $q^{-1}Y^{(2)}_2Y^{(2)}_5$ & $Y^{(2)}_3$  &  
$Y^{(2)}_4$ & $(Y^{(2)}_5)^{-1}$ &  $Y^{(2)}_6$ & $Y^{(2)}_7$ &   $Y^{(2)}_8$ & $qY^{(2)}_5Y^{(2)}_9$
\\
$\tau_{5,-}$ & $Y^{(2)}_1$  & $Y^{(2)}_2$ & $Y^{(2)}_3$ & $Y^{(2)}_4$ & $(Y^{(2)}_5)^{-1}$ 
& $q^{-1}Y^{(2)}_5Y^{(2)}_6$ & $Y^{(2)}_7$ & $q^{-1}Y^{(2)}_5Y^{(2)}_8$ & $Y^{(2)}_9$
\end{tabular}

\noindent
\vspace{0.5cm}

\begin{tabular}{c|ccccccccc}
 & $Y^{(2)}_1$ & $Y^{(2)}_2$ & $Y^{(2)}_3$ & $Y^{(2)}_4$ & $Y^{(2)}_5$ 
 & $Y^{(2)}_6$ & $Y^{(2)}_7$ & $Y^{(2)}_8$ & $Y^{(2)}_9$ \\
\hline
$\tau_{8,+}$ & $Y^{(1)}_1$ & $Y^{(1)}_2$  &  $Y^{(1)}_3$  &  
$q^{-1}Y^{(1)}_4Y^{(1)}_8$
& $Y^{(1)}_5$ & $Y^{(1)}_6$ & $Y^{(1)}_7$ &  $(Y^{(1)}_8)^{-1}$ & $qY^{(1)}_8Y^{(1)}_9$
\\
$\tau_{8,-}$ & $Y^{(1)}_1$  & $Y^{(1)}_2$ &  $Y^{(1)}_3$ & $Y^{(1)}_4$
& $qY^{(1)}_5Y^{(1)}_8$ & $Y^{(1)}_6$ & $qY^{(1)}_7Y^{(1)}_8$  & $(Y^{(1)}_8)^{-1}$ & $Y^{(1)}_9$
\end{tabular}

\vspace{0.2cm}
\noindent
The tables mean, for instance  $\tau_{8,+}(Y^{(5)}_1) = Y^{(4)}_1$ and 
$\tau_{8,-}(Y^{(5)}_1) = qY^{(4)}_1Y^{(4)}_8$.
We set 
\begin{align}\label{taue}
\tau_{\varepsilon_1, \varepsilon_2, \varepsilon_3, \varepsilon_4} = 
\tau_{8,\varepsilon_1}  
\tau_{5,\varepsilon_2}  
\tau_{4,\varepsilon_3}  
\tau_{8,\varepsilon_4}  \sigma_{45}:\;
\mathcal{Y}(B')  \rightarrow \mathcal{Y}(B),
 \end{align}
 and call it the monomial part of $\widehat{R}$.
 
\vspace{0.5cm}
\begin{example}\label{ex:1}
$\tau_{--++}$ and $\tau_{--++}^{-1}$ are given as follows: 
\begin{align}
&\tau_{--++}: \left\{ ~~
\begin{alignedat}{3}
Y'_1 & \mapsto Y_1,\quad &
Y'_4 & \mapsto q Y^{-1}_5Y^{-1}_8, \quad &
Y'_7 & \mapsto qY_7Y_8,
\\
Y'_2 & \mapsto Y_2Y_4Y_5,\quad & 
Y'_5 & \mapsto Y_5,\quad  &
Y'_8 &  \mapsto qY^{-1}_4Y^{-1}_5,
\\
Y'_3 &\mapsto q^{-1}Y_3 Y_4,\quad &
Y'_6 & \mapsto Y_5Y_6Y_8,\quad &
Y'_9 &\mapsto Y_9,
\end{alignedat}
\right.
\\
&\tau_{--++}^{-1}: \left\{ ~~
\begin{alignedat}{3}
Y_1 & \mapsto Y'_1, \quad &
Y_4 & \mapsto q Y'^{-1}_5Y'^{-1}_8, \quad &
Y_7 & \mapsto Y'_5Y'_4Y'_7,
\\
Y_2 & \mapsto q Y'_2Y'_8, \quad & 
Y_5 & \mapsto Y'_5, \quad & 
Y_8 & \mapsto qY'^{-1}_4Y'^{-1}_5,
\\
Y_3 & \mapsto Y'_5Y'_3Y'_8, \quad &
Y_6 & \mapsto qY'_4Y'_6, \quad &
Y_9 & \mapsto Y'_9.
\end{alignedat}
\right.
\displaybreak[0]
\end{align}
By using them, $\widehat{R}$ in (\ref{rhat}) for the choice 
$(\varepsilon_1,  \varepsilon_2,  \varepsilon_3,  \varepsilon_4)=(-,-,+,+)$ 
is expressed as
\begin{align}
\widehat{R} &= 
\mathrm{Ad}\bigl(\Psi_q((Y^{(1)}_8)^{-1})^{-1}\bigr)\tau_{8,-}
\mathrm{Ad}\bigl(\Psi_q((Y^{(2)}_5)^{-1})^{-1}\bigr)\tau_{5,-}
\mathrm{Ad}\bigl(\Psi_q(Y^{(3)}_4)\bigr)\tau_{4,+}
\mathrm{Ad}\bigl(\Psi_q(Y^{(4)}_8)\bigr)\tau_{8,+}\sigma_{45}
\nonumber \\
&= 
\mathrm{Ad}\bigl(\Psi_q(Y_8^{-1})^{-1}\bigr)
\mathrm{Ad}\bigl(\Psi_q(qY^{-1}_5Y^{-1}_8)^{-1}\bigr)
\mathrm{Ad}\bigl(\Psi_q(Y_4)\bigr)
\mathrm{Ad}\bigl(\Psi_q(qY_4Y_5)\bigr)\tau_{--++}
\label{dode0}\\
&= 
\mathrm{Ad}\bigl(\Psi_q(Y_8^{-1})^{-1}\bigr)
\mathrm{Ad}\bigl(\Psi_q(qY^{-1}_5Y^{-1}_8)^{-1}\bigr)\tau_{--++}
\mathrm{Ad}\bigl(\Psi_q(qY'^{-1}_5Y'^{-1}_8)\bigr)
\mathrm{Ad}\bigl(\Psi_q(Y'^{-1}_8)\bigr).
\label{dode}
\displaybreak[0]
\end{align}
The formula (\ref{dode}) is derived from (\ref{dode0}) 
by moving $\tau_{--++}$ to the left by using  $\tau_{--++}^{-1}$.
In (\ref{dode}), the monomial part $\tau_{--++}$ may be replaced by 
$\mathrm{Ad}\bigl(\Psi_q(Y_5^{-1})^{-1}\bigr)\tau_{--++}
\mathrm{Ad}\bigl(\Psi_q(Y'^{-1}_5)\bigr)$ since $\tau_{--++}Y_5'  = Y_5\tau_{--++}$.
Then one can apply the pentagon identity (\ref{penta}) to the first and the last three $\Psi_q$'s to deduce a relatively simple formula 
\begin{align}
\widehat{R} = 
\mathrm{Ad}\bigl(\Psi_q(Y_5^{-1})^{-1}\Psi_q(Y_8^{-1})^{-1}\bigr)\tau_{--++}
\mathrm{Ad}\bigl(\Psi_q(Y'^{-1}_8)\Psi_q(Y'^{-1}_5)\bigr).
\label{Rf1}
\end{align}
\end{example}

\begin{example}\label{ex:15}
$\tau_{+-+-}$ and $\tau_{+-+-}^{-1}$ are given as follows: 
\begin{align}
&\tau_{+-+-}: \left\{ ~~
\begin{alignedat}{3}
Y'_1 & \mapsto q^2Y_1Y_4Y_5,\quad &
Y'_4 & \mapsto Y_4, \quad &
Y'_7 & \mapsto Y_7,
\\
Y'_2 & \mapsto Y_2,\quad & 
Y'_5 & \mapsto q^{-1}Y^{-1}_4Y^{-1}_8,\quad  &
Y'_8 &  \mapsto qY^{-1}_4Y^{-1}_5,
\\
Y'_3 &\mapsto q^{-2}Y_3 Y_4Y_8,\quad &
Y'_6 & \mapsto q^{-1}Y_5Y_6,\quad &
Y'_9 &\mapsto qY_8Y_9,
\end{alignedat}
\right.
\\
&\tau_{+-+-}^{-1}: \left\{ ~~
\begin{alignedat}{3}
Y_1 & \mapsto q^{-1}Y'_1Y'_8, \quad &
Y_4 & \mapsto Y'_4, \quad &
Y_7 & \mapsto Y'_7,
\\
Y_2 & \mapsto Y'_2, \quad & 
Y_5 & \mapsto q^{-1}Y'^{-1}_4Y'^{-1}_8, \quad & 
Y_8 & \mapsto qY'^{-1}_4Y'^{-1}_5,
\\
Y_3 & \mapsto qY'_3Y'_5, \quad &
Y_6 & \mapsto Y'_4Y'_6Y'_8, \quad &
Y_9 & \mapsto Y'_4Y'_5Y'_9.
\end{alignedat}
\right.
\displaybreak[0]
\end{align}
By using them, $\widehat{R}$ in (\ref{rhat}) for the choice 
$(\varepsilon_1,  \varepsilon_2,  \varepsilon_3,  \varepsilon_4)=(+,-,+,-)$ 
is expressed as
\begin{align}
\widehat{R} &= 
\mathrm{Ad}\bigl(\Psi_q(Y^{(1)}_8)\bigr)\tau_{8,+}
\mathrm{Ad}\bigl(\Psi_q((Y^{(2)}_5)^{-1})^{-1}\bigr)\tau_{5,-}
\mathrm{Ad}\bigl(\Psi_q(Y^{(3)}_4)\bigr)\tau_{4,+}
\mathrm{Ad}\bigl(\Psi_q((Y^{(4)}_8)^{-1})^{-1}\bigr)\tau_{8,-}\sigma_{45}
\nonumber\\
&= 
\mathrm{Ad}\bigl(\Psi_q(Y_8)\bigr)
\mathrm{Ad}\bigl(\Psi_q(Y^{-1}_5)^{-1}\bigr)
\mathrm{Ad}\bigl(\Psi_q(q^{-1}Y_4Y_8)\bigr)
\mathrm{Ad}\bigl(\Psi_q(qY^{-1}_4Y^{-1}_5)^{-1}\bigr)\tau_{+-+-}
\label{dode7}\\
&= 
\mathrm{Ad}\bigl(\Psi_q(Y_8)\Psi_q(Y^{-1}_5)^{-1}\bigr)
\tau_{+-+-}
\mathrm{Ad}\bigl(\Psi_q(Y'^{-1}_5)\Psi_q(Y'_8)^{-1}\bigr).
\label{dode8}
\end{align}
\end{example}

In addition, we note the fact that $\tau_{-+-+} = \tau_{++--}$.
For comparison, we provide a list of the formulas for $\widehat{R}$ corresponding to all possible sign choices in 
Appendix \ref{app:RY}.
Performing a straightforward calculation using any one 
of the formulas for $\widehat{R}$ in Example \ref{ex:1}, Example \ref{ex:15} or in Appendix \ref{app:RY}, 
one arrives at 
\begin{proposition}\label{pr:RY}
The cluster transformation $\widehat{R}: \mathcal{Y}(B')  \rightarrow \mathcal{Y}(B)$ is given by
\begin{alignat}{3}
Y'_1  &\mapsto \Lambda_4 Y_1,  &
Y'_4 & \mapsto \Lambda_4 Y^{-1}_5 \Lambda_8^{-1}, \quad &
Y'_7 & \mapsto \Lambda_8 Y_7,
\nonumber\\
Y'_2 & \mapsto Y_2Y_4Y_5\Lambda^{-1}_4, \quad & 
Y'_5 & \mapsto \Lambda_5 Y^{-1}_8 \Lambda^{-1}_4, &
Y'_8 & \mapsto \Lambda_8 Y^{-1}_4 \Lambda_5^{-1},
\label{RYY}\\
Y'_3 & \mapsto Y_3Y_8Y_4 \Lambda^{-1}_8, &
Y'_6 & \mapsto \Lambda_5 Y_6, & 
Y'_9 & \mapsto Y_9Y_5Y_8 \Lambda^{-1}_5,
\nonumber
\displaybreak[0]
\end{alignat}
where $\Lambda_4, \Lambda_5, \Lambda_6$ read as 
\begin{align}\label{aldef}
\Lambda_4 = 1 + q^{-1}Y_4 + Y_4 Y_5,\quad
\Lambda_5 = 1 + q^{-1}Y_5 + Y_5 Y_8,\quad
\Lambda_8 = 1+ q^{-1} Y_8 + Y_8 Y_4.
\end{align}
\end{proposition}

\begin{remark}\label{re:c2}
$\widehat{R}$ preserves the center in Remark \ref{re:c1}.  
In fact, using (\ref{RYY}) one can easily check 
\begin{align*} 
\widehat{R}(Y'_9Y'_5Y'_1) = Y_9Y_5Y_1,\quad 
\widehat{R}(Y'_3Y'_8Y'_6) = Y_3Y_8Y_6,\quad 
\widehat{R}(Y'_2Y'_4Y'_7) = Y_2Y_4Y_7.
\end{align*}
\end{remark}

\subsection{$\widehat{R}$ satisfies the tetrahedron equation}
Up to this point in our construction, 
$\widehat{R}$ is a transformation between the 9 variables $\{Y_1,\ldots, Y_9\}$ and $\{Y'_1,\ldots, Y'_9\}$.
In the upcoming section, we will create a scenario where it can be attributed to a transformation of variables,
we call them canonical variables or $q$-Weyl variables,  
located at the vertices $1, 2, 3$ of the wiring diagram (highlighted in red) in Figure \ref{fig:R123}.

\begin{figure}[H]
\begin{align*}
\begin{tikzpicture}
\begin{scope}[>=latex,xshift=0pt]
{\color{red}
\fill (1,0.5) circle(2pt) coordinate(A) node[below]{1};
\fill (2,1.5) circle(2pt) coordinate(B) node[above]{2};
\fill (3,0.5) circle(2pt) coordinate(C) node[below]{3};
}
\draw [-] (0,2) to [out = 0, in = 135] (B);
\draw [-] (B) -- (C); 
\draw [->] (C) to [out = -45, in = 180] (4,0);
\draw [-] (0,1) to [out = 0, in = 135] (A); 
\draw [-] (A) to [out = -45, in = -135] (C);
\draw [->] (C) to [out = 45, in = 180] (4,1);
\draw [-] (0,0) to [out = 0, in = -135] (A); 
\draw [-] (A) -- (B);
\draw [->] (B) to [out = 45, in = 180] (4,2);
\coordinate (P1) at (4.5,1);
\coordinate (P2) at (5.5,1);
{\color{red}
\draw[<-] (P1) -- (P2);
\draw (5,1) circle(0pt) node[above]{$\widehat{R}_{123}$};
}
{\color{blue}
\draw (1.3,1.9) circle(2pt) coordinate(1) node[above]{$1$};
\draw (2.7,1.9) circle(2pt) coordinate(2) node[above]{$2$};
\draw (0.4,0.9) circle(2pt) coordinate(3) node[above]{$3$};
\draw (1.5,1) circle(2pt) coordinate(4) node[above left]{$4$};
\draw (2.5,1) circle(2pt) coordinate(5) node[above right]{$5$};
\draw (3.6,0.9) circle(2pt) coordinate(6) node[above]{$6$};
\draw (0.4,0.1) circle(2pt) coordinate(7) node[below]{$7$};
\draw (2,0.1) circle(2pt) coordinate(8) node[below]{$8$};
\draw (3.6,0.1) circle(2pt) coordinate(9) node[below]{$9$};
\qarrow{1}{2}
\qarrow{2}{5}
\qarrow{5}{4}
\qarrow{4}{1}
\qarrow{4}{8}
\qarrow{8}{7}
\qarrow{7}{3}
\qarrow{3}{4}
\qarrow{5}{6}
\qarrow{6}{9}
\qarrow{9}{8}
\qarrow{8}{5}
}
\end{scope}
\begin{scope}[>=latex,xshift=170pt]
{\color{red}
\fill (3,1.5) circle(2pt) coordinate(A) node[above]{1};
\fill (2,0.5) circle(2pt) coordinate(B) node[below]{2};
\fill (1,1.5) circle(2pt) coordinate(C) node[above]{3};
}
\draw [-] (0,0) to [out = 0, in = -135] (B);
\draw [-] (B) -- (A); 
\draw [->] (A) to [out = 45, in = 180] (4,2);
\draw [-] (0,1) to [out = 0, in = -135] (C); 
\draw [-] (C) to [out = 45, in = 135] (A);
\draw [->] (A) to [out = -45, in = 180] (4,1);
\draw [-] (0,2) to [out = 0, in = 135] (C); 
\draw [-] (C) -- (B);
\draw [->] (B) to [out = -45, in = 180] (4,0);
{\color{blue}
\draw (1.3,0.1) circle(2pt) coordinate(7) node[below]{$7$};
\draw (2.7,0.1) circle(2pt) coordinate(9) node[below]{$9$};
\draw (0.4,1.05) circle(2pt) coordinate(3) node[below]{$3$};
\draw (1.5,1) circle(2pt) coordinate(5) node[below left]{$5$};
\draw (2.5,1) circle(2pt) coordinate(4) node[below right]{$4$};
\draw (3.6,1.05) circle(2pt) coordinate(6) node[below]{$6$};
\draw (0.4,1.9) circle(2pt) coordinate(1) node[above]{$1$};
\draw (2,1.9) circle(2pt) coordinate(8) node[above]{$8$};
\draw (3.6,1.9) circle(2pt) coordinate(2) node[above]{$2$};
\qarrow{1}{8}
\qarrow{8}{5}
\qarrow{5}{3}
\qarrow{3}{1}
\qarrow{4}{8}
\qarrow{8}{2}
\qarrow{2}{6}
\qarrow{6}{4}
\qarrow{5}{4}
\qarrow{4}{9}
\qarrow{9}{7}
\qarrow{7}{5}
}
\end{scope}
\end{tikzpicture}
\end{align*}
\caption{Cluster transformation $\widehat{R}_{123}$, which acts on the $q$-Weyl variables 
attached to the vertices $1,2,3$ of the wiring diagram colored in red.}
\label{fig:R123}
\end{figure}
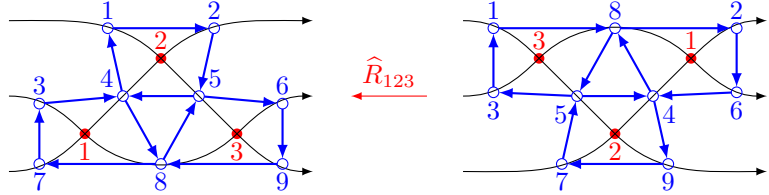

\noindent
A precise description will be given in (\ref{Rcom}).
Here we suppose that it is the case and denote $\widehat{R}$ by $\widehat{R}_{123}$.
The following result is essentially due to \cite{SY22}.
\begin{proposition}\label{pr:sy}
$\widehat{R}$ satisfies the tetrahedron equation:
\begin{align}\label{TE1}
\widehat{R} _{456}\widehat{R} _{236}\widehat{R} _{135} \widehat{R} _{124}=
\widehat{R} _{124}\widehat{R} _{135}\widehat{R} _{236}\widehat{R} _{456} .
\end{align}
\end{proposition}
\begin{proof}
For each reduced word for  
the longest element of the Weyl group $W(A_3)$, draw a wiring diagram and a square quiver
extending Figure \ref{fig:R123} naturally.
The quivers and the canonical variables living on the crossings of the wiring diagrams (red vertices $1,\ldots, 6$)
are connected by the cluster transformations 
$\widehat{R}_{ijk}$ as in Figure \ref{fig:rikisaku}.
\vspace*{-2cm}
\begin{figure}[H]
\hspace*{-1.1cm}
\includegraphics[clip,scale=0.85]{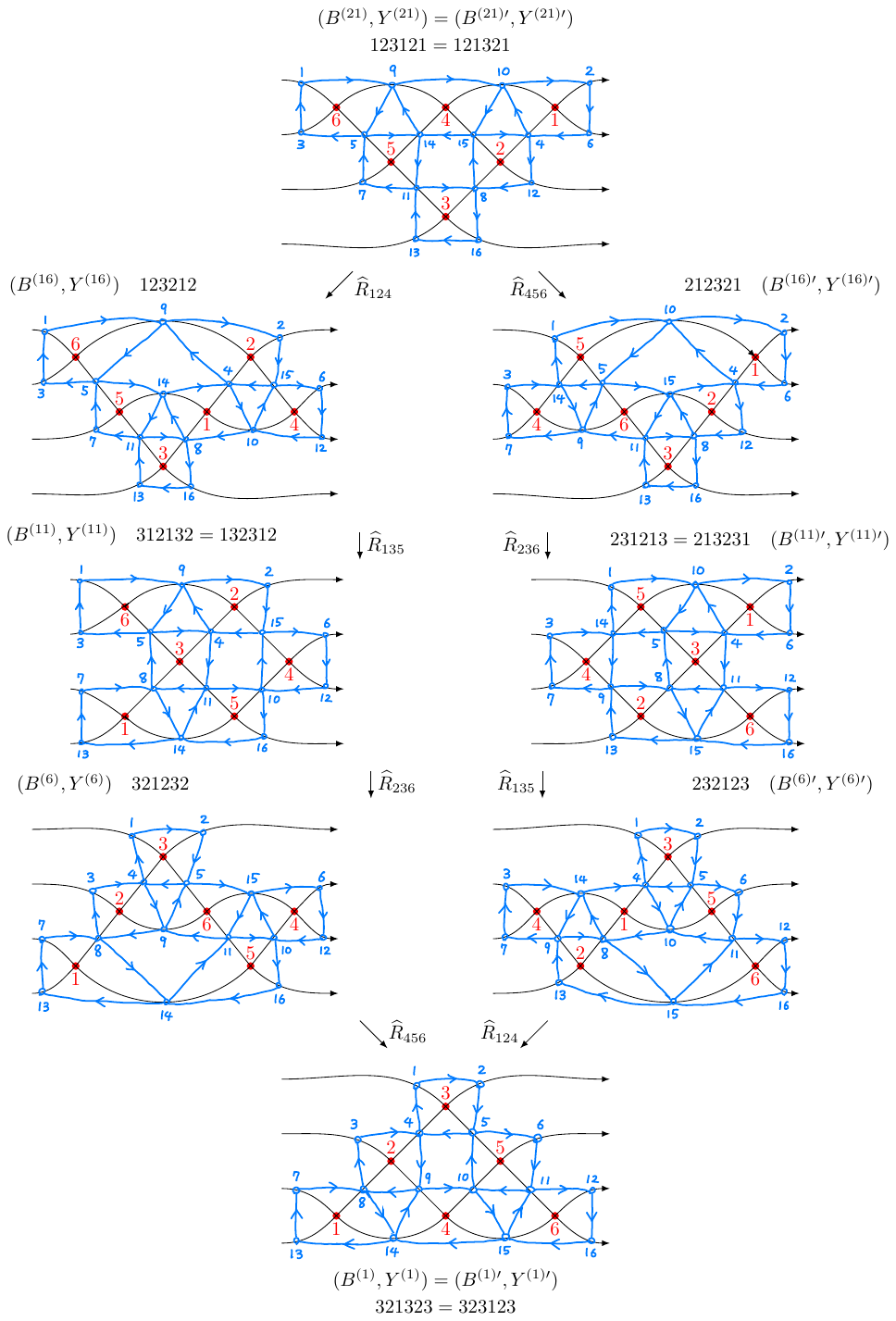}
\vspace*{-2.5cm}
\caption{Cluster transformations $\widehat{R}_{ijk}$.
The wiring diagrams have 6 crossings (red).
The quivers (blue) have 16 vertices.
The seeds $(B^{(t)}, \bY^{(t)})$ and $(B^{(t)\prime}, \bY^{(t)\prime})$ will be 
explained in detail in Section \ref{ss:ms}.}
\label{fig:rikisaku}
\end{figure}

\noindent
In Figure \ref{fig:rikisaku},
let $\nu$ and $\nu'$ be the mutation sequences corresponding to the left path 
$(B^{(1)},Y^{(1)}) \rightarrow 
(B^{(6)},Y^{(6)}) \rightarrow 
(B^{(11)},Y^{(11)}) \rightarrow 
(B^{(16)},Y^{(16)}) \rightarrow 
(B^{(21)},Y^{(21)})=\nu(B^{(1)},Y^{(1)})$ 
and the right path 
$(B^{(1)\prime},Y^{(1)\prime}) \rightarrow 
(B^{(6)\prime},Y^{(6)\prime}) \rightarrow 
(B^{(11)\prime},Y^{(11)\prime}) \rightarrow 
(B^{(16)\prime},Y^{(16)\prime}) \rightarrow 
 (B^{(21)\prime},Y^{(21)\prime})=\nu'(B^{(1)},Y^{(1)})$, 
 respectively.
Let $\nu(B^{(1)},y^{(1)})$ and $\nu'(B^{(1)},y^{(1)})$ be the 
tropical $y$-seeds generated by the same mutation sequences. 
It has been checked \cite[Sec. A.2]{SY22} 
that they satisfy the equality  
$\nu(B^{(1)},y^{(1)})= \nu'(B^{(1)},y^{(1)})$.
Thus Theorem \ref{thm:piv} enforces the equality of quantum $y$-seeds
$\nu(B^{(1)}, Y^{(1)}) = \nu'(B^{(1)\prime}, Y^{(1)\prime})$.
In terms of cluster transformations, it is means the tetrahedron equation  
$\widehat{R}_{456}\widehat{R}_{236}\widehat{R}_{135}\widehat{R}_{124} = 
\widehat{R}_{124}\widehat{R}_{135}\widehat{R}_{236}\widehat{R}_{456}$.
\end{proof}

\subsection{Monomial solutions to the tetrahedron equation}\label{ss:ms}

In this subsection we provide additional details regarding Figure \ref{fig:rikisaku} and Proposition \ref{pr:sy}.
Let $(B^{(1)}, \bY^{(1)}) = (B^{(1)\prime}, \bY^{(1)\prime})$ 
be the initial quantum $y$-seed corresponding to the quiver at the bottom of  Figure \ref{fig:rikisaku}.
The quantum $y$-seeds $(B^{(t)}, \bY^{(t)})$ and 
$ (B^{(t)\prime}, \bY^{(t)\prime})\, (t=2,\ldots, 21)$, which pertain 
to the left and the right paths are determined from it by the mutation sequences, and 
we have just shown that the 
final results coincide, i.e., $(B^{(21)}, \bY^{(21)}) = (B^{(21)\prime}, \bY^{(21)\prime})$.
Set $\bY^{(t)}=(Y^{(t)}_1,\ldots, Y^{(t)}_{16})$ and 
$\bY^{(t)\prime}=(Y^{(t)\prime}_1,\ldots, Y^{(t)\prime}_{16})$.

The quantum $y$-seeds 
$(B^{(t)}, \bY^{(t)})\, (t=2,\ldots, 21)$ are determined from the initial one 
$(B^{(1)}, \bY^{(1)}) $ by
\begin{equation}\label{b21L}
\begin{split}
&
(B^{(1)}, \bY^{(1)}) \underset{\varepsilon_1}{\overset{\mu_{15}}{\longleftrightarrow}}
(B^{(2)}, \bY^{(2)})  \underset{\varepsilon_2}{\overset{\mu_{11}}{\longleftrightarrow}}
(B^{(3)}, \bY^{(3)})  \underset{\varepsilon_3}{\overset{\mu_{10}}{\longleftrightarrow}}
(B^{(4)}, \bY^{(4)})   \underset{\varepsilon_4}{\overset{\mu_{15}}{\longleftrightarrow}}
(B^{(5)}, \bY^{(5)})  \overset{\sigma_{10,11}}{\longleftrightarrow} 
(B^{(6)}, \bY^{(6)}),
\\
&
(B^{(6)}, \bY^{(6)}) \underset{\varepsilon_1}{\overset{\mu_{9}}{\longleftrightarrow}}
(B^{(7)}, \bY^{(7)})   \underset{\varepsilon_2}{\overset{\mu_{5}}{\longleftrightarrow}} 
(B^{(8)}, \bY^{(8)})  \underset{\varepsilon_3}{\overset{\mu_{4}}{\longleftrightarrow}}
(B^{(9)}, \bY^{(9)})   \underset{\varepsilon_4}{\overset{\mu_{9}}{\longleftrightarrow} }
(B^{(10)}, \bY^{(10)})  \overset{\sigma_{4,5}}{\longleftrightarrow} 
(B^{(11)}, \bY^{(11)}),
\\
&
(B^{(11)}, \bY^{(11)}) \underset{\varepsilon_1}{\overset{\mu_{14}}{\longleftrightarrow} }
(B^{(12)}, \bY^{(12)})   \underset{\varepsilon_2}{\overset{\mu_{11}}{\longleftrightarrow} }
(B^{(13)}, \bY^{(13)})  \underset{\varepsilon_3}{\overset{\mu_{8}}{\longleftrightarrow} }
(B^{(14)}, \bY^{(14)})  \underset{\varepsilon_4}{ \overset{\mu_{14}}{\longleftrightarrow}} 
(B^{(15)}, \bY^{(15)})  \overset{\sigma_{8,11}}{\longleftrightarrow} 
(B^{(16)}, \bY^{(16)}),
\\
&
(B^{(16)}, \bY^{(16)}) \underset{\varepsilon_1}{\overset{\begin{color}{red}\mu_{10}\end{color}}{\longleftrightarrow}}
(B^{(17)}, \bY^{(17)})   \underset{\varepsilon_2}{\overset{\begin{color}{red}\mu_{15}\end{color}}{\longleftrightarrow}} 
(B^{(18)}, \bY^{(18)})  \underset{\varepsilon_3}{\overset{\mu_{4}}{\longleftrightarrow}}
(B^{(19)}, \bY^{(19)})  \underset{\varepsilon_4}{ \overset{\mu_{10}}{\longleftrightarrow}} 
(B^{(20)}, \bY^{(20)})  \overset{\sigma_{4,15}}{\longleftrightarrow} 
(B^{(21)}, \bY^{(21)}),
\end{split}
\end{equation}
where the notation is parallel with (\ref{mus}).
In particular the choice of signs $\varepsilon_1, \ldots, \varepsilon_4$ does not influence the mutations themselves.
According to (\ref{rhat}),  each line in the above corresponds to a cluster transformation appearing in the left path of 
Figure \ref{fig:rikisaku} as follows: 
\begin{equation}\label{r124a}
\begin{split}
\widehat{R}_{456} &= 
\mathrm{Ad}\bigl(\Psi_q((Y^{(1)}_{15})^{\varepsilon_1})^{\varepsilon_1}\bigr)\tau_{15,\varepsilon_1}
\mathrm{Ad}\bigl(\Psi_q((Y^{(2)}_{11})^{\varepsilon_2})^{\varepsilon_2}\bigr)\tau_{11,\varepsilon_2}
\mathrm{Ad}\bigl(\Psi_q((Y^{(3)}_{10})^{\varepsilon_3})^{\varepsilon_3}\bigr)\tau_{10,\varepsilon_3}
\mathrm{Ad}\bigl(\Psi_q((Y^{(4)}_{15})^{\varepsilon_4})^{\varepsilon_4}\bigr)\tau_{15,\varepsilon_4}
\sigma_{10,11},
\\
\widehat{R}_{236} &= 
\mathrm{Ad}\bigl(\Psi_q((Y^{(6)}_{9})^{\varepsilon_1})^{\varepsilon_1}\bigr)\tau_{9,\varepsilon_1}
\mathrm{Ad}\bigl(\Psi_q((Y^{(7)}_{5})^{\varepsilon_2})^{\varepsilon_2}\bigr)\tau_{5,\varepsilon_2}
\mathrm{Ad}\bigl(\Psi_q((Y^{(8)}_{4})^{\varepsilon_3})^{\varepsilon_3}\bigr)\tau_{4,\varepsilon_3}
\mathrm{Ad}\bigl(\Psi_q((Y^{(9)}_{9})^{\varepsilon_4})^{\varepsilon_4}\bigr)\tau_{9,\varepsilon_4}
\sigma_{4,5},
\\
\widehat{R}_{135} &= 
\mathrm{Ad}\bigl(\Psi_q((Y^{(11)}_{14})^{\varepsilon_1})^{\varepsilon_1}\bigr)\tau_{14,\varepsilon_1}
\mathrm{Ad}\bigl(\Psi_q((Y^{(12)}_{11})^{\varepsilon_2})^{\varepsilon_2}\bigr)\tau_{11,\varepsilon_2}
\mathrm{Ad}\bigl(\Psi_q((Y^{(13)}_{8})^{\varepsilon_3})^{\varepsilon_3}\bigr)\tau_{8,\varepsilon_3}
\mathrm{Ad}\bigl(\Psi_q((Y^{(14)}_{14})^{\varepsilon_4})^{\varepsilon_4}\bigr)\tau_{14,\varepsilon_4}
\sigma_{8,11},
\\
\widehat{R}_{124} &= 
\mathrm{Ad}\bigl(\Psi_q((Y^{(16)}_{10})^{\varepsilon_1})^{\varepsilon_1}\bigr)\tau_{10,\varepsilon_1}
\mathrm{Ad}\bigl(\Psi_q((Y^{(17)}_{15})^{\varepsilon_2})^{\varepsilon_2}\bigr)\tau_{15,\varepsilon_2}
\mathrm{Ad}\bigl(\Psi_q((Y^{(18)}_{4})^{\varepsilon_3})^{\varepsilon_3}\bigr)\tau_{4,\varepsilon_3}
\mathrm{Ad}\bigl(\Psi_q((Y^{(19)}_{10})^{\varepsilon_4})^{\varepsilon_4}\bigr)\tau_{10,\varepsilon_4}
\sigma_{4,15}.
\end{split}
\end{equation}

The quantum $y$-seeds 
$(B^{(t)\prime}, Y^{(t)\prime})\, (t=2,\ldots, 21)$ are determined  from the initial one 
$(B^{(1)\prime}, Y^{(1)\prime})$ by
\begin{equation}\label{b11R}
\begin{split}
&
(B^{(1)\prime}, \bY^{(1)\prime}) \underset{\varepsilon_1}{\overset{\mu_{14}}{\longleftrightarrow}} 
(B^{(2)\prime}, \bY^{(2)\prime}) \underset{\varepsilon_2}{\overset{\mu_{9}}{\longleftrightarrow}}
(B^{(3)\prime}, \bY^{(3)\prime}) \underset{\varepsilon_3}{\overset{\mu_{8}}{\longleftrightarrow}}
(B^{(4)\prime}, \bY^{(4)\prime}) \underset{\varepsilon_4}{\overset{\mu_{14}}{\longleftrightarrow}} 
(B^{(5)\prime}, \bY^{(5)\prime}) \overset{\sigma_{8,9}}{\longleftrightarrow} 
(B^{(6)\prime}, \bY^{(6)\prime}),
\\
&
(B^{(6)\prime}, \bY^{(6)\prime}) \underset{\varepsilon_1}{\overset{\mu_{10}}{\longleftrightarrow}}
(B^{(7)\prime}, \bY^{(7)\prime}) \underset{\varepsilon_2}{\overset{\mu_{5}}{\longleftrightarrow}}
(B^{(8)\prime}, \bY^{(8)\prime}) \underset{\varepsilon_3}{\overset{\mu_{4}}{\longleftrightarrow}}
(B^{(9)\prime}, \bY^{(9)\prime}) \underset{\varepsilon_4}{\overset{\mu_{10}}{\longleftrightarrow}} 
(B^{(10)\prime}, \bY^{(10)\prime}) \overset{\sigma_{4,5}}{\longleftrightarrow} 
(B^{(11)\prime}, \bY^{(11)\prime}),
\\
&
(B^{(11)\prime}, \bY^{(11)\prime}) \underset{\varepsilon_1}{\overset{\mu_{15}}{\longleftrightarrow}}
(B^{(12)\prime}, \bY^{(12)\prime}) \underset{\varepsilon_2}{\overset{\mu_{11}}{\longleftrightarrow}}
(B^{(13)\prime}, \bY^{(13)\prime}) \underset{\varepsilon_3}{\overset{\mu_{8}}{\longleftrightarrow}}
(B^{(14)\prime}, \bY^{(14)\prime}) \underset{\varepsilon_4}{\overset{\begin{color}{red}\mu_{15}\end{color}}{\longleftrightarrow}}
(B^{(15)\prime}, \bY^{(15)\prime}) \overset{\sigma_{8,11}}{\longleftrightarrow} 
(B^{(16)\prime}, \bY^{(16)\prime}),
\\
&
(B^{(16)\prime}, \bY^{(16)\prime}) \underset{\varepsilon_1}{\overset{\mu_{9}}{\longleftrightarrow} }
(B^{(17)\prime}, \bY^{(17)\prime}) \underset{\varepsilon_2}{\overset{\begin{color}{red}\mu_{5}\end{color}}{\longleftrightarrow} }
(B^{(18)\prime}, \bY^{(18)\prime}) \underset{\varepsilon_3}{\overset{\mu_{14}}{\longleftrightarrow} }
(B^{(19)\prime}, \bY^{(19)\prime}) \underset{\varepsilon_4}{\overset{\mu_{9}}{\longleftrightarrow} }
(B^{(20)\prime}, \bY^{(20)\prime}) \overset{\sigma_{5,14}}{\longleftrightarrow} 
(B^{(21)\prime}, \bY^{(21)\prime}).
\end{split}
\end{equation}
They correspond to the cluster transformations in the right path of Figure \ref{fig:rikisaku} as follows:
\begin{equation}\label{r456b}
\begin{split}
\widehat{R}_{124} &= 
\mathrm{Ad}\bigl(\Psi_q((Y^{(1)\prime}_{14})^{\varepsilon_1})^{\varepsilon_1}\bigr)\tau_{14,\varepsilon_1}
\mathrm{Ad}\bigl(\Psi_q((Y^{(2)\prime}_{9})^{\varepsilon_2})^{\varepsilon_2}\bigr)\tau_{9,\varepsilon_2}
\mathrm{Ad}\bigl(\Psi_q((Y^{(3)\prime}_{8})^{\varepsilon_3})^{\varepsilon_3}\bigr)\tau_{8,\varepsilon_3}
\mathrm{Ad}\bigl(\Psi_q((Y^{(4)\prime}_{14})^{\varepsilon_4})^{\varepsilon_4}\bigr)\tau_{14,\varepsilon_4}
\sigma_{8,9},
\\
\widehat{R}_{135} &= 
\mathrm{Ad}\bigl(\Psi_q((Y^{(6)\prime}_{10})^{\varepsilon_1})^{\varepsilon_1}\bigr)\tau_{10,\varepsilon_1}
\mathrm{Ad}\bigl(\Psi_q((Y^{(7)\prime}_{5})^{\varepsilon_2})^{\varepsilon_2}\bigr)\tau_{5,\varepsilon_2}
\mathrm{Ad}\bigl(\Psi_q((Y^{(8)\prime}_{4})^{\varepsilon_3})^{\varepsilon_3}\bigr)\tau_{4,\varepsilon_3}
\mathrm{Ad}\bigl(\Psi_q((Y^{(9)\prime}_{10})^{\varepsilon_4})^{\varepsilon_4}\bigr)\tau_{10,\varepsilon_4}
\sigma_{4,5},
\\
\widehat{R}_{236} &= 
\mathrm{Ad}\bigl(\Psi_q((Y^{(11)\prime}_{15})^{\varepsilon_1})^{\varepsilon_1}\bigr)\tau_{15,\varepsilon_1}
\mathrm{Ad}\bigl(\Psi_q((Y^{(12)\prime}_{11})^{\varepsilon_2})^{\varepsilon_2}\bigr)\tau_{11,\varepsilon_2}
\mathrm{Ad}\bigl(\Psi_q((Y^{(13)\prime}_{8})^{\varepsilon_3})^{\varepsilon_3}\bigr)\tau_{8,\varepsilon_3}
\mathrm{Ad}\bigl(\Psi_q((Y^{(14)\prime}_{15})^{\varepsilon_4})^{\varepsilon_4}\bigr)\tau_{15,\varepsilon_4}
\sigma_{8,11},
\\
\widehat{R}_{456} &= 
\mathrm{Ad}\bigl(\Psi_q((Y^{(16)\prime}_{9})^{\varepsilon_1})^{\varepsilon_1}\bigr)\tau_{9,\varepsilon_1}
\mathrm{Ad}\bigl(\Psi_q((Y^{(17)\prime}_{5})^{\varepsilon_2})^{\varepsilon_2}\bigr)\tau_{5,\varepsilon_2}
\mathrm{Ad}\bigl(\Psi_q((Y^{(18)\prime}_{14})^{\varepsilon_3})^{\varepsilon_3}\bigr)\tau_{14,\varepsilon_3}
\mathrm{Ad}\bigl(\Psi_q((Y^{(19)\prime}_{9})^{\varepsilon_4})^{\varepsilon_4}\bigr)\tau_{9,\varepsilon_4}
\sigma_{5,14}.
\end{split}
\end{equation}
Although the formulas (\ref{r124a}) and (\ref{r456b}) may appear distinct, 
they all signify the same transformation described in Proposition \ref{pr:RY}
for suitable subsets of Y-variables.
This fact justifies denoting them by the common symbol $\widehat{R}$.

\begin{remark}\label{re:red}
Consider the tropical $y$-seeds generated by the same mutation sequences from the initial one 
$(B^{(1)}, y^{(1)}) = (B^{(1)\prime}, y^{(1)\prime})$.
Suppose $y^{(1)}_i = y^{(1)\prime}_i$ is positive for all $i=1,\ldots, 16$.
Then the four mutations highlighted in red in (\ref{b21L}) and  (\ref{b11R})
are associated to a a negative tropical sign of the $y$-variable at the mutation point 
(the $y$-seed in the left), while the remaining ones are positive.
\end{remark}

Let us introduce the monomial parts of the cluster transformations  (\ref{r124a}) and (\ref{r456b}):
\begin{alignat*}{2}
\tau_{456   | \,  \varepsilon_1,\varepsilon_2,\varepsilon_3,\varepsilon_4}&:= 
\tau_{15,\varepsilon_1}\tau_{11,\varepsilon_2}\tau_{10,\varepsilon_3}\tau_{15,\varepsilon_4}\sigma_{10,11}:
\quad &\mathcal{Y}(B^{(6)}) &\rightarrow \mathcal{Y}(B^{(1)}),
\\
\tau_{236   | \,   \varepsilon_1,\varepsilon_2,\varepsilon_3,\varepsilon_4} &:=
\tau_{9,\varepsilon_1}\tau_{5,\varepsilon_2}\tau_{4,\varepsilon_3}\tau_{9,\varepsilon_4}\sigma_{4,5}:
&\mathcal{Y}(B^{(11)}) &\rightarrow \mathcal{Y}(B^{(6)}),
\\
\tau_{135   | \,   \varepsilon_1,\varepsilon_2,\varepsilon_3,\varepsilon_4} &:=
\tau_{14,\varepsilon_1}\tau_{11,\varepsilon_2}\tau_{8,\varepsilon_3}\tau_{14,\varepsilon_4}\sigma_{8,11}:
&\mathcal{Y}(B^{(16)}) &\rightarrow \mathcal{Y}(B^{(11)}),
\\
\tau_{124   | \,   \varepsilon_1,\varepsilon_2,\varepsilon_3,\varepsilon_4} &:=
\tau_{10,\varepsilon_1}\tau_{15,\varepsilon_2}\tau_{4,\varepsilon_3}\tau_{10,\varepsilon_4}\sigma_{4,15}:
&\mathcal{Y}(B^{(21)}) &\rightarrow \mathcal{Y}(B^{(16)}),
\\
\tau'_{124   | \,   \varepsilon_1,\varepsilon_2,\varepsilon_3,\varepsilon_4} &:=
\tau_{14,\varepsilon_1}\tau_{9,\varepsilon_2}\tau_{8,\varepsilon_3}\tau_{14,\varepsilon_4}\sigma_{8,9}:
\quad &\mathcal{Y}(B^{(6)\prime}) &\rightarrow \mathcal{Y}(B^{(1)\prime}),
\\
\tau'_{135   | \,   \varepsilon_1,\varepsilon_2,\varepsilon_3,\varepsilon_4} &:=
\tau_{10,\varepsilon_1}\tau_{5,\varepsilon_2}\tau_{4,\varepsilon_3}\tau_{10,\varepsilon_4}\sigma_{4,5}:
\quad &\mathcal{Y}(B^{(11)\prime}) &\rightarrow \mathcal{Y}(B^{(6)\prime}),
\\
\tau'_{236   | \,   \varepsilon_1,\varepsilon_2,\varepsilon_3,\varepsilon_4} &:=
\tau_{15,\varepsilon_1}\tau_{11,\varepsilon_2}\tau_{8,\varepsilon_3}\tau_{15,\varepsilon_4}\sigma_{8,11}:
\quad &\mathcal{Y}(B^{(16)\prime}) &\rightarrow \mathcal{Y}(B^{(11)\prime}),
\\
\tau'_{456   | \,   \varepsilon_1,\varepsilon_2,\varepsilon_3,\varepsilon_4} &:=
\tau_{9,\varepsilon_1}\tau_{5,\varepsilon_2}\tau_{14,\varepsilon_3}\tau_{9,\varepsilon_4}\sigma_{5,14}:
\quad &\mathcal{Y}(B^{(21)\prime}) &\rightarrow \mathcal{Y}(B^{(16)\prime}).
\displaybreak[0]
\end{alignat*}
The primes in $\tau'_{ijk   | \,   \varepsilon_1,\varepsilon_2,\varepsilon_3,\varepsilon_4}$ 
are added just for distinction.
The maps $\tau_{ijk   | \,   \varepsilon_1,\varepsilon_2,\varepsilon_3,\varepsilon_4}$ and 
$\tau'_{ijk   | \,   \varepsilon_1,\varepsilon_2,\varepsilon_3,\varepsilon_4}$ 
consistently adhere to $\tau_{\varepsilon_1,\varepsilon_2,\varepsilon_3,\varepsilon_4}$ in 
(\ref{taue})  with respect to the subset of Y-variables.

Now we are ready to explain monomial solutions to the tetrahedron equation.
Proposition \ref{pr:sy},  Figure \ref{fig:rikisaku} and Remark \ref{re:sgni} indicate 
the equality of the tropical $y$-variables $y^{(21)} = y^{(21)\prime}$ {\em provided} that 
all the signs associated with the monomial part of the mutation are chosen to be the tropical signs.
Considering Remark \ref{re:red} alongside, we find that 
\begin{equation}\label{ihte}
\tau_{456|\, + + + +} \tau_{236|\, + + + +}\tau_{135|\, + + + +}\tau_{124|\, - - + +}
=
\tau'_{124|\, + + + +}\tau'_{135|\, + + + +}\tau'_{236|\, + + + -}\tau'_{456|\, + - + +}
\end{equation}
is valid instead of the naive choice of $\tau_{ijk|\, ++++}$ and $\tau'_{ijk|\,++++}$ everywhere.
This is an inhomogeneous version of the tetrahedron equation, 
as the maps involved are not always uniform in their sign indices\footnote{Once 
separated from the companion automorphism part, the 
monomial parts do indeed depend on the chosen sign.}.
The coincident image of $Y^{(21)}_1,\ldots, Y^{(21)}_{16}$ 
by the two sides are  sign coherent monomials in the initial Y-variables 
$\bY^{(1)} = (Y_1,\ldots, Y_{16})$. 
Their explicit form is available in (\ref{y21}).

A natural question is whether there are monomial solutions to the tetrahedron equation
with the signs homogeneously chosen as $(\varepsilon_1,\varepsilon_2,\varepsilon_3,\varepsilon_4)$:
\begin{equation}\label{hteq}
\begin{split}
&\tau_{456|\, \varepsilon_1,\varepsilon_2,\varepsilon_3,\varepsilon_4} 
\tau_{236|\, \varepsilon_1,\varepsilon_2,\varepsilon_3,\varepsilon_4} 
\tau_{135|\, \varepsilon_1,\varepsilon_2,\varepsilon_3,\varepsilon_4} 
\tau_{124|\, \varepsilon_1,\varepsilon_2,\varepsilon_3,\varepsilon_4} 
\\
&=
\tau'_{124|\, \varepsilon_1,\varepsilon_2,\varepsilon_3,\varepsilon_4} 
\tau'_{135|\, \varepsilon_1,\varepsilon_2,\varepsilon_3,\varepsilon_4} 
\tau'_{236|\, \varepsilon_1,\varepsilon_2,\varepsilon_3,\varepsilon_4} 
\tau'_{456|\, \varepsilon_1,\varepsilon_2,\varepsilon_3,\varepsilon_4}.
\end{split}
\end{equation}
The answer is given by a direct calculation as follows.
\begin{proposition} \label{pr:mono}
The monomial part 
satisfies the tetrahedron equation \eqref{hteq} 
if and only if 
$\varepsilon_2=-\varepsilon_3 = -$, i.e., 
$(\varepsilon_1,\varepsilon_2,\varepsilon_3,\varepsilon_4) \in 
\{(+,-,+,+), (+,-,+,-), (-,-,+,+), (-,-,+,-)\}$.
\end{proposition}
Examples \ref{ex:1} and \ref{ex:15} describe the monomial parts  
$\tau_{--++}$ and $\tau_{+-+-}$ explicitly.
Analogous information is supplied for the remaining two cases in Appendix \ref{ap:sup}.

\subsection{Dilogarithm identities}\label{ss:did}

Now we turn to the dilogarithm identities that will be utilized later. 
Substitute (\ref{r124a}) and (\ref{r456b}) into the 
LHS and the RHS of (\ref{TE1}) respectively.
The result takes the form
\begin{equation}\label{psiu}
\begin{split}
&\mathrm{Ad}\left(
\Psi_q(U_1)^{\varepsilon_1} \cdots \Psi_q(U_4)^{\varepsilon_4}\right)
\tau_{456| \varepsilon_1,  \varepsilon_2, \varepsilon_3, \varepsilon_4}
\mathrm{Ad}\left(
\Psi_q(U_5)^{\varepsilon_1} \cdots \Psi_q(U_8)^{\varepsilon_4}\right)
\tau_{236| \varepsilon_1,  \varepsilon_2, \varepsilon_3, \varepsilon_4}
\\
&\times 
\mathrm{Ad}\left(
\Psi_q(U_9)^{\varepsilon_1} \cdots \Psi_q(U_{12})^{\varepsilon_4}\right)
\tau_{135 | \varepsilon_1,  \varepsilon_2, \varepsilon_3, \varepsilon_4}
\mathrm{Ad}\left(
\Psi_q(U_{13})^{\varepsilon_1} \cdots \Psi_q(U_{16})^{\varepsilon_4}\right)
\tau_{124| \varepsilon_1,  \varepsilon_2, \varepsilon_3, \varepsilon_4}
\\
&= 
\mathrm{Ad}\left(
\Psi_q(U'_1)^{\varepsilon_1} \cdots \Psi_q(U'_4)^{\varepsilon_4}\right)
\tau'_{124| \varepsilon_1,  \varepsilon_2, \varepsilon_3, \varepsilon_4}
\mathrm{Ad}\left(
\Psi_q(U'_5)^{\varepsilon_1} \cdots \Psi_q(U'_8)^{\varepsilon_4}\right)
\tau'_{135| \varepsilon_1,  \varepsilon_2, \varepsilon_3, \varepsilon_4}
\\
&\times 
\mathrm{Ad}\left(
\Psi_q(U'_9)^{\varepsilon_1} \cdots \Psi_q(U'_{12})^{\varepsilon_4}\right)
\tau'_{236| \varepsilon_1,  \varepsilon_2, \varepsilon_3, \varepsilon_4}
\mathrm{Ad}\left(
\Psi_q(U'_{13})^{\varepsilon_1} \cdots \Psi_q(U'_{16})^{\varepsilon_4}\right)
\tau'_{456| \varepsilon_1,  \varepsilon_2, \varepsilon_3, \varepsilon_4},
\end{split}
\end{equation}
where $U_t$ and $U'_t\, (t=1,\ldots, 16)$ denote the Y-variables depending
on  $(\varepsilon_1,  \varepsilon_2, \varepsilon_3, \varepsilon_4)$.
Pushing the monomial parts to the right brings (\ref{psiu}) into the form
\begin{equation}\label{zteq}
\begin{split}
&\mathrm{Ad}\bigl(\Psi_q(Z_1)^{\varepsilon_1}\cdots \Psi_q(Z_{16})^{\varepsilon_{4}}\bigr)
\tau_{456|\varepsilon_1,\varepsilon_2,\varepsilon_3,\varepsilon_4}
\tau_{236|\varepsilon_1,\varepsilon_2,\varepsilon_3,\varepsilon_4}
\tau_{135|\varepsilon_1,\varepsilon_2,\varepsilon_3,\varepsilon_4}
\tau_{124|\varepsilon_1,\varepsilon_2,\varepsilon_3,\varepsilon_4}
\\
&=
\mathrm{Ad}\bigl(\Psi_q(Z'_1)^{\varepsilon_1}\cdots \Psi_q(Z'_{16})^{\varepsilon_{4}}\bigr)
\tau'_{124|\varepsilon_1,\varepsilon_2,\varepsilon_3,\varepsilon_4}
\tau'_{135|\varepsilon_1,\varepsilon_2,\varepsilon_3,\varepsilon_4}
\tau'_{236|\varepsilon_1,\varepsilon_2,\varepsilon_3,\varepsilon_4}
\tau'_{456|\varepsilon_1,\varepsilon_2,\varepsilon_3,\varepsilon_4},
\end{split}
\end{equation}
where $Z_i$ and $Z'_i$ are monomials of $Y_1,\ldots, Y_{16}$ 
determined by 
\begin{align}\label{zu}
Z_i = \begin{cases} U_i & (i = 1,\ldots, 4),
\\
 \tau_{456|\varepsilon_1,\varepsilon_2,\varepsilon_3,\varepsilon_4}
 (U_i)  & (i=5,\ldots, 8),
\\
\tau_{456|\varepsilon_1,\varepsilon_2,\varepsilon_3,\varepsilon_4}
\tau_{236|\varepsilon_1,\varepsilon_2,\varepsilon_3,\varepsilon_4}
(U_i) & (i=9, \ldots, 12),
\\
\tau_{456|\varepsilon_1,\varepsilon_2,\varepsilon_3,\varepsilon_4}
\tau_{236|\varepsilon_1,\varepsilon_2,\varepsilon_3,\varepsilon_4}
\tau_{135|\varepsilon_1,\varepsilon_2,\varepsilon_3,\varepsilon_4}
(U_i) & (i=13, \ldots, 16).
\end{cases}
\end{align}
The elements $Z'_i$ are similarly determined from the 
RHS of (\ref{psiu}).
From (\ref{zteq}) and Proposition \ref{pr:mono} we deduce 
\begin{align}\label{adeq}
\mathrm{Ad}\bigl(\Psi_q(Z_1)^{\varepsilon_1}\cdots \Psi_q(Z_{16})^{\varepsilon_{4}}\bigr)
=\mathrm{Ad}\bigl(\Psi_q(Z'_1)^{\varepsilon_1}\cdots \Psi_q(Z'_{16})^{\varepsilon_{4}}\bigr)
\end{align}
for $(\varepsilon_1,\varepsilon_2,\varepsilon_3,\varepsilon_4) \in 
\{(+,-,+,+), (+,-,+,-), (-,-,+,+), (-,-,+,-)\}$.
Actually a stronger equality holds.

\begin{theorem}\label{th:dilog}
For $(\varepsilon_1,\varepsilon_2,\varepsilon_3,\varepsilon_4) \in  
\{(+,-,+,-), (-,-,+,+)\}$, 
the products of quantum dilogarithms within $\mathrm{Ad}$ in  \eqref{adeq}
are well-defined formal Laurent series in 
the eight Y-variables 
$Y_4$, $Y_5$, $Y_8$, $Y_9$, $Y_{10}$, $Y_{11}$, $Y_{14}$ and  $Y_{15}$.
Moreover they are equal, i.e., 
\begin{align}\label{pseq}
\Psi_q(Z_1)^{\varepsilon_1}\cdots \Psi_q(Z_{16})^{\varepsilon_{4}}
=
\Psi_q(Z'_1)^{\varepsilon_1}\cdots \Psi_q(Z'_{16})^{\varepsilon_{4}}.
\end{align}
\end{theorem}

\begin{proof}
We show the claim for $(\varepsilon_1,\varepsilon_2,\varepsilon_3,\varepsilon_4) =(-,-,+,+)$.
The case $(+,-,+,-)$ is similar.
The data  
${\scriptsize\begin{pmatrix}
Z_1& \ldots &  Z_8 \\ Z_9 & \ldots & Z_{16} 
\end{pmatrix}}
$ for $(\varepsilon_1,\varepsilon_2,\varepsilon_3,\varepsilon_4)=(-,-,+,+)$ 
is given as follows:
\begin{align}\label{zdL}
{\scriptsize
\begin{pmatrix}
Y_{15}^{-1} & Y^{-1}_{11}Y^{-1}_{15} & Y_{10} &   qY_{10}Y_{11} & q^{-1}Y^{-1}_9Y^{-1}_{10}
& Y^{-1}_5Y^{-1}_9Y^{-1}_{10} & Y_4 & qY_4Y_5 
\\
qY^{-1}_{14}Y^{-1}_{15} & q^2Y^{-1}_{11}Y^{-1}_{14}Y^{-1}_{15}  
& Y_8Y_9Y_{10} & qY_8Y_9 Y_{10} Y_{11}  & Y^{-1}_{14} & qY^{-1}_5Y^{-1}_9 Y_{11} Y^{-1}_{14}
& Y^{-1}_5Y_8Y_{11} & qY_8Y_9
\end{pmatrix}.
}
\end{align}
Similarly 
${\scriptsize\begin{pmatrix}
Z'_1& \ldots &  Z'_8 \\ Z'_9 & \ldots & Z'_{16} 
\end{pmatrix}}
$ 
is given as follows:
\begin{align}\label{zdR}
{\scriptsize
\begin{pmatrix}
Y^{-1}_{14} & qY^{-1}_9Y^{-1}_{14} & Y_8 & qY_8Y_9 & Y^{-1}_9 Y^{-1}_{10}Y^{-1}_{14}
& qY^{-1}_5Y^{-1}_9Y^{-1}_{10}Y^{-1}_{14} & Y_4Y_8Y_9 & qY_4Y_5Y_8Y_9
\\
Y^{-1}_{15} &  qY^{-1}_{11}Y^{-1}_{15} & Y_{10} & qY_{10}Y_{11} & q^{-1}Y^{-1}_9Y^{-1}_{10}
& Y^{-1}_5Y^{-1}_9 Y^{-1}_{10} & Y_4 & qY_4Y_5
\end{pmatrix}.
}
\end{align}
Note that $Z_{15}$ in (\ref{zdL}) is not sign coherent.
In order to show the well-definedness of the LHS of (\ref{pseq}), 
expand the 16 $\Psi_q$'s via (\ref{expa}) with the summation 
variables $n_1, \ldots,n_{16}\in \Z_{\ge 0}$.
By using the $q$-commutativity of Y-variables, one can align each term of the expansion uniquely as
\begin{equation}\label{c16}
\begin{split}
&\Psi_q(Z_1)^{-1}\Psi_q(Z_2)^{-1}\Psi_q(Z_3)\Psi_q(Z_4) 
\cdots \Psi_q(Z_{13})^{-1}\Psi_q(Z_{14})^{-1}\Psi_q(Z_{15})\Psi_q(Z_{16})
 \\
 &= \sum_{{\bf n} \in (\Z_{\ge 0})^{16}}C({\bf n}) 
 Y_4^{p_1} Y_5^{p_2} Y_8^{p_3} Y_9^{p_4} Y_{10}^{p_5} Y_{11}^{p_6} Y_{14}^{p_7} Y_{15}^{p_8},
 \end{split}
 \end{equation}
where $C({\bf n})$ is a rational function of $q$ depending on ${\bf n}=(n_1, \ldots, n_{16})$.
The powers $p_i$'s are given by 
\begin{alignat}{2}
p_1 &= n_7+n_8, &  p_5 &= n_3+n_4-n_5-n_6+n_{11}+n_{12},
\nonumber\\
p_2 &= -n_6+n_8-n_{14}-n_{15}, & p_6 &= -n_2+n_4-n_{10}+n_{12}+n_{14}+n_{15},
\nonumber\\
p_3 &= n_{11} + n_{12} + n_{15} + n_{16}, & p_7 &= -n_9 - n_{10} - n_{13}-n_{14},
\nonumber\\
p_4 &= -n_5-n_6+n_{11}+n_{12}-n_{14}+n_{16}, \quad & p_8 &= -n_1-n_2-n_9-n_{10}.
\label{pne}
\end{alignat}
The series (\ref{c16}) is well-defined if the coefficient of the monomial 
$Y_4^{p_1} Y_5^{p_2} Y_8^{p_3} Y_9^{p_4} Y_{10}^{p_5} Y_{11}^{p_6} Y_{14}^{p_7} Y_{15}^{p_8}$
for any given $(p_1,\ldots, p_8) \in \Z^8$ is finite.
This is shown by checking that there are none or finitely many 
${\bf n} \in (\Z_{\ge 0})^{16}$ 
satisfying the eight equations (\ref{pne})\footnote{This simple argument does not hold for the other two cases 
$(\varepsilon_1,\varepsilon_2,\varepsilon_3,\varepsilon_4)=(+,-,+,+), (-,-,+,-)$ 
in Proposition \ref{pr:mono}.
A more detailed examination is required for them, which we skip in this paper.}. 
This is straightforward.
The well-definedness of the RHS is verified in the same manner.

Next we prove (\ref{pseq}).
From (\ref{adeq}) we know that $\mathrm{Ad}\bigl((\text{LHS})(\text{RHS})^{-1}\bigr)=1$.
From an argument similar to the proof of \cite[Th.3.5]{KN11}, 
it follows that $(\text{LHS})(\text{RHS})^{-1}= c$, where $c$ only depends on $q$. 
To determine $c$, we compare the constant terms contained in the LHS and the RHS.
For the LHS, one looks for ${\bf n} \in (\Z_{\ge 0})^{16}$ such that 
$p_1=\cdots = p_8=0$.
It is easy to see that ${\bf n}=(0,\ldots, 0)$ 
is the only solution indicating that the constant term of the LHS is 1. 
Similarly, the constant term of the RHS is found to be 1. Therefore $c=1$.
\end{proof}

\section{Realization in terms of $q$-Weyl algebra}\label{s:qw}

\subsection{Y-variables and $q$-Weyl algebra}

Hereafter we set $q=\e^\hbar$.
By $q$-Weyl algebra we mean an associative algebra generated by 
$U^{\pm 1}$ and $W^{\pm 1}$ obeying the relations
$U U^{-1} = U^{-1} U= W W^{-1} = W^{-1}W = 1$ and 
$UW = q^2 WU$.
To each crossing $i$ of the wiring diagram we introduce a complex parameter 
$\lambda_i$ and the canonical variables $u_i, w_i$ 
satisfying 
\begin{align}\label{uwh}
[u_i, w_j]= 2\hbar \delta_{ij}, \quad [u_i, u_j]=[w_i, w_j]=0.
\end{align}
Thus their exponentials obey the relations in a direct product of the $q$-Weyl algebra, e.g., 
$\e^{u_i} \e^{w_j} = q^{2\delta_{ij}}\e^{w_j}\e^{u_i}$.
Given a wiring diagram and the associated square quiver, we ``parameterize" the 
Y-variables by the graphical rule explained in Figure \ref{fig:para}.

\begin{figure}[H]
\begin{tikzpicture}
{\color{red}
\fill (2,1.0) circle(2pt) coordinate(crossing) node[above]{$i$};
}
\draw [-] (0,2) to [out = 0, in = 135] (crossing)[thick] ;
\draw [->] (crossing) to [out = -45, in = 180] (4,0)[thick] ;
\draw [-] (0,0) to [out = 0, in = -135] (crossing)[thick] ;
\draw [->] (crossing) to [out = 45, in = 180] (4,2)[thick] ;
\node (q1) at (1.05,1.75) {\color{blue}$\circ$};
\node (q2) at (2.95,1.75) {\color{blue}$\circ$};
\node (q3) at (2.95,0.25) {\color{blue}$\circ$};
\node (q4) at (1.05,0.25) {\color{blue}$\circ$};
{\color{blue}
\draw (3,1.9) circle(0pt) node[above]{$\e^{u_i+\lambda_i}$};
\draw (3,0.05) circle(0pt) node[below]{$\e^{w_i}$};
\draw (1.1,0.1) circle(0pt) node[below]{$\e^{-u_i}$};
\draw (1.3,1.9) circle(0pt) node[above]{$\e^{-w_i-\lambda_i}$};
\draw [->] (node cs:name=q1) --(node cs:name=q2)[thick] node[above]{};
\draw [->] (node cs:name=q2) --(node cs:name=q3) [thick] node[below] {};
\draw [->] (node cs:name=q3) --(node cs:name=q4) [thick] node[below] {};
\draw [->] (node cs:name=q4) --(node cs:name=q1)[thick]  node[above]{};
}
\end{tikzpicture}
\caption{Graphical rule to parametrize the Y-variables in terms of $q$-Weyl algebra generators.
A Y-variable situated at a vertex (blue circle) of the square quiver acquires factors 
$\e^{-w_i-\lambda_i}, \e^{u_i+\lambda_i}, \e^{-u_i}$ and $\e^{w_i}$ from the neighboring 
crossing $i$ (colored in red) of the wiring diagram 
 if the vertex is located in 
the NW, NE, SW and SE of $i$, respectively.
The ordering of these factors from different $i$'s (if any) is inconsequential 
due to the commutativity of the associated canonical variables.  }
\label{fig:para}
\end{figure}
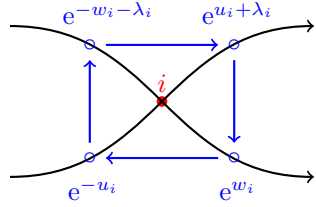

\noindent
From Figure \ref{fig:para} it is obvious that all the relations 
$Y_r Y_s = q^{2b_{rs}}Y_sY_r$ (\ref{eq:q-Y}) are satisfied under this parameterization.

To state the claim formally, 
let $\mathcal{W}_n$ be the direct product of $q$-Weyl algebras generated by 
$\e^{\pm u_i}, \e^{\pm w_i}$ for $i=1,\ldots, n$.
Let further $\mathcal{A}_n$ be the non-commuting fractional field of $\mathcal{W}_n$.
Then for $B$ corresponding to the left diagram in Figure \ref{fig:R123}, 
we have a morphism $\phi: \mathcal{Y}(B) \rightarrow  \mathcal{A}_3$, 
formally presented in the form $Y_r = \phi(Y_r)$, as follows:
\begin{align}\label{Yw}
\begin{matrix}
\vspace{0.2cm}
Y_1 = \e^{-w_2-\lambda_2},
&
Y_4 = \e^{u_1-u_2+\lambda_1},
&
Y_7 = \e^{-u_1}
\\
\vspace{0.2cm}
\!\!Y_2 = \e^{u_2+\lambda_2},
&\;
Y_5 = \e^{w_2-w_3-\lambda_3},\;\;
& \;
Y_8 = \e^{w_1-u_3},
\\
Y_3 = \e^{-w_1-\lambda_1},\;\;
&
Y_6 = \e^{u_3+\lambda_3},
&
Y_9 = \e^{w_3}.
\end{matrix}
\end{align}
Similarly for the right diagram of Figure \ref{fig:R123}, we have
\begin{align}\label{Ypw}
\begin{matrix}
\vspace{0.2cm}
Y'_1 = \e^{-w_3-\lambda_3}, \;\;
&
Y'_4= \e^{u_2-u_1+\lambda_2},
& 
Y'_7 = \e^{-u_2}, 
\\ \vspace{0.2cm}
Y'_2 = \e^{u_1+\lambda_1}, \;\;
&
Y'_5 = \e^{w_3-w_2-\lambda_2},\;\;
&
Y'_8 = \e^{u_3-w_1-\lambda_1+\lambda_3},
\\
Y'_3 = \e^{-u_3}, 
&
Y'_6 = \e^{w_1},
&
Y'_9 = \e^{w_2}.
\end{matrix}
\end{align}
For simplicity this will also be denoted as $\phi: \mathcal{Y}(B') \rightarrow  \mathcal{A}_3$
using the same symbol $\phi$, where $B'$ denotes the exchange matrix 
corresponding to the right diagram of Figure \ref{fig:R123}.
Note that the centers in Remark \ref{re:c2} take the values 
$\e^{-\lambda_2-\lambda_3}, \e^{-\lambda_1+\lambda_3}, \e^{\lambda_1+\lambda_2}$
in this parameterization.

\subsection{Extracting $R_{123}$ from $\widehat{R}_{123}$}

Let us illustrate the action of the monomial part 
$\tau_{\varepsilon_1, \varepsilon_2, \varepsilon_3, \varepsilon_4}$ 
 (\ref{taue}) of $\widehat{R}_{123}$ on the canonical variables for the case
 $(\varepsilon_1, \varepsilon_2, \varepsilon_3, \varepsilon_4) = (-,-,+,+)$.
 From Example \ref{ex:1} and (\ref{Yw})--(\ref{Ypw}),  we find that 
$\tau_{--++}$ is translated into a transformation $\tau^{uw}_{--++}$ 
of the canonical variables\footnote{$\tau^{uw}_{--++}$ is naturally regarded also as 
a transformation in $\mathcal{W}_3$ via exponentials.} as follows:
\begin{align}
\tau^{uw}_{--++}:\; 
\begin{cases}
u_1 \mapsto u_1+w_2-w_3+\lambda_2-\lambda_3, 
\quad \;
w_1  \mapsto w_1+ w_2-w_3,
\\
u_2  \mapsto u_1+u_3-w_1, 
\qquad \qquad \qquad 
w_2  \mapsto w_3,
\\
u_3  \mapsto -u_1+u_2+w_1, 
\qquad \qquad \quad \;
w_3  \mapsto w_2+\lambda_2-\lambda_3.
\end{cases}
\label{uwL1}
\end{align}
In order to realize it as an adjoint action, we introduce the group $N_n$ 
generated by 
\begin{align}
\e^{\pm \tfrac{1}{2\hbar}u_iw_j}\, (i \neq j),\quad 
\e^{\pm \tfrac{1}{2\hbar}u_iu_j}, 
\e^{\pm \tfrac{1}{2\hbar}w_iw_j}\, (i<j) ,\quad 
\e^{ \tfrac{a}{\hbar}u_i},\e^{ \tfrac{a}{\hbar}w_i}\, (a \in \C),\; b \in \C^\times
\end{align}
with $i,j\in \{1,\ldots, n\}$.
The multiplication is defined by (generalized) BCH formula and (\ref{uwh}), 
which is well defined due to the grading by $\hbar^{-1}$.
Let $\mathfrak{S}_n$ be the symmetric group generated by the transpositions 
$\rho_{ij} ~(i,j \in \{1,\ldots, n\})$.
It acts on $N_n$ via the adjoint action, inducing permutations of the indices of the canonical variables.
Thus one can form the semidirect product $N_n \rtimes \mathfrak{S}_n$, and 
let it act on $\mathcal{W}_n$ by the adjoint action.

Now the monomial part $\tau^{uw}_{--++}$ 
is described as the adjoint action as follows:
\begin{align}
\tau^{uw}_{--++} & = \mathrm{Ad}(P_{--++}),
\label{Pbch0}\\
P_{--++}&= 
\e^{\tfrac{1}{2\hbar}(u_1-w_1)(w_2-w_3)}
\rho_{23}\e^{\tfrac{\lambda_2-\lambda_3}{2\hbar}(u_3-w_1)}
\in N_3 \rtimes \mathfrak{S}_3,
\label{Pbch1}
\end{align}
where $\rho_{23}\e^{\tfrac{\lambda_2-\lambda_3}{2\hbar}(u_3-w_1)} = 
\e^{\tfrac{\lambda_2-\lambda_3}{2\hbar}(u_2-w_1)}\rho_{23}$ 
holds, since the parameters $\lambda_2$ and $\lambda_3$ are not exchanged.
Extending the indices and suppressing the sign choice of $--++$ in (\ref{Pbch0}) and (\ref{Pbch1}), we introduce
\begin{align}
\tau^{uw}_{ijk} &= \mathrm{Ad}(P_{ijk}),
\label{tijk}\\
P_{ijk} &= P(\lambda_i, \lambda_j, \lambda_k)_{ijk}
= \e^{\tfrac{1}{2\hbar}(u_i-w_i)(w_j-w_k)}\rho_{jk}
\e^{\tfrac{\lambda_j-\lambda_k}{2\hbar}(u_k-w_i)} \in N_6 \rtimes \mathfrak{S}_6,
\label{pijk} 
\end{align}
By a lengthy yet straightforward calculation using the BCH formula, one can prove
\begin{lemma}\label{le:tep}
$P_{ijk}$ satisfies the tetrahedron equation in $N_6 \rtimes \mathfrak{S}_6$ by itself:
\begin{align}\label{pte}
P_{456}P_{236}P_{135}P_{124} = 
P_{124}P_{135}P_{236}P_{456}.
\end{align}
\end{lemma}

Define $R_{123}= R(\lambda_1,\lambda_2,\lambda_3)_{123}$ by
\begin{align}
R_{123} 
&= 
\Psi_q(\e^{u_3-w_1})^{-1}
\Psi_q(\e^{u_3-w_1-w_2+w_3+\lambda_3})^{-1}
\Psi_q(\e^{u_1-u_2+\lambda_1})
\Psi_q(\e^{u_1-u_2+w_2-w_3+\lambda_1-\lambda_3})
P_{123} 
\label{RL0}
\\
&= \Psi_q(\e^{w_3-w_2+\lambda_3})^{-1}
\Psi_q(\e^{u_3-w_1})^{-1}
P_{123} \Psi_q(\e^{w_1-u_3+\lambda_1-\lambda_3})
\Psi_q(\e^{w_2-w_3+\lambda_2}),
\label{RL1}
\end{align}
where the equality of the two expressions can be checked easily 
from (\ref{uwL1}), (\ref{Pbch0}) and the pentagon identity (\ref{penta}).
Let $\widehat{R}^{uw}_{123}$ be the cluster transformation $\widehat{R}_{123}$ 
viewed as the one for the canonical variables.
Then from (\ref{dode0}),  (\ref{Rf1}), (\ref{Yw}) and (\ref{Ypw}),  we have 
\begin{align}
\widehat{R}^{uw}_{123} &= \mathrm{Ad}\bigl(R(\lambda_1,\lambda_2,\lambda_3)_{123}\bigr).
\end{align}
Formally these results may be stated as the commutativity of the diagrams:
\begin{align}\label{Rcom}
\begin{CD}
\mathcal{Y}(B') @> {\tau_{--++}}>> \mathcal{Y}(B) \\
@V{\phi}VV @ VV{\phi}V \\
\mathcal{A}_3 @ >{\tau^{uw}_{--++}}>> \mathcal{A}_3
\end{CD}
\qquad\qquad 
\begin{CD}
\mathcal{Y}(B') @> {\widehat{R}_{123}}>> \mathcal{Y}(B) \\
@V{\phi}VV @ VV{\phi}V \\
\mathcal{A}_3 @ >{\widehat{R}^{uw}_{123}}>> \mathcal{A}_3
\end{CD}
\end{align}

\begin{remark}\label{re:tuw}
For any sign $\varepsilon_1, \varepsilon_2, \varepsilon_3, \varepsilon_4$, there exists 
$\tau^{uw}_{\varepsilon_1, \varepsilon_2, \varepsilon_3, \varepsilon_4}$ which sends 
$u_i, w_i$ into a linear combinations of them and $\lambda_i$' s as in (\ref{uwL1}), 
and makes the following diagram commutative:
\begin{align}\label{Tcom}
\begin{CD}
\mathcal{Y}(B') @> {\tau_{\varepsilon_1, \varepsilon_2, \varepsilon_3, \varepsilon_4}}>> \mathcal{Y}(B) \\
@V{\phi}VV @ VV{\phi}V \\
\mathcal{A}_3 @ >{\tau^{uw}_{\varepsilon_1, \varepsilon_2, \varepsilon_3, \varepsilon_4}}>> \mathcal{A}_3
\end{CD}
\end{align}
Concrete forms of $\tau^{uw}_{\varepsilon_1, \varepsilon_2, \varepsilon_3, \varepsilon_4}$ are 
available in Appendix \ref{app:ruw} for the cases in Table \ref{tab:adt}.
\end{remark}

Extending (\ref{RL0}) and (\ref{RL1}), we introduce 
$R_{ijk}=R(\lambda_i, \lambda_j, \lambda_k)_{ijk}$ 
by
\begin{align}
R_{ijk} 
&= 
\Psi_q(\e^{u_k-w_i})^{-1}
\Psi_q(\e^{u_k-w_i-w_j+w_k+\lambda_k})^{-1}
\Psi_q(\e^{u_i-u_j+\lambda_i})
\Psi_q(\e^{u_i-u_j+w_j-w_k+\lambda_i-\lambda_k})
P_{ijk} 
\label{rijk0}\\
& = \Psi_q(\e^{w_k-w_j+\lambda_k})^{-1}
\Psi_q(\e^{u_k-w_i})^{-1}
P_{ijk}\Psi_q(\e^{w_i-u_k+\lambda_i-\lambda_k})
\Psi_q(\e^{w_j-w_k+\lambda_j}),
\label{rijk}
\end{align}
where $P_{ijk}$ is given by (\ref{pijk}).
Now we state the main result of the paper.
 
\begin{theorem}\label{th:main}
$R(\lambda_i, \lambda_j, \lambda_k)_{ijk} $ satisfies the tetrahedron equation:
\begin{equation}\label{tela}
\begin{split}
&\; R(\lambda_4, \lambda_5, \lambda_6)_{456}
R(\lambda_2, \lambda_3, \lambda_6)_{236}
R(\lambda_1, \lambda_3, \lambda_5)_{135}
R(\lambda_1, \lambda_2, \lambda_4)_{124}
\\
&\;= R(\lambda_1, \lambda_2, \lambda_4)_{124}
R(\lambda_1, \lambda_3, \lambda_5)_{135}
R(\lambda_2, \lambda_3, \lambda_6)_{236}
R(\lambda_4, \lambda_5, \lambda_6)_{456}.
\end{split}
\end{equation}
\end{theorem}
\begin{proof}
Consider the dilogarithm identity (\ref{pseq}) with
$(\varepsilon_1,\varepsilon_2,\varepsilon_3,\varepsilon_4)
=(-, -, +, +)$ in terms of canonical variables\footnote{The signs 
$\varepsilon_1,\varepsilon_2,\varepsilon_3,\varepsilon_4$ are 
actually $-, -, +, +$, but for clarity, they are left as they are.} 
\begin{align}\label{zt16}
\Psi_q(\tilde{Z}_1)^{\varepsilon_1}\cdots \Psi_q({\tilde Z}_{16})^{\varepsilon_4}
=
\Psi_q(\tilde{Z'}_1)^{\varepsilon_1}\cdots \Psi_q(\tilde{Z'}_{16})^{\varepsilon_4}.
\end{align}
Here $\tilde{Z}_i = \phi(Z_i)$ and $\tilde{Z'}_i = \phi(Z'_i)$ are given by 
(\ref{zdL}) and (\ref{zdR}) by further substituting the parameterization of $Y_i$'s in terms of 
the $q$-Weyl algebra generators $\e^{\pm u_j}, \e^{\pm w_j},\, (j=1,\ldots, 6)$
according to the rule in Figure \ref{fig:para} applied to the bottom diagram 
in Figure \ref{fig:rikisaku}\footnote{The map $\phi$ does not spoil the 
well-definedness of the expansion like (\ref{c16}) with 
respect to $\e^{u_i}$ and $\e^{w_i}$ since it preserves the rank of the quantum torus generated
by $Y_i (i=4,5,8,9,10,11,14,15)$.}.
Multiplication of (\ref{pte}) to (\ref{zt16}) from the right leads to
\begin{align}\label{zp16}
\Psi_q(\tilde{Z}_1)^{\varepsilon_1}\cdots \Psi_q({\tilde Z}_{16})^{\varepsilon_4}
P_{456}P_{236}P_{135}P_{124}
=
\Psi_q(\tilde{Z'}_1)^{\varepsilon_1}\cdots \Psi_q(\tilde{Z'}_{16})^{\varepsilon_4}
P_{124}P_{135}P_{236}P_{456}.
\end{align}
Let us consider the LHS. 
It is obviously equal to
\begin{equation}\label{tochu1}
\begin{split}
&\Psi_q(\tilde{Z}_1)^{\varepsilon_1}\cdots \Psi_q({\tilde Z}_{4})^{\varepsilon_4}
P_{456}
\\
&\times 
P^{-1}_{456}
\Psi_q(\tilde{Z}_5)^{\varepsilon_1}\cdots \Psi_q({\tilde Z}_{8})^{\varepsilon_4}
P_{456} \cdot P_{236}
\\
& \times 
P^{-1}_{236} P^{-1}_{456}
\Psi_q(\tilde{Z}_9)^{\varepsilon_1}\cdots \Psi_q({\tilde Z}_{12})^{\varepsilon_4}
P_{456} P_{236} \cdot P_{135}
\\
& \times 
P^{-1}_{135} P^{-1}_{236} P^{-1}_{456}
\Psi_q(\tilde{Z}_{13})^{\varepsilon_1}\cdots \Psi_q({\tilde Z}_{16})^{\varepsilon_4}
P_{456} P_{236} P_{135} \cdot P_{124}.
\end{split}
\end{equation}
On the other hand from (\ref{tijk}) and the image of $(\ref{zu})$ by $\phi$, we know
\begin{align}
{\tilde U}_i = \begin{cases}
\tilde{Z}_i & (i=1,\ldots, 4),
\\
P_{456}^{-1} \tilde{Z}_i P_{456} & (i=5,\ldots, 8),
\\
P_{236}^{-1}P_{456}^{-1} \tilde{Z}_i P_{456}P_{236} & (i=9,\ldots, 12),
\\
P_{135}^{-1}P_{236}^{-1}P_{456}^{-1} \tilde{Z}_i 
P_{456}P_{236} P_{135} & (i=13,\ldots, 16).
\end{cases}
\end{align}
where $\tilde{U}_i = \phi(U_i)$.
Thus (\ref{tochu1}) is cast into the form
\begin{equation}
\begin{split}
&\Psi_q(\tilde{U}_1)^{\varepsilon_1}\cdots \Psi_q({\tilde U}_{4})^{\varepsilon_4}
P_{456}
\Psi_q(\tilde{U}_5)^{\varepsilon_1}\cdots \Psi_q({\tilde U}_{8})^{\varepsilon_4}
P_{236}
\\
&\times 
\Psi_q(\tilde{U}_9)^{\varepsilon_1}\cdots \Psi_q({\tilde U}_{12})^{\varepsilon_4}
P_{135}
\Psi_q(\tilde{U}_{13})^{\varepsilon_1}\cdots \Psi_q({\tilde U}_{16})^{\varepsilon_4}
P_{124}.
\end{split}
\end{equation}
This is identified with the LHS of (\ref{tela}) for (\ref{rijk0}).
The RHS of (\ref{tela}) is similarly derived from that in (\ref{zp16}).
\end{proof}

By making slight adjustments to the conventions, 
the formulas  (\ref{RL1}) and  (\ref{Pbch1}) 
precisely reproduce the solution to the tetrahedron equation first established in 
\cite[eq.(1.37)]{S99}.
The combinations of quantum Y-variables (\ref{aldef}), when expressed in terms of 
the $q$-Weyl algebra generators, are essentially 
$\Lambda_1, \Lambda_2$ and $\Lambda_3$ in \cite[eq.(1.17)]{S99}.
The fact that the origin of this remarkable solution 
is finally clarified by the quantum cluster algebra theory 
nearly a quarter of a century after its discovery is of notable interest and significance.
This is especially so since it may suggest further results of a similar nature, offering
a unified perspective on the solutions known as 
the Zamolodchikov-Bazhanov-Baxter model \cite{Z81,B83,BB92} and 
those in \cite{KMY23}.
See the end of Section \ref{s:me} for a further comment.
As a related topic, we note an interpretation of a classical limit of the Lax operator in \cite{BS06} 
in terms of cluster algebra and perfect networks \cite{GSZ21}.

\subsection{Formulas for other signs}
The content in the previous subsection has been based on the sign chosen as
$(\varepsilon_1, \varepsilon_2, \varepsilon_3, \varepsilon_4)=(-,-,+,+)$. 
Let us argue the situation for other choices.
As remarked after (\ref{rhat}), the cluster transformation $\widehat{R}$ itself is 
independent of the signs.
Let us write the formulas in Appendix \ref{app:RY} symbolically as 
\begin{align}
\widehat{R} = \mathrm{Ad}\left(
\Psi_q(U_1)^{\varepsilon_1} 
\Psi_q(U_2)^{\varepsilon_2} 
\Psi_q(U_3)^{\varepsilon_3} 
\Psi_q(U_4)^{\varepsilon_4} \right)
\tau_{\varepsilon_1, \varepsilon_2, \varepsilon_3, \varepsilon_4},
\end{align}
where $\tau_{\varepsilon_1, \varepsilon_2, \varepsilon_3, \varepsilon_4}$ is defined in (\ref{taue}).
The arguments  $U_i$'s are 
$(\varepsilon_1, \varepsilon_2, \varepsilon_3, \varepsilon_4)$-dependent 
monomials in $Y_1,\ldots, Y_9$.
In terms of canonical variables, it reads as
 \begin{align}\label{rhuw}
\widehat{R}^{uw} = \mathrm{Ad}\left(
\Psi_q(\tilde{U}_1)^{\varepsilon_1} 
\Psi_q(\tilde{U}_2)^{\varepsilon_2} 
\Psi_q(\tilde{U}_3)^{\varepsilon_3} 
\Psi_q(\tilde{U}_4)^{\varepsilon_4} \right)
\tau^{uw}_{\varepsilon_1, \varepsilon_2, \varepsilon_3, \varepsilon_4},
\end{align}
where $\tilde{U}_i = \phi(U_i)$ with $\phi$ defined in (\ref{Yw}).

Our first question is 
when $\tau^{uw}_{\varepsilon_1, \varepsilon_2, \varepsilon_3, \varepsilon_4}$ here 
(see Remark \ref{re:tuw}) admits an adjoint action description similar to (\ref{pijk}).
We investigate it by slightly enlarging $\mathfrak{S}_3$ in (\ref{Pbch1}).
Let $\mathcal{S}_3$ be an automorphism group of $\mathcal{W}_3$ 
generated by permutations of the indices of canonical variables 
and the following transformations that preserve 
the canonical commutation relation (\ref{uwh})\footnote{It should be noted 
 that our exploration is not the most comprehensive, since $\mathcal{S}_3$  is a small subgroup of the infinite automorphism group of $\mathcal{W}_3$.}:
\begin{align}\label{pcr}
(u_i,w_i) \mapsto (u_i,w_i),\;  (w_i, -u_i), \; (-w_i, u_i), \; (-u_i,-w_i).
\end{align}
We represent an elements of $\mathcal{S}_3$ as 
$\bigl({u_1, u_2, u_3, w_1, w_2, w_3 \atop  -w_3, u_2, w_1, u_3, w_2, -u_1}\bigr)$ for example,
which indicates that the elements in the first row are transformed 
into the corresponding elements in the second row.
The group $\mathcal{S}_3$ contains  $\mathfrak{S}_3$ previously considered in (\ref{Pbch1}).
For example $\rho_{23} \in \mathfrak{S}_3$ is expressed as 
$\bigl({u_1, u_2, u_3, w_1, w_2, w_3 \atop  u_1, u_3, u_2, w_1, w_3, w_2}\bigr)$.
The order of $\mathcal{S}_3$ is $4^3\times 3!=384$. 
Now we seek the realization $\tau^{uw}_{\varepsilon_1, \varepsilon_2, \varepsilon_3, \varepsilon_4}$ 
as an adjoint action of the form
\begin{equation}\label{XL}
\begin{split}
&\tau^{uw}_{\varepsilon_1, \varepsilon_2, \varepsilon_3, \varepsilon_4} = 
\mathrm{Ad}(P_{\varepsilon_1, \varepsilon_2, \varepsilon_3, \varepsilon_4} ),
\quad 
P_{\varepsilon_1, \varepsilon_2, \varepsilon_3, \varepsilon_4}
=\e^{\tfrac{1}{2\hbar}X} \rho \, \e^{\tfrac{1}{2\hbar}\mathcal{L}}
\in N_3 \rtimes \mathcal{S}_3,
\\
&X = \sum_{i<j}(\alpha_{ij}u_iu_j+ \beta_{ij}w_iw_j)+\sum_{i\neq j}\gamma_{ij}u_iw_j,
\quad 
\mathcal{L} = \sum_i(\delta_i u_i + \delta'_i w_i), \quad 
\rho \in \mathcal{S}_3
\end{split}
\end{equation}
with $\alpha_{ij}, \beta_{ij}, \gamma_{ij}, \delta_i, \delta'_j \in \C$.
Through a direct calculation we find the set 
of signs $(\varepsilon_1, \varepsilon_2, \varepsilon_3, \varepsilon_4) $
which permits such a description.

\begin{proposition}\label{pr:tau}
The realization of $\tau^{uw}_{\varepsilon_1, \varepsilon_2, \varepsilon_3, \varepsilon_4}$
in the form \eqref{XL} 
is possible only for the 6 cases in Table \ref{tab:adt}.
\end{proposition}

\begin{table}[H]
\begin{tabular}{c|ccc}
  & $\rho$ & $X$ & $\mathcal{L}$ 
\\ \hline \\
$P_{----}$
& $\bigl({u_1, u_2, u_3, w_1, w_2, w_3 \atop  u_2, u_3, w_1, w_2, w_3, -u_1}\bigr)$
& $u_2 w_1 - u_2 w_3 + w_1 w_3$ 
& $\lambda_1(-u_3+w_1+w_2-w_3)+\lambda_2(u_3-w_1)$
\\ 
$P_{--++}$
& $\phantom{\Biggl(}\bigl({u_1, u_2, u_3, w_1, w_2, w_3 \atop  u_1, u_3, u_2, w_1, w_3, w_2}\bigr)$
& $(u_1 - w_1) (w_2 - w_3)$
& $(\lambda_2-\lambda_3)(u_3-w_1)$
\\ 
$P_{-+-+}$
& $\phantom{\Biggl(}\bigl({u_1, u_2, u_3, w_1, w_2, w_3 \atop -w_3, u_2, w_1, u_3, w_2, -u_1}\bigr)$
& $(w_1 - u_3) (u_2 + w_2)$
& $\begin{matrix}\lambda_1(-u_3+w_1+w_2-w_3) + \lambda_2(u_3-w_1) \\
+ \lambda_3(u_1-u_2-u_3+w_1)\end{matrix}$
\\ 
$P_{+-+-}$
& $\phantom{\Biggl(}\bigl({u_1, u_2, u_3, w_1, w_2, w_3 \atop u_2, u_1, u_3, w_2, w_1, w_3}\bigr)$
& $(u_1 - u_2) (w_3 - u_3)$
& $(\lambda_1-\lambda_2)(w_1-u_3)$
\\
$P_{++--}$
& $\phantom{\Biggl(}\bigl({u_1, u_2, u_3, w_1, w_2, w_3 \atop -w_3, u_2, w_1, u_3, w_2, -u_1}\bigr)$
& $(w_1 - u_3) (u_2 + w_2)$
& $\begin{matrix}\lambda_1(-u_3+w_1+w_2-w_3) + \lambda_2(u_3-w_1) \\
+ \lambda_3(u_1-u_2-u_3+w_1)\end{matrix}$
\\
$P_{++++}$
& $\bigl({u_1, u_2, u_3, w_1, w_2, w_3 \atop  -w_3, u_1, u_2, u_3, w_1, w_2}\bigr)$
& $-u_1 u_3 + u_1 w_2 - u_3 w_2$ 
& $\lambda_2(u_3-w_1)+\lambda_3(u_1-u_2-u_3+w_1)$
\end{tabular}

\caption{
The list of the monomial parts of the cluster transformation $\widehat{R}$ 
which can be realized by the adjoint action of the form (\ref{XL}).
The results on $P_{-+-+}$ and $P_{++--}$ are identical since 
$\tau_{-+-+} = \tau_{++--}$ as mentioned after Example \ref{ex:15}.
The result (\ref{Pbch1}) agrees with $\tau^{uw}_{--++}$ here.}
\label{tab:adt}
\end{table}

\noindent
Formulas of $R_{\ve_1,\ve_2, \ve_3, \ve_4}$ for the cases listed in Table \ref{tab:adt}
are available in Appendix \ref{app:ruw}.

\begin{remark}\label{re:PP}
$P_{+-+-}$ is obtained from $P_{--++}$ by the change
$(u_i,w_i,\lambda_i,\hbar) \rightarrow (w_{4-i}, u_{4-i},\lambda_{4-i}, -\hbar)$,  
which keeps the canonical commutation relation (\ref{uwh}).
Using this fact one can prove that $P_{+-+-}$ also satisfies the tetrahedron equation (\ref{pte}) by itself
by attributing the claim to Lemma \ref{le:tep}.
With the exception of $P_{--++}$ and $P_{+-+-}$, the remaining 
$P_{\varepsilon_1, \varepsilon_2, \varepsilon_3, \varepsilon_4}$'s in Table \ref{tab:adt} do not satisfy the 
tetrahedron equation (\ref{pte}) by themselves\footnote{This fact is not contradictory to the argument
made in regard to the second question below.}.
\end{remark}

From Proposition \ref{pr:tau} and (\ref{rhuw}), it follows that 
 \begin{align}
\widehat{R}^{uw} &= \mathrm{Ad}\left(R_{\varepsilon_1, \varepsilon_2, \varepsilon_3, \varepsilon_4}\right),
\label{psR}
\\
R_{\varepsilon_1, \varepsilon_2, \varepsilon_3, \varepsilon_4} & := 
\Psi_q(\tilde{U}_1)^{\varepsilon_1} 
\Psi_q(\tilde{U}_2)^{\varepsilon_2} 
\Psi_q(\tilde{U}_3)^{\varepsilon_3} 
\Psi_q(\tilde{U}_4)^{\varepsilon_4}
P_{\varepsilon_1, \varepsilon_2, \varepsilon_3, \varepsilon_4},
\label{Ree}
\end{align}
for the cases in Table \ref{tab:adt}.
Our second question concerns the dependence 
of $R_{\varepsilon_1, \varepsilon_2, \varepsilon_3, \varepsilon_4}$
on $\varepsilon_1, \varepsilon_2, \varepsilon_3, \varepsilon_4$,
which is non-trivial for the objects within $\Ad (\;\,)$. 
Leaving the most general setting aside, we investigate it here for a specific representation on $V^{\otimes 3}$ 
which is defined in subsection \ref{ss:idr}.
It is easy to see that the property $\mathrm{Ad}(\Xi) =1$ for $\Xi \in \mathrm{End}(V^{\otimes 3})$
implies $\Xi= \mathrm{const} \times \mathrm{id}$.   
From (\ref{psR}) it follows that 
$\mathrm{Ad}(R_{\varepsilon_1, \varepsilon_2, \varepsilon_3, \varepsilon_4}
R_{\varepsilon'_1, \varepsilon'_2, \varepsilon'_3, \varepsilon'_4}^{-1})=1$ 
for any pair of signs 
in Table \ref{tab:adt}.
Therefore 
$R_{\varepsilon_1, \varepsilon_2, \varepsilon_3, \varepsilon_4}$ 
is independent of $\varepsilon_1, \varepsilon_2, \varepsilon_3, \varepsilon_4$ 
up to normalization
as long as it makes sense as a formal Laurent series.

\section{Matrix elements}\label{s:me}

\subsection{An infinite dimensional representation}\label{ss:idr}
Let $V = \bigoplus_{m \in \Z} \C |m\rangle$ and 
$V^\ast = \bigoplus_{m \in \Z} \C \langle m|$ be the left and the right modules over 
the $q$-Weyl algebra such that 
\begin{align}
\e^u | m\rangle = |m-1\rangle, \quad  \e^w| m \rangle = q^{2m}| m\rangle, \quad 
\langle m | \e^u  =  \langle m+1|, \quad  \langle m | \e^w  = \langle  m| q^{2m}.
\end{align}
We set the dual pairing of $V^\ast$ and $V$ as $\langle m | m'\rangle = \delta_{m,m'}$, which satisfies 
$\langle m | (g|m'\rangle) = (\langle m | g) |m'\rangle$.
Bases of $V^{\otimes 3}$ and $V^{\ast \otimes 3}$ will be denoted by $|m_1,m_2, m_3\rangle$ and 
$\langle m_1, m_2, m_3 |$, respectively $(m_i \in \Z)$.
We regard $R(\lambda_1, \lambda_2, \lambda_3)_{123}$ in (\ref{RL1}) as an element of $\mathrm{End}(V^{\otimes 3})$ 
and calculate its matrix element
$R^{a,b,c}_{i,j,k} := \langle a,b,c| R(\lambda_1, \lambda_2, \lambda_3)_{123}|i,j,k\rangle$.
We set $\kappa_i = \e^{\lambda_i}$, and obtain 
\begin{align}
R^{a,b,c}_{i,j,k} &= \langle a,b,c| 
\Psi_q(\e^{w_3-w_2+\lambda_3})^{-1}
\Psi_q(\e^{u_3-w_1})^{-1}
P_{123} \Psi_q(\e^{w_1-u_3+\lambda_1-\lambda_3})
\Psi_q(\e^{w_2-w_3+\lambda_2})
|i,j,k\rangle
\nonumber\\
&=
\Psi_q(\kappa_3q^{2c-2b})^{-1}\Psi_q(\kappa_2q^{2j-2k})
\langle a,b,c| \Psi_q(q^{-2a}\e^{u_3})^{-1}P_{123} \Psi_q(\tfrac{\kappa_1}{\kappa_3}q^{2i}\e^{-u_3})|i,j,k\rangle
\nonumber \\
&= \Psi_q(\kappa_3q^{2c-2b})^{-1}\Psi_q(\kappa_2q^{2j-2k})
\sum_{m,n \ge 0}\frac{q^{n^2-2na}(-\tfrac{\kappa_1}{\kappa_3}q^{2i+1})^m}{(q^2)_n(q^2)_m}
\langle a,b,c+n | P_{123}|i,j,k+m\rangle,
\label{mahi}
\end{align}
where (\ref{expa}) has been used and $(q^2;q^2)_m$ is written as $(q^2)_m$ for short.
Note that $1/(q^2)_m=0$ for $m  \in \Z_{\le -1}$ by the definition.
In order to calculate the elements of $P_{123}$, we assume that 
$r := \frac{\lambda_2-\lambda_3}{2\hbar} \in \Z$.
Then from (\ref{Pbch1}) we have
\begin{align}
\langle a,b,c | P_{123}|i,j,k\rangle
&= 
\langle a,b,c |\e^{(u_1-w_1)(b-c)}
\rho_{23}\e^{r(u_3-w_1)}|i,j,k\rangle
\nonumber\\
&= 
\langle a,b,c |q^{(b-c)^2} \e^{(b-c)u_1}\e^{-(b-c)w_1} \rho_{23}q^{-2ri}|i,j,k-r\rangle
\nonumber\\
&= 
\langle a+b-c,b,c |i,k-r,j\rangle q^{(b-c)^2-2i(b-c)-2ri}
\nonumber\\
&
= \delta^{a+b}_{i+j}\delta^{b+r}_k \delta^c_j q^{(b-c)(b-c-2i)-2ri},
\end{align}
where $q=\e^\hbar$ has been used.
Applying this result to (\ref{mahi}) we obtain
\begin{align}
R^{a,b,c}_{i,j,k} &= 
\delta^{a+b}_{i+j} \left(\frac{\kappa_1}{\kappa_3}\right)^{b+r-k}q^\alpha
\frac{(-\kappa_3q^{2c-2b+1})_\infty(q^{2+2j-2c})_\infty}{(q^2)_{\infty}
(-\kappa_2q^{2j-2k+1})_\infty(\tfrac{\kappa_3}{\kappa_2}q^{2k-2b})_{b+r-k}},
\label{rqq}
\\
\alpha &= (c - k + r) (c + k - r - 2 b) + 2 i (c - k).
\end{align}
To get (\ref{rqq}), we have applied $(q^2)_m = (-1)^m q^{m(m+1)}(q^{-2m})_m$ for $m=b+r-k$.
One observes some similarity between it and (\ref{pppp}).

\subsection{Modular double $R$}\label{ss:md}

This subsection is an informal review of a part of the results on the $R$ matrix $R(\lambda_1, \lambda_2, \lambda_3)_{123}$ 
in the modular double setting due to \cite{S10}.
We use a parameter $\bb$ and set 
\begin{align}\label{qq}
\hbar = i\pi \bb^2,\quad q= \e^{i \pi \bb^2}, \quad \bar{q} = \e^{-i \pi \bb^{-2}},
\quad \eta = \frac{\bb+\bb^{-1}}{2}.
\end{align}
The non-compact quantum dilogarithm is defined by 
\begin{align}\label{ncq}
\Phi_\bb(z) = \exp\left(
\frac{1}{4}\int \frac{\e^{-2izw}}{\sinh(w \bb)\sinh(w/\bb)}\frac{dw}{w}\right)
= \frac{(\e^{2\pi(z+i\eta)\bb};q^2)_\infty}{(\e^{2\pi(z-i\eta)\bb^{-1}};\bar{q}^2)_\infty},
\end{align}
where the singularity at $w=0$ is circled from above.
The infinite product formula is valid in the so-called strong coupling regime $0<\eta < 1$.
It enjoys the symmetry 
$\Phi_\bb(z) = \Phi_{\bb^{-1}}(z)$, and  
satisfies the following relations:
\begin{align}
\Phi_\bb(z)\Phi_\bb(-z) &= \e^{i \pi z^2-i \pi (1-2\eta^2)/6},
\label{pinv}\\
\frac{\Phi_\bb(z-i \bb^{\pm 1}/2)}{\Phi_\bb(z+i \bb^{\pm 1}/2)}
&= 1+\e^{2\pi z \bb^{\pm1}}.
\label{prec}
\end{align}

Recall that  $u_j, w_j\,(j=1,2,3)$ are the canonical variables in the previous section obeying (\ref{uwh}).
Here we work with the new canonical variables $\hat{x}_j,  \hat{p}_j$ and parameters $\lambda_j$ obtained by rescaling 
the previous $u_j, w_j$ and $\lambda_j$ as
\begin{align}
u_k \rightarrow  2\pi \bb \hat{x}_k,\quad
w_k \rightarrow  2\pi \bb \hat{p}_k,
\quad [\hat{x}_j, \hat{p}_k] = \frac{i}{2\pi} \delta_{j,k},
\quad \lambda_j \rightarrow 2\pi \bb \lambda_j.
\end{align}
We write  $\lambda_{j,k}= \lambda_j-\lambda_k$ for short.
By the modular double $R$ (in the ``momentum" representation),
we mean $R(\lambda_1, \lambda_2, \lambda_3)_{123}$
in the representation on $L^2(\R^3)$ such that 
\begin{enumerate}
\item[(i)] $\hat{p}_j$ acts as a multiplication by $p_j$, 
\item[(ii)] $\hat{x}_j$ acts as $\tfrac{i}{2\pi}\tfrac{\partial}{\partial p_j}$,
\item[(iii)] symmetry under the exchange $\bb \leftrightarrow   \bb^{-1}$ is implemented.
\end{enumerate}
In view of 
\begin{align}
\frac{\Psi_q(\e^{2\pi \bb(z+ i\bb/2)})}
{\Psi_q(\e^{2\pi \bb(z- i\bb/2)})}
= \frac{\Phi_\bb(z-i \bb/2)}{\Phi_\bb(z+i \bb/2)},
\end{align}
the modular double $R$ is obtained formally by replacing $\Psi_q(\e^{2\pi \bb\hat{z}})$ by 
$\Phi_\bb(\hat{z})^{-1}$  in $R(\lambda_1, \lambda_2, \lambda_3)_{123}$.
We adopt the formula (\ref{RL1}) and (\ref{pijk}), which lead to 
\begin{align}
 \mathscr{R}_{123}
 &=
\Phi_\bb(\hat{p}_3-\hat{p}_2+\lambda_3)
\Phi_\bb(\hat{x}_3-\hat{p}_1) \,\mathscr{P}_{123}\,
\Phi_\bb(\hat{p}_1-\hat{x}_3+\lambda_{1,3})^{-1}
\Phi_\bb(\hat{p}_2-\hat{p}_3+\lambda_2)^{-1},
\label{RRf}
\\
\mathscr{P}_{123} &= \e^{2\pi i(\hat{x}_1-\hat{p}_1)(\hat{p}_3-\hat{p}_2)} \rho_{23}
\e^{2\pi i\lambda_{2,3}(\hat{p}_1-\hat{x}_3)}.
\label{PPf}
 \end{align}
 The modular double $\mathscr{R}_{123} = \mathscr{R}(\lambda_1,\lambda_2, \lambda_3)_{123}$ 
 is an integral operator
 acting on functions of $(p_1,p_2,p_3)$ from $L^2(\R^3)$
 satisfying the tetrahedron equation (\ref{tela}). 
 To describe its kernel,  we introduce the quantum mechanical notation
 in the momentum ($\hat{p}$-diagonal) representation $(p_j, p'_j  \in \R)$:
 \begin{align*}
 &\hat{p}_j  |p_j\rangle = p_j  |p_j\rangle,\quad \langle p_j|\hat{p}_j = \langle p_j|p_j,
 \quad \langle p_j | p'_j \rangle = \delta(p_j-p'_j),
 \\
& |p_1,p_2,p_3\rangle  = |p_1\rangle \otimes  |p_2\rangle \otimes  |p_3\rangle,
 \quad 
  \langle p_1,p_2,p_3 | = \langle p_1| \otimes  \langle p_2 | \otimes  \langle p_3|.
 \end{align*}
  We will also use the coordinate states 
 $\langle x| = \int dp \langle p| \e^{2\pi i xp}$ and 
 $|x\rangle = \int dp \,\e^{-2\pi i xp}|p\rangle$ with 
 $\langle x | x'\rangle = \delta(x-x')$,
where the integrals here and in what follows are always taken over the real line $\R$.
To simplify the notation, 
the difference between $|p\rangle$ and $|x\rangle$, for instance, 
should be inferred from the symbols themselves. 
Thus for example we have 
\begin{align}\label{qm}
\langle x | p \rangle = \langle p | x \rangle^\ast = \e^{2\pi i x p},\quad 
\langle p | \e^{2\pi i \alpha \hat{x}} = \langle p-\alpha|,
\quad \langle x | \e^{2\pi i \alpha \hat{p}} = \langle x+\alpha|.
\end{align}

The following result was originally presented in \cite{S10}. 
Interestingly, it takes the form of ``$\Phi_\bb$-analogue of the cross ratio".
Here we furnish our derivation.
 \begin{proposition}\label{pr:ker}
 The integral kernel of $\mathscr{R}_{123}$ is given as follows:
 \begin{align}
& \langle p_1, p_2, p_3 | \mathscr{R}_{123} |p'_1, p'_2, p'_3\rangle 
= \varrho_p \,\delta(p_1+p_2-p'_1-p'_2) \e^{2\pi i \phi_p}
\frac{\Phi_\bb(p_3-p_2+\lambda_3) \Phi_\bb(p'_3-p'_2-\lambda_2) }
{\Phi_\bb(p_3-p'_2-i \eta) \Phi_\bb(p'_3-p_2-\lambda_{2,3}-i\eta) },
\label{pppp}\\
&\phi_p = p_1(p'_3-p_3)+\lambda_2(p_2-p'_2)+\lambda_1(p'_3-p_2)+i\eta(p'_2-p_3),
\quad 
\varrho_p = \e^{\pi i (-\eta^2-\lambda_3^2 -2\lambda_1\lambda_2+2\lambda_1\lambda_3)}.
\label{ppff}
 \end{align}
 \end{proposition}
 
\begin{proof}
In what follows,  $p_j, p'_j, k_3$ will be used for 
labels of momentum states, while $x_3, x'_3$ correspond to coordinate states.
From (\ref{RRf}) we get
\begin{align}
&\langle p_1, p_2, p_3 | \mathscr{R}_{123} |p'_1,p'_2,p'_3\rangle 
=
\frac{\Phi_\bb(p_3-p_2+\lambda_3)}{\Phi_\bb(p'_2-p'_3+\lambda_2)}\mathscr{R}_0,
\label{haji}\\
&\mathscr{R}_0 = 
\langle p_1, p_2, p_3 |
\Phi_\bb(\hat{x}_3-\hat{p}_1) \,\mathscr{P}_{123}\,
\Phi_\bb(\hat{p}_1-\hat{x}_3+\lambda_{1,3})^{-1}
 |p'_1,p'_2,p'_3\rangle.
 \end{align}
By inserting $1 = \int |x_3\rangle \langle x_3|dx_3 = \int |x'_3\rangle \langle x'_3|dx'_3$ in the third component, 
$\mathscr{R}_0$ is rewritten as 
\begin{align*}
&\mathscr{R}_0 = \int  dx_3dx'_3 \e^{2 \pi i(x'_3p'_3-x_3p_3)}
\frac{\Phi_\bb(x_3-p_1)}{\Phi_\bb(p'_1-x'_3+\lambda_{1,3})}
\langle p_1, p_2, x_3 | \mathscr{P}_{123} |p'_1, p'_2, x'_3\rangle,
\\
&\langle p_1, p_2, x_3 | \mathscr{P}_{123} |p'_1, p'_2, x'_3\rangle
=  \e^{2 \pi i \lambda_{2,3}(p'_1-x'_3)}
\langle p_1, p_2, x_3 | \e^{2\pi i(\hat{x}_1-\hat{p}_1)(\hat{p}_3-\hat{p}_2)}  |p'_1, x'_3, p'_2\rangle
\\
&
= \e^{2 \pi i \lambda_{2,3}(p'_1-x'_3)}\int dk_3 \,\e^{2\pi i x_3k_3}\langle p_1, p_2, k_3 | 
\e^{2\pi i(\hat{x}_1-\hat{p}_1)(k_3-p_2)}  |p'_1, x'_3, p'_2\rangle
\\
&= \e^{2 \pi i \lambda_{2,3}(p'_1-x'_3)}\int dk_3 \,\e^{2\pi i (x_3k_3-x'_3p_2)}\delta(k_3-p_2')
\langle p_1| 
\e^{2\pi i(\hat{x}_1-\hat{p}_1)(k_3-p_2)}  |p'_1\rangle,
\\
&= \e^{2 \pi i (\lambda_{2,3}(p'_1-x'_3) + x_3p'_2 - x'_3p_2)}
\langle p_1| 
\e^{2\pi i(\hat{x}_1-\hat{p}_1)(p'_2-p_2)}  |p'_1\rangle.
\displaybreak[0]
\end{align*}
In the intermediate step, we have once again utilized the insertion of 
$1 = \int |k_3\rangle \langle k_3|dk_3$.
From 
$\e^{2\pi i \alpha \hat{x}_1}\e^{-2\pi i \alpha \hat{p}_1} 
= \e^{\pi i \alpha^2} \e^{2\pi i\alpha(\hat{x}_1-\hat{p}_1)}$ and (\ref{qm}), the bracket is evaluated as
\begin{align*}
\langle p_1| \e^{2\pi i(\hat{x}_1-\hat{p}_1)(p'_2-p_2)} |p'_1\rangle
&= \e^{-\pi i (p'_2-p_2)^2}\langle p_1 | \e^{2\pi i (p'_2-p_2)\hat{x}_1} 
\e^{-2\pi i (p'_2-p_2)\hat{p}_1}|p'_1\rangle
\\
 &= \e^{-\pi i (p'_2-p_2)^2-2\pi i (p'_2-p_2)p'_1}\langle p_1-p'_2+p_2|p'_1\rangle
 \\
 &=  \e^{-\pi i (p_1-p'_1)(p_1+p'_1)}\delta(p_1+p_2-p'_1-p'_2).
\end{align*}
Thus $\mathscr{R}_0$ is factorized into the two independent integrals as
\begin{equation}
\begin{split}
\mathscr{R}_0 &= \delta(p_1+p_2-p'_1-p'_2) 
\e^{-\pi i (p_1-p'_1)(p_1+p'_1)+2 \pi i \lambda_{2,3}p'_1}
\\
&\times \left(\int dx_3 \e^{2\pi ix_3(p'_2-p_3)}\Phi_\bb(x_3-p_1)\right)
\left(  \int dx'_3 \frac{\e^{2\pi i x'_3(p'_3-p_2-\lambda_{2,3})}}{\Phi_\bb(p'_1-x'_3+\lambda_{1,3})}
\right).
\end{split}
\end{equation}
Evaluating them by (\ref{ram1}),  (\ref{ram2}), substituting the resulting $\mathscr{R}_0$ into (\ref{haji}) and 
applying (\ref{pinv}), we obtain (\ref{pppp}) with (\ref{ppff}).
\end{proof}

The integral kernel 
for $\mathscr{R}$ 
in the ``coordinate" ($\hat{x}$-diagonal) representation can be derived similarly, 
or by taking the Fourier transformation of (\ref{pppp}) as
\begin{equation}
\begin{split}
&\langle x_1, x_2, x_3 | \mathscr{R}_{123} |x'_1, x'_2, x'_3\rangle 
\\
&= \int dp_1dp_2dp_3dp'_1dp'_2dp'_3 
\e^{2\pi i(x_1p_1+x_2p_2+x_3p_3-x'_1p'_1-x'_2p'_2-x'_3p'_3)}
\langle p_1, p_2, p_3 | \mathscr{R} _{123}|p'_1, p'_2, p'_3\rangle.
\end{split}
\end{equation}
The outcome, again attributed to \cite{S10}, is given by 
\begin{align}
&\langle x_1, x_2, x_3 | \mathscr{R}_{123} |x'_1, x'_2, x'_3\rangle 
= \varrho_x \delta(x_2+x_3-x'_2-x'_3) \e^{2\pi i \phi_x}
\frac{\Phi_\bb(x_1-x'_2+i\eta)\Phi_\bb(x'_1-x_2+\lambda_{1,2}+i\eta)}
{\Phi_\bb(x_1-x_2+\lambda_1)\Phi_\bb(x'_1-x'_2-\lambda_2)},
\label{xxxx}
\\
&\phi_x = (x_1-x'_1)x_3+(x_2-x'_1)(\lambda_3+i\eta) + (x'_2-x_2)\lambda_2,
\quad
\varrho_x = \e^{2\pi i(\lambda_1\lambda_2-\lambda_1\lambda_3+\lambda_2\lambda_3+\eta^2)}\varrho_p.
\end{align}

When $q$ is specialized to a root of unity, the integral kernel (\ref{xxxx}) 
essentially reduces to the matrix elements of the finite dimensional R-matrix 
in \cite[eq.(3.4)]{SMS96}.
It is known to be connected to the so-called Zamolodchikov-Bazhanov-Baxter model \cite{BB92}
via the vertex-IRC (Interaction Round a Cube) 
duality transformation \cite[Chap. 3.3]{S10}.

\section{Conclusion}\label{s:con}
We have studied the cluster transformation, denoted as  $\hat{R}=\hat{R}_{123}$, in 
the quantum cluster algebra associated with the square quiver as depicted in Figures 
\ref{fig:mus} and \ref{fig:R123}. 
By employing the realization of the quantum Y-variables using $q$-Weyl algebras as in (\ref{Yw})--(\ref{Ypw}), 
we have extracted an operator $R=R_{123}$ such that $\hat{R}=\mathrm{Ad}(R)$.
It satisfies the tetrahedron equation involving the spectral parameters as 
established in Theorem \ref{th:main}.
We have discovered that the formula (\ref{RL1}) for the $R$ precisely reproduces the known $R$-matrix in \cite{S99},   
the origins of which had remained shrouded in mystery for many years.
Additionally, we have presented a detailed derivation of the matrix elements of $R$.
The result of our findings offers significant insight into understanding the connection between 3D integrability and the theory of quantum cluster algebras.

The tetrahedron equation was introduced as a natural 3D version of the
Yang-Baxter equation as explained in Introduction. Geometrically it should
be interpreted as a 3D version of braids and knots, 2-braids and surface
knots. It will be interesting to find applications of the result in the
present paper to the geometry of such objects  (cf. \cite{CS96}), and to quantum systems  (cf. \cite{KFKS15,PGAG22}).


\appendix

\section{Formulas of $\widehat{R}_{123}$ for general sign 
$(\varepsilon_1,  \varepsilon_2,  \varepsilon_3,  \varepsilon_4)$}\label{app:RY}

The formulas for the cluster transformation $\widehat{R}$ in (\ref{rhat}), specifically involving the monomial part 
$\tau_{\varepsilon_1,  \varepsilon_2,  \varepsilon_3,  \varepsilon_4}$ on the right are given as follows:
\begin{align*}
\widehat{R}&= \mathrm{Ad}\left(\Psi_q(Y_8)\Psi_q(Y_5)\Psi_q(q^{-1}Y_4Y_8)\Psi_q(Y_4)\right)\tau_{++++}
\\
&= \mathrm{Ad}\left(\Psi_q(Y_8)\Psi_q(Y_5)\Psi_q(q^{-1}Y_4Y_8)\Psi_q(Y^{-1}_4)^{-1}\right)\tau_{+++-} 
\\
&=\mathrm{Ad}\left(\Psi_q(Y_8)\Psi_q(Y_5)\Psi_q(q^{-1}Y^{-1}_4Y^{-1}_8)^{-1}\Psi_q(Y^{-1}_8)\right)\tau_{++-+}
\displaybreak[0] 
\\
&=\mathrm{Ad}\left(\Psi_q(Y_8)\Psi_q(Y_5)\Psi_q(q^{-1}Y^{-1}_4Y^{-1}_8)^{-1}\Psi_q(Y_8)^{-1}\right)\tau_{++--} 
\\
&=\mathrm{Ad}\left(\Psi_q(Y_8)\Psi_q(Y^{-1}_5)^{-1}\Psi_q(q^{-1}Y_4Y_8)\Psi_q(qY_4Y_5)\right)\tau_{+-++} 
\\
&=\mathrm{Ad}\left(\Psi_q(Y_8)\Psi_q(Y^{-1}_5)^{-1}\Psi_q(q^{-1}Y_4Y_8)\Psi_q(qY^{-1}_4Y^{-1}_5)^{-1}\right)\tau_{+-+-} 
\displaybreak[0]
\\
&=\mathrm{Ad}\left(\Psi_q(Y_8)\Psi_q(Y^{-1}_5)^{-1}\Psi_q(q^{-1}Y^{-1}_4Y^{-1}_8)^{-1}
\Psi_q(q^{-1}Y_5Y^{-1}_8)\right)\tau_{+--+} 
\\
&=\mathrm{Ad}\left(\Psi_q(Y_8)\Psi_q(Y^{-1}_5)^{-1}\Psi_q(q^{-1}Y^{-1}_4Y^{-1}_8)^{-1}
\Psi_q(q^{-1}Y^{-1}_5Y_8)^{-1}\right)\tau_{+---} 
\displaybreak[0]
\\
&=\mathrm{Ad}\left(\Psi_q(Y^{-1}_8)^{-1}\Psi_q(qY_5Y_8)\Psi_q(Y_4)\Psi_q(qY_4Y^{-1}_8)\right)\tau_{-+++} 
\\
&=\mathrm{Ad}\left(\Psi_q(Y^{-1}_8)^{-1}\Psi_q(qY_5Y_8)\Psi_q(Y_4)\Psi_q(qY^{-1}_4Y_8)^{-1}\right)\tau_{-++-} 
\\
&=\mathrm{Ad}\left(\Psi_q(Y^{-1}_8)^{-1}\Psi_q(qY_5Y_8)\Psi_q(Y^{-1}_4)^{-1}
\Psi_q(Y^{-1}_8)\right)\tau_{-+-+} 
\displaybreak[0]
\\
&=\mathrm{Ad}\left(\Psi_q(Y^{-1}_8)^{-1}\Psi_q(qY_5Y_8)\Psi_q(Y^{-1}_4)^{-1}
\Psi_q(Y_8)^{-1}\right)\tau_{-+--} 
\\
&=\mathrm{Ad}\left(\Psi_q(Y^{-1}_8)^{-1}\Psi_q(qY^{-1}_5Y^{-1}_8)\Psi_q(Y_4)
\Psi_q(qY_4Y_5)\right)\tau_{--++} 
\\
&=\mathrm{Ad}\left(\Psi_q(Y^{-1}_8)^{-1}\Psi_q(qY^{-1}_5Y^{-1}_8)\Psi_q(Y_4)
\Psi_q(qY^{-1}_4Y^{-1}_5)^{-1}\right)\tau_{--+-} 
\displaybreak[0]
\\
&=\mathrm{Ad}\left(\Psi_q(Y^{-1}_8)^{-1}\Psi_q(qY^{-1}_5Y^{-1}_8)\Psi_q(Y^{-1}_4)^{-1}
\Psi_q(Y_5)\right)\tau_{---+} 
\\
&=\mathrm{Ad}\left(\Psi_q(Y^{-1}_8)^{-1}\Psi_q(qY^{-1}_5Y^{-1}_8)\Psi_q(Y^{-1}_4)^{-1}
\Psi_q(Y^{-1}_5)^{-1}\right)\tau_{----}.
\displaybreak[0]
\end{align*}
The formulas containing the monomial part 
$\tau_{\varepsilon_1,  \varepsilon_2,  \varepsilon_3,  \varepsilon_4}$ in the center are given as follows:
\begin{align*}
\widehat{R}
&=\mathrm{Ad}\left(\Psi_q(Y_8)\Psi_q(Y_5)\right)
\tau_{++++} \mathrm{Ad}\left(\Psi_q(qY'^{-1}_5Y'^{-1}_8)\Psi_q(Y'^{-1}_8)\right)
\\
&=\mathrm{Ad}\left(\Psi_q(Y_8)\Psi_q(Y_5)\right)
\tau_{+++-} \mathrm{Ad}\left(\Psi_q(Y'^{-1}_5)\Psi_q(Y'_8)^{-1}\right)
\\
&=\mathrm{Ad}\left(\Psi_q(Y_8)\Psi_q(Y_5)\right)
\tau_{++-+} \mathrm{Ad}\left(\Psi_q(qY'_5Y'_8)^{-1}\Psi_q(Y'^{-1}_8)\right)
\\
&=\mathrm{Ad}\left(\Psi_q(Y_8)\Psi_q(Y_5)\right)
\tau_{++--} \mathrm{Ad}\left(\Psi_q(Y'_5)^{-1}\Psi_q(Y'_8)^{-1}\right)
\\
&=\mathrm{Ad}\left(\Psi_q(Y_8)\Psi_q(Y^{-1}_5)^{-1}\right)
\tau_{+-++} \mathrm{Ad}\left(\Psi_q(qY'^{-1}_5Y'^{-1}_8)\Psi_q(Y'^{-1}_8)\right)
\\
&=\mathrm{Ad}\left(\Psi_q(Y_8)\Psi_q(Y^{-1}_5)^{-1}\right)
\tau_{+-+-} \mathrm{Ad}\left(\Psi_q(Y'^{-1}_5)\Psi_q(Y'_8)^{-1}\right)
\\
&=\mathrm{Ad}\left(\Psi_q(Y_8)\Psi_q(Y^{-1}_5)^{-1}\right)
\tau_{+--+} \mathrm{Ad}\left(\Psi_q(qY'_5Y'_8)^{-1}\Psi_q(Y'^{-1}_8)\right)
\\
&=\mathrm{Ad}\left(\Psi_q(Y_8)\Psi_q(Y^{-1}_5)^{-1}\right)
\tau_{+---} \mathrm{Ad}\left(\Psi_q(Y'_5)^{-1}\Psi_q(Y'_8)^{-1}\right)
\\
&=\mathrm{Ad}\left(\Psi_q(Y^{-1}_8)^{-1}\Psi_q(qY_5Y_8)\right)
\tau_{-+++} \mathrm{Ad}\left(\Psi_q(qY'^{-1}_5Y'^{-1}_8)\Psi_q(Y'^{-1}_8)\right)
\\
&=\mathrm{Ad}\left(\Psi_q(Y^{-1}_8)^{-1}\Psi_q(qY_5Y_8)\right)
\tau_{-++-} \mathrm{Ad}\left(\Psi_q(Y'^{-1}_5)\Psi_q(Y'_8)^{-1}\right)
\\
&=\mathrm{Ad}\left(\Psi_q(Y^{-1}_8)^{-1}\Psi_q(qY_5Y_8)\right)
\tau_{-+-+} \mathrm{Ad}\left(\Psi_q(qY'_5Y'_8)^{-1}\Psi_q(Y'^{-1}_8)\right)
\\
&=\mathrm{Ad}\left(\Psi_q(Y^{-1}_8)^{-1}\Psi_q(qY_5Y_8)\right)
\tau_{-+--} \mathrm{Ad}\left(\Psi_q(Y'_5)^{-1}\Psi_q(Y'_8)^{-1}\right)
\\
&=\mathrm{Ad}\left(\Psi_q(Y^{-1}_8)^{-1}\Psi_q(qY^{-1}_5Y^{-1}_8)^{-1}\right)
\tau_{--++} \mathrm{Ad}\left(\Psi_q(qY'^{-1}_5Y'^{-1}_8)\Psi_q(Y'^{-1}_8)\right)
\\
&=\mathrm{Ad}\left(\Psi_q(Y^{-1}_8)^{-1}\Psi_q(qY^{-1}_5Y^{-1}_8)^{-1}\right)
\tau_{--+-} \mathrm{Ad}\left(\Psi_q(Y'^{-1}_5)\Psi_q(Y'_8)^{-1}\right)
\\
&=\mathrm{Ad}\left(\Psi_q(Y^{-1}_8)^{-1}\Psi_q(qY^{-1}_5Y^{-1}_8)^{-1}\right)
\tau_{---+} \mathrm{Ad}\left(\Psi_q(qY'_5Y'_8)^{-1}\Psi_q(Y'^{-1}_8)\right)
\\
&=\mathrm{Ad}\left(\Psi_q(Y^{-1}_8)^{-1}\Psi_q(qY^{-1}_5Y^{-1}_8)^{-1}\right)
\tau_{----} \mathrm{Ad}\left(\Psi_q(Y'_5)^{-1}\Psi_q(Y'_8)^{-1}\right).
\displaybreak[0]
\end{align*}

\section{Supplement to Section \ref{ss:ms}}\label{ap:sup}

The two sides of the inhomogeneous tetrahedron equation (\ref{ihte}) yield the 
following monomial transformation:
\begin{equation}\begin{split}\label{y21}
Y^{(21)}_1 &\mapsto Y_1,  \; 
 \;  Y^{(21)}_5  \mapsto q^{-2}Y^{-1}_9Y^{-1}_{10}Y^{-1}_{15},  
 \;  Y^{(21)}_9  \mapsto Y^{-1}_4,  
 \; Y^{(21)}_{13}  \mapsto Y_{13},
\\
Y^{(21)}_2 &\mapsto Y_2Y_4Y_5Y_8Y_9, 
 \; Y^{(21)}_6 \mapsto Y_6, 
 \; Y^{(21)}_{10} \mapsto qY^{-1}_8 Y^{-1}_9, 
 \; Y^{(21)}_{14} \mapsto q^{-1}Y_9Y_{10},
\\
Y^{(21)}_3 &\mapsto  q^{-4}Y_3Y_4Y_9Y_{10}Y_{15}, 
 \; Y^{(21)}_7 \mapsto q^{-2}Y_7Y_8Y_{14}, 
 \; Y^{(21)}_{11} \mapsto Y^{-1}_{14}, 
 \; Y^{(21)}_{15}  \mapsto Y^{-1}_{10},
\\
Y^{(21)}_{4}  &\mapsto qY_{10}Y_{11}, 
 \; Y^{(21)}_{8}  \mapsto qY^{-1}_5Y^{-1}_9Y^{-1}_{10}Y^{-1}_{11}, 
 \; Y^{(21)}_{12}  \mapsto Y_{12},   
 \; Y^{(21)}_{16}  \mapsto  q^4Y_5Y_9Y_{10}Y_{11}Y_{14}Y_{15}Y_{16}.
\end{split}
\displaybreak[0]
\end{equation}

Let us describe the monomial parts $\tau_{+-++}$ and  $\tau_{--+-}$ in (\ref{taue}) mentioned
in Proposition \ref{pr:mono}. We also give their inverse.

$\tau_{+-++}$ and $\tau_{+-++}^{-1}$ are given as follows: 
\begin{align}
&\tau_{+-++}: \left\{ ~~
\begin{alignedat}{3}
Y'_1 & \mapsto Y_1,\quad &
Y'_4 & \mapsto Y^{-1}_5, \quad &
Y'_7 & \mapsto Y_7,
\\
Y'_2 & \mapsto Y_2Y_4Y_5,\quad & 
Y'_5 & \mapsto q^{-1}Y_5Y^{-1}_8,\quad  &
Y'_8 &  \mapsto qY^{-1}_4Y^{-1}_5,
\\
Y'_3 &\mapsto Y_4 Y_3Y_8,\quad &
Y'_6 & \mapsto q^{-1}Y_5Y_6,\quad &
Y'_9 &\mapsto qY_8Y_9,
\end{alignedat}
\right.
\displaybreak[0]
\\
&\tau_{+-++}^{-1}: \left\{ ~~
\begin{alignedat}{3}
Y_1 & \mapsto Y'_1, \quad &
Y_4 & \mapsto q Y'_4Y'^{-1}_8, \quad &
Y_7 & \mapsto Y'_7,
\\
Y_2 & \mapsto q Y'_2Y'_8, \quad & 
Y_5 & \mapsto Y'^{-1}_4, \quad & 
Y_8 & \mapsto qY'^{-1}_4Y'^{-1}_5,
\\
Y_3 & \mapsto Y'_5Y'_3Y'_8, \quad &
Y_6 & \mapsto qY'_4Y'_6, \quad &
Y_9 & \mapsto Y'_4Y'_5Y'_9.
\end{alignedat}
\right.
\end{align}
Note that the image is not necessarily sign coherent\footnote{The sign coherence holds 
under the tropical exchange relation as stated in Theorem \ref{thm:sign-coherence}.
It corresponds to $\tau_{k,\varepsilon}$ with $\varepsilon= (\text{tropical sign})$,
which is not necessarily valid in the present setting of prescribing the signs.}.

$\tau_{--+-}$ and $\tau_{--+-}^{-1}$ are given as follows: 
\begin{align}
&\tau_{--+-}: \left\{ ~~
\begin{alignedat}{3}
Y'_1 & \mapsto Y_4Y_1Y_5,\quad &
Y'_4 & \mapsto qY_4Y^{-1}_8, \quad &
Y'_7 & \mapsto qY_7Y_8,
\\
Y'_2 & \mapsto Y_2,\quad & 
Y'_5 & \mapsto  Y^{-1}_4,\quad  &
Y'_8 &  \mapsto qY^{-1}_4Y^{-1}_5,
\\
Y'_3 &\mapsto q^{-1}Y_3 Y_4,\quad &
Y'_6 & \mapsto Y_5Y_6Y_8,\quad &
Y'_9 &\mapsto Y_9,
\end{alignedat}
\right.
\\
&\tau_{--+-}^{-1}: \left\{ ~~
\begin{alignedat}{3}
Y_1 & \mapsto q^{-1}Y'_1Y'_8, \quad &
Y_4 & \mapsto  Y'^{-1}_5, \quad &
Y_7 & \mapsto Y'_5Y'_4Y'_7,
\\
Y_2 & \mapsto Y'_2, \quad & 
Y_5 & \mapsto q^{-1}Y'_5Y'^{-1}_8, \quad & 
Y_8 & \mapsto qY'^{-1}_4Y'^{-1}_5,
\\
Y_3 & \mapsto qY'_3Y'_5, \quad &
Y_6 & \mapsto Y'_4Y'_6Y'_8, \quad &
Y_9 & \mapsto Y'_9.
\end{alignedat}
\right.
\end{align}

\section{$\tau^{uw}_{\ve_1,\ve_2,\ve_3,\ve_4}$ 
in (\ref{Tcom}) and $R_{\ve_1,\ve_2,\ve_3,\ve_4}$ in (\ref{Ree}) for $(\ve_1,\ve_2,\ve_3,\ve_4)$ in 
 Table \ref{tab:adt}}\label{app:ruw}

We include the case $\tau^{uw}_{-,-,+,+}$ previously addressed in (\ref{uwL1}) to facilitate comparison.
\begin{align}
\tau^{uw}_{----}:&\; 
\begin{cases}
u_1 \mapsto u_2-\lambda_1+\lambda_2, 
\quad  \qquad \qquad \qquad \quad
w_1  \mapsto w_1+ w_2-w_3,
\\
u_2  \mapsto u_2+u_3-w_1-\lambda_1, 
\qquad \qquad \quad\; \;
w_2  \mapsto w_3,
\\
u_3  \mapsto w_1+\lambda_1, 
\qquad \qquad \quad  \qquad \qquad\quad \,
w_3  \mapsto -u_1+u_2+w_3-\lambda_1+\lambda_2,
\end{cases}
\\
\tau^{uw}_{--++}:&\; 
\begin{cases}
u_1 \mapsto u_1+w_2-w_3+\lambda_2-\lambda_3, 
\quad \qquad\;
w_1  \mapsto w_1+ w_2-w_3,
\\
u_2  \mapsto u_1+u_3-w_1, 
\qquad \qquad \qquad  \qquad
w_2  \mapsto w_3,
\\
u_3  \mapsto -u_1+u_2+w_1, 
\qquad \qquad \quad  \qquad\;
w_3  \mapsto w_2+\lambda_2-\lambda_3,
\end{cases}
\\
\tau^{uw}_{-+-+}: &\; 
\begin{cases}
u_1 \mapsto u_2+w_2-w_3-\lambda_1+\lambda_2-\lambda_3, 
\quad  \;
w_1  \mapsto u_3+\lambda_3,
\\
u_2  \mapsto u_2+u_3-w_1-\lambda_1, 
\qquad \qquad \qquad 
w_2  \mapsto -u_3+w_1+w_2-\lambda_3,
\\
u_3  \mapsto w_1+\lambda_1, 
\qquad \qquad \quad  \qquad \qquad \quad\; \,
w_3  \mapsto -u_1+u_2+w_2-\lambda_1+\lambda_2-\lambda_3,
\end{cases}
\\
\tau^{uw}_{+-+-}: &\; 
\begin{cases}
u_1 \mapsto u_2-\lambda_1+\lambda_2, 
\quad  \qquad \qquad \qquad \quad\; 
w_1  \mapsto u_3+w_2-w_3,
\\
u_2  \mapsto u_1, 
\qquad \qquad \qquad  \qquad \qquad \qquad\;\;\;
w_2  \mapsto -u_3+w_1+w_3,
\\
u_3  \mapsto -u_1+u_2+u_3, 
\qquad \qquad \quad  \qquad\;\;
w_3  \mapsto -u_1+u_2+w_3-\lambda_1+\lambda_2,
\end{cases}
\\
\tau^{uw}_{++--}: &\; 
\begin{cases}
u_1 \mapsto u_2+w_2-w_3-\lambda_1+\lambda_2-\lambda_3, 
\quad \;
w_1  \mapsto u_3+\lambda_3,
\\
u_2  \mapsto u_2+u_3-w_1-\lambda_1, 
\qquad \qquad \qquad 
w_2  \mapsto  -u_3+w_1+w_2-\lambda_3,
\\
u_3  \mapsto w_1+\lambda_1, 
\qquad \qquad \quad  \qquad \qquad \quad \; \; 
w_3  \mapsto -u_1+u_2+w_2-\lambda_1+\lambda_2-\lambda_3,
\end{cases}
\\
\tau^{uw}_{++++}: &\; 
\begin{cases}
u_1 \mapsto u_1+w_2-w_3+\lambda_2-\lambda_3, 
\quad  \qquad\;\;
w_1  \mapsto u_3+\lambda_3,
\\
u_2  \mapsto u_1, 
\qquad \qquad \qquad  \qquad \qquad \qquad \;\;\;
w_2  \mapsto  -u_3+w_1+w_2-\lambda_3,
\\
u_3  \mapsto -u_1+u_2+u_3, 
\qquad \qquad \quad  \qquad\;\;
w_3  \mapsto  w_2+\lambda_2-\lambda_3.
\end{cases}
\displaybreak[0]
\end{align}

Now we proceed to $R_{\ve_1,\ve_2,\ve_3,\ve_4}$.
We include the case $(-,-,+,+)$ in (\ref{Rf1}) for comparison.
\begin{equation}\label{cdk}
\begin{split}
(-,-,-,-):\; 
&\Psi_q(\e^{u_3-w_1})^{-1}\Psi_q(\e^{u_3-w_1-w_2+w_3+\lambda_3})^{-1}P_{----}
\Psi_q(\e^{w_3-w_2-\lambda_2})^{-1}\Psi_q(\e^{u_3-w_1-\lambda_1+\lambda_3})^{-1},
\\
(-,-,+,+):\; 
& \Psi_q(\e^{w_3-w_2+\lambda_3})^{-1}
\Psi_q(\e^{u_3-w_1})^{-1}
P_{--++} \Psi_q(\e^{w_1-u_3+\lambda_1-\lambda_3})
\Psi_q(\e^{w_2-w_3+\lambda_2}),
\\
(-,+,-,+):\; 
&\Psi_q(\e^{u_3-w_1})^{-1}\Psi_q(\e^{w_1+w_2-w_3-u_3-\lambda_3})P_{-+-+}
\Psi_q(\e^{u_3-w_1-w_2+w_3-\lambda_1-\lambda_2+\lambda_3})^{-1}
\Psi_q(\e^{w_1-u_3+\lambda_1-\lambda_3}),
\\
(+,-,+,-):\;
& \Psi_q(\e^{w_1-u_3})\Psi_q(\e^{w_3-w_2+\lambda_3})^{-1}P_{+-+-}
\Psi_q(\e^{w_2-w_3+\lambda_2})
\Psi_q(\e^{u_3-w_1-\lambda_1+\lambda_3})^{-1},
\\
(+,+,-,-):\;
& \Psi_q(\e^{w_1-u_3})\Psi_q(\e^{w_2-w_3-\lambda_3})P_{++--}
\Psi_q(\e^{w_3-w_2-\lambda_2})^{-1}
\Psi_q(\e^{u_3-w_1-\lambda_1+\lambda_3})^{-1},
\\
(+,+,+,+):\;
& \Psi_q(\e^{w_1-u_3})\Psi_q(\e^{w_2-w_3-\lambda_3})P_{++++}
\Psi_q(\e^{w_1+w_2-w_3-u_3+\lambda_1+\lambda_2-\lambda_3})
\Psi_q(\e^{w_1-u_3+\lambda_1-\lambda_3}).
\end{split}
\end{equation}

\section{Integral formula involving non-compact quantum dilogarithm}\label{ap:nc}

The following is known as a modular double analogue of the Ramanujan ${}_1\Psi_1$-sum 
\begin{align}
\int dt \frac{\Phi_\bb(t+u)}{\Phi_\bb(t+v)} \e^{2\pi i wx}
&= \frac{\Phi_\bb(u-v-i\eta)\Phi_\bb(w+i\eta)}{K\Phi_\bb(u-v+w-i\eta)}
\e^{-2\pi i w(v+i\eta)}
\label{rama1}\\
&= \frac{K\Phi_\bb(v-u-w+i\eta)}{\Phi_\bb(v-u+i\eta)\Phi_\bb(-w-i\eta)}
\e^{-2\pi i w(u-i\eta)},
\label{rama2}
\end{align}
where $K = \e^{-i\pi(4\eta^2+1)/12}$.
See \cite[eq.(B.73)]{S10} for the condition concerning the validity of the integrals.
From $\Phi_\bb(u)\vert_{u \rightarrow -\infty} \rightarrow 1$,
their limit $u,v \rightarrow -\infty$ reduces to
\begin{align}
\int dt \frac{\e^{2\pi i w t}}{\Phi_\bb(t+v)} &= \frac{\Phi_\bb(w+i\eta)}{K}\e^{-2\pi iw(v+i\eta)},
\label{ram1}
\\
\int dt \Phi_\bb(t+u)\e^{2\pi i w t}
&= \frac{K}{\Phi_\bb(-w-i\eta)} \e^{-2\pi i w(u-i\eta)}.
\label{ram2}
\end{align}

\bigskip

\section*{Acknowledgments}
The authors thank Sergey Sergeev for useful communication, 
Junya Yagi and Akihito Yoneyama for stimulating discussions.
RI  is supported by  JSPS KAKENHI Grant Number 19K03440 and 23K03048.
YT is supported by  JSPS KAKENHI Grant Number JP21K03240 and 22H01117.
AK thanks the organizers of 
{\em Representation Theory, Integrable Systems, and Related Topics} 
at the Beijing Institute of Mathematical Sciences and Applications (July 31 - August 4, 2023) 
for their generous invitation, and Andrey Marshakov, Eric Ragoucy, Nicolai Reshetikhin, 
Bart Vlaar, and Da-jun Zhang for their kind interest.
He also thanks Murray Batchelor, Rodney Baxter, Vladimir Bazhanov, 
Vladimir Mangazeev, and Sergey Sergeev for their warm hospitality 
during his stay at Australian National University (March 21 - April 7, 2023), 
where a part of this work was conducted. 
\vspace{0.5cm}


\end{document}